\documentclass[a4paper,twoside]{article}
\usepackage{amsmath,amsfonts,amsthm}
\usepackage[nottoc]{tocbibind}
\usepackage{mathrsfs}
\usepackage{enumitem}
\usepackage{stmaryrd}
\usepackage{amssymb}
\usepackage{tikz-cd}
\usepackage{indentfirst}
\usepackage{microtype}
\usepackage{lmodern}
\usepackage[a4paper,top=3cm, bottom=2.5cm, left=2cm, right=2cm]{geometry}

\usepackage{color}
\definecolor{dkgreen}{rgb}{0,0.6,0}
\definecolor{mblue}{rgb}{0,0.5,1}
\definecolor{gray}{rgb}{0.5,0.5,0.5}
\definecolor{mauve}{rgb}{0.58,0,0.82}

\usepackage{hyperref}
\usepackage{cleveref}
\hypersetup{
  linktoc=page,
	colorlinks = true,
	linkcolor = mblue, 
	filecolor = cyan,
	urlcolor = mauve, 
	citecolor = dkgreen, 
}

\theoremstyle{plain}
\newtheorem{thm}{Theorem}[section] 

\newtheorem*{thm*}{Theorem}
\newtheorem{cor}[thm]{Corollary}
\newtheorem{lem}[thm]{Lemma}  
\newtheorem{prop}[thm]{Proposition}
 
\theoremstyle{definition}
\newtheorem{defn}[thm]{Definition}

\newtheorem{rek}[thm]{Remark}
\newtheorem{exe}[thm]{Example}

\makeatletter
    \let\c@equation\c@thm
\makeatother
\numberwithin{equation}{section}

\newcommand{\q}[1]{\left( #1 \right)}

\newcommand{\scr}[1]{\mathscr{#1}}
\newcommand{\bb}[1]{\mathbb{#1}}
\newcommand{\rr}[1]{\mathrm{#1}}
\newcommand{\cc}[1]{\mathcal{#1}}
\newcommand{\Hdr}[1]{\mathrm{H}^{#1}_{\mathrm{dR}}}

\newcommand{\inv}{^{-1}}
\newcommand{\cros}{^{\times}}

\let\emptyset\varnothing
\let\bar\overline
\let\tilde\widetilde
\let\hat\widehat

\newcommand{\dR}{\mathrm{dR}}
\newcommand{\frob}{\mathrm{Frob}}

\newcommand{\Hom}{\mathrm{Hom}}

\newcommand{\im}{\mathrm{im}}
\newcommand{\id}{\mathrm{id}}

\newcommand{\Kl}{\mathrm{Kl}}

\newcommand{\pr}{\mathrm{pr}}
\newcommand{\rk}{\mathrm{rk}}

\newcommand{\Sym}{\mathrm{Sym}}
\newcommand{\SL}{\mathrm{SL}}

\newcommand{\tor}{\mathrm{tor}}
\newcommand{\tr}{\mathrm{Tr}}

\newcommand{\ord}{\mathrm{ord}}

\newcommand{\bql}{\bar{\mathbb{Q}}_\ell}

\newcommand{\fpp}{\mathbb{F}_p}
\newcommand{\fqq}{\mathbb{F}_q}

\newcommand{\gr}{\mathrm{gr}}
\newcommand{\coker}{\mathrm{coker}\,}

\begin{document}

\title{Hodge numbers of motives attached to Kloosterman and Airy moments}

\author{Yichen Qin}
\date{}
\maketitle
\begin{abstract}
    Fres\'an, Sabbah, and Yu constructed motives  $\mathrm{M}_{n+1}^k(\mathrm{Kl})$ over $\mathbb{Q}$ encoding symmetric power moments of Kloosterman sums in $n$ variables. When $n=1$, they use the irregular Hodge filtration on the exponential mixed Hodge structure associated with $\mathrm{M}_{2}^k(\mathrm{Kl})$ to compute the Hodge numbers of $\mathrm{M}_{2}^k(\mathrm{Kl})$, which turn out to be either $0$ or $1$. In this article, I explain how to compute the (irregular) Hodge numbers of $\mathrm{M}_{n+1}^k(\mathrm{Kl})$ for $n=2$ or for general values of $n$ such that $\gcd(k,n+1)=1$. I will also discuss related motives attached to Airy moments constructed by Sabbah and Yu. In particular, the computation shows that there are Hodge numbers bigger than $1$ in most cases. 
\end{abstract}

\tableofcontents 

\section{Introduction}\label{sec:intro}

\subsection{Background}
Kloosterman sums in $n$ variables are the exponential sums over finite fields defined for each prime power $q=p^r$ and each $a\in \fqq\cros$ by
\begin{equation*}
	\begin{split}
		\mathrm{Kl}_{n+1}(a;q):=\sum_{x_1,\ldots,x_n\in \fqq^{\times}}\exp\left(2\pi i/p\cdot \tr_{\fqq/\fpp}\left(x_1+\cdots+x_n+\frac{a}{x_1\cdots x_n}\right)\right),
	\end{split}
\end{equation*}
where $\tr_{\fqq/\fpp}(x)$ is the trace from $\fqq$ to $\fpp$. They are the finite field analogs of Bessel functions 
	\begin{equation*}
		\begin{split}
			\rr{Be}_{n+1}(z):=\oint_{(S^1)^n}\exp\left(x_1+\cdots+x_n+\frac{z}{x_1\cdots x_n}\right)\frac{\mathrm{d}x_1}{x_1}\cdots\frac{\mathrm{d}x_n}{x_n},
		\end{split}
	\end{equation*}
which satisfy the Bessel differential equations $(z\partial_z)^{n+1}-z=0$.

We fix a prime $\ell\neq p$ and an embedding $\iota\colon \bql\to \bb{C}$. After Deligne \cite[Sommes trig.]{SGA41/2}, there exists a lisse $\ell$-adic local system $\Kl_{n+1}$ of rank $n+1$ over $\bb{G}_{m,\fqq}$, called the \textit{Kloosterman sheaf}, such that for each $a\in \fqq\cros$, 
\begin{equation*}
	\begin{split}
		\iota(\tr(\frob_q,(\Kl_{n+1})_{\bar a}))=(-1)^n\Kl_{n+1}(a;q).
	\end{split}
\end{equation*}
In other words, we realize the Kloosterman sums as traces of Frobenius acting on $\Kl_{n+1}$.

For each $k\geq 1$, the \textit{$k$-th symmetric power moments of Kloosterman sums} are the integers
\begin{equation}\label{eq:moments}
	\begin{split}
		m_{n+1}^k(q):=-\sum_{a\in \fqq\cros}\tr(\frob_q,(\Sym^k\Kl_{n+1})_{\bar a}),
	\end{split}
\end{equation}
where $\Sym^k\Kl_{n+1}$ is the $k$-th symmetric power of the Kloosterman sheaf. To study these moments as $q$ varies, we build an $L$-function $L(k,n+1;s)$ by taking the Euler product in which the local factor at the prime $p$ comes from a generating series of the sequence $\{m_{n+1}^k(p^n)\}_n$. Fu and Wan studied these local factors in detail \cite{fu2005functions}. 

A priori, the $L$-functions are only defined on some half-plane where the real part of $s$ is large enough. In \cite[Thm.\,1.2]{fresan2018hodge}, Fres\'an, Sabbah, and Yu proved that when $n=1$, the $L$-functions $L(k,2;s)$ extend meromorphically to the complex plane and satisfy some functional equations conjectured by Broadhurst and Roberts in \cite{Broadhurst2016a, Broadhurst2017}.

By the construction of $L(k,n+1;s)$, the local factors in the Euler product come from the characteristic polynomials of Frobenius acting on the middle $\ell$-adic cohomologies of $\Sym^k\Kl_{n+1}$ (i.e., the images of the compact support cohomologies in the usual cohomologies). Inspired by the analogy between Kloosterman sums and Bessel functions, Fres\'an, Sabbah, and Yu considered the \textit{Kloosterman connection}, also denoted by $\Kl_{n+1}$, which is the rank $n+1$ connection on $\bb{G}_{m,\bb{C}}$ corresponding to the Bessel differential equation $(z\partial_z)^{n+1}-z=0$. They interpret the middle de Rham cohomology of the connection $\Sym^k\Kl_{n+1}$ as the de Rham realization of a Nori motive $\mathrm{M}_{n+1}^k(\Kl)$ over $\bb{Q}$, which admits a geometric description as a subquotient of $\mathrm{H}^{nk-1}_{\rr{c}}(\mathcal{K})(-1)$. Here $\mathcal{K}$ is the hypersurface defined by the Laurent polynomial 
	\begin{equation*}
		\begin{split}
			g_k=\sum_{i=1}^k\biggl(\sum_{j=1}^n y_{i,j}+ \frac{1}{\prod_{j=1}^n y_{i,j}}\biggr)
		\end{split}
	\end{equation*}
in the torus $\bb{G}_m^{nk}$, see \cite[(3.1)]{fresan2018hodge} and Remark~\ref{rek:mot-classical}.

The $\ell$-adic realizations of $\mathrm{M}_{n+1}^k(\Kl)$ give rise to a family of $\ell$-adic Galois representations, whose $L$-functions coincide with $L(k,n+1;s)$. To relate $L(k,n+1;s)$ with $L$-functions of analytic objects, one relies on a potential automorphy theorem of Patrikis--Taylor \cite{patrikis_taylor_2015}. To apply this theorem, one must check the crucial assumption that the Hodge numbers of the de Rham realization of $\mathrm{M}_{n+1}^k(\Kl)$ are either $0$ or $1$. When $n=1$, this condition is verified by \cite[Thm.\,1.8]{fresan2018hodge}.

\subsection{Main theorem}
Let $\rr{M}_{n+1}^k(\Kl)_{\rr{dR}}$ be the de Rham realizations of $\rr{M}_{n+1}^k(\Kl)$, which underlie pure Hodge structures of weight $nk+1$. In this article, we compute the Hodge numbers of $\rr{M}_{n+1}^k(\Kl)_{\rr{dR}}$ for $n>1$. 
\begin{thm}\label{thm:hodge number-Kl}
  
	\begin{enumerate}[label=\arabic*., ref=\ref{thm:hodge number-Kl}.\theenumi]
		\item \label{thm:hodge number-Kl-1}Assume that $\gcd(k,n+1)=1$. The Hodge numbers $h^{p,nk+1-p}$ of $\rr{M}_{n+1}^k(\Kl)_{\rr{dR}}$ are the coefficients of $t^px^k$ in the formal power series expansion of the rational function
			\begin{equation*}
				\begin{split}
					\frac{(1-t)t^{n+1}}{(1-t^{n+1})(1-x)(1-t x)\cdots(1-t^n x)}
				\end{split}
			\end{equation*}
		if $p\leq \frac{nk+1}{2}$, and we have $h^{p,nk+1-p}=h^{nk+1-p,p}$ if $p>\frac{nk+1}{2}$.
		\item \label{thm:hodge number-Kl-2}Assume that $n=2$ and $3\mid k$. The Hodge numbers $h^{p,2k+1-p}$ of $\rr{M}_{3}^k(\Kl)_{\rr{dR}}$ are given by
			\begin{equation*}
				\begin{split}
					h^{p,2k+1-p}=
			\begin{cases}
				\lfloor\frac{p}{6}\rfloor -\delta_{p,k}& p\equiv 0,1,2,4\ \mathrm{mod} \ 6,\\
				\lfloor\frac{p}{6}\rfloor+1 & p\equiv 3,5\ \mathrm{mod} \ 6,
			\end{cases}
				\end{split}
			\end{equation*}	 
		if $p\leq k$, and we have $h^{p,2k+1-p}=h^{2k+1-p,p}$ if $p>k$. Here $\delta_{p,k}$ is the Kronecker symbol. 
	\end{enumerate}
\end{thm}
\begin{rek}\label{rek:hodge-number}
    The above theorem is a crucial ingredient in proving a generalization of the work of Fres\'an--Sabbah--Yu in the author's forthcoming paper~\cite{Qin23function} to those $\mathrm{M}^k_{n+1}(\Kl)$ having Hodge numbers either $0$ or $1$, such as when $n=2$ and $k\leq 9$. However, we find that in most cases, $\mathrm{M}^k_{n+1}(\Kl)$ can have Hodge numbers bigger than $1$. For example, the Hodge numbers of $\mathrm{M}_3^{10}(\Kl)_{\rr{dR}}$ are given by 
			\begin{equation*}
				\begin{split}
					(h^{p,21-p})_{0\leq p\leq 21}=(0,0,0,1,0,1,1,1,1,2,1,1,2,1,1,1,1,0,1,0,0,0).
				\end{split}
			\end{equation*}
\end{rek}

There is a parallel story for Airy moments. The Airy sums are exponential sums defined for each prime power $q=p^r$ and each $a\in \fqq\cros$ by
	\begin{equation*}
		\begin{split}
			\rr{Ai}_n(a;q):=\sum_{x\in \fqq}\exp\bigl(2\pi i/p\cdot \tr_{\fqq/\fpp}\bigl(x^{n+1}-a\cdot x\bigr)\bigr).
		\end{split}
	\end{equation*}
They are finite field analogs of Airy functions
	\begin{equation*}
		\begin{split}
			\rr{Ai}_n(z):=\int_0^\infty \exp\Bigl(\frac{1}{n+1}x^{n+1}-z\cdot x\Bigr)\rr{d}x,
		\end{split}
	\end{equation*}
which satisfy the Airy differential equation\footnote{We use the integral of $\exp(\frac{1}{n+1}x^{n+1}-z\cdot x)$ instead of $\exp(x^{n+1}-z\cdot x)$ to define $\rr{Ai}_n(z)$, because we want $\rr{Ai}_n(z)$ to satisfy the differential equations $\partial_z^n-z=0$.} $\partial_z^n-z=0$. Like Kloosterman sums, we interpret Airy sums as traces of Frobenius acting on some $\ell$-adic local system $\rr{Ai}_n$ over $\bb{A}^1_{\fqq}$. The $k$-th symmetric power moments of Airy sums are the algebraic integers 
	\begin{equation*}
		\begin{split}
			-\sum_{a\in \fqq}\tr(\frob_q, (\Sym^k\rr{Ai}_n)_{\bar a}).
		\end{split}
	\end{equation*}

Sabbah and Yu constructed \textit{ulterior motives} $\rr{M}_n^k(\rr{Ai})$ attached to the $k$-th symmetric power moments of Airy sums in the sense of \cite{Anderson1986motive}. Contrary to the motives attached to Kloosterman moments, the de Rham realizations of $\rr{M}_n^k(\rr{Ai})$ underlie \textit{finite monodromic mixed Hodge structures} \cite[\S2]{sabbah2021airy}, which are pure of weight $k+1$ with $\bb{Q}$-indexed Hodge filtration. When $n=1$, they computed the Hodge filtration of $\rr{M}_2^k(\rr{Ai})_{\rr{dR}}$, whose Hodge numbers are either $0$ or $1$ \cite[Thm.\,1.1]{sabbah2021airy}. In the following theorem, we calculate the Hodge numbers of $\rr{M}_n^k(\rr{Ai})_{\rr{dR}}$ for $n>1$ when $\gcd(k,n)=1$.

\begin{thm}\label{thm:hodge number-Ai}
	Assume that $\gcd(k,n)=1$. The possible jumps of the $\bb{Q}$-indexed Hodge filtration on $\rr{M}_n^k(\rr{Ai})_{\rr{dR}}$ occur at $\frac{p+n+k}{n+1}$ for $0\leq p\leq nk-n-k+1$, and the Hodge numbers $h^{\frac{p+n+k}{n+1},\frac{nk+1-p}{n+1}}$ are the coefficients of $t^px^k$ in the formal power series expansion of the rational function
		\begin{equation*}
			\begin{split}
				\frac{1-t}{(1-t^{n})(1-x)(1-t x)\cdots(1-t^{n-1}x)}.
			\end{split}
		\end{equation*}
\end{thm}

In addition to the above cases of symmetric power moments of Kloosterman sums and Airy sums, we also deal with moments for reductive groups, defined by replacing $\Sym^k\Kl_{n+1}$ with Kloosterman sheaves for reductive groups from \cite{HNY-Kloosterman} in the formula~\eqref{eq:moments}, see Propositions~\ref{prop:hodge-numbers-Kl3V21}. In \cite{Qin23function}, we use Hodge numbers of two-dimensional motives to solve some conjectures of Evans type, raised by Yun in \cite{yun2015galois}, relating the symmetric power moments of Kloosterman sums with Fourier coefficients of modular forms.

\subsection{Sketch of the proof}
The proof relies on the irregular Hodge theory, developed in a series of works such as \cite{Deligne2007,Sabbah-2010-laplace,yu2012irregular,esnault-sabbah-yu2017e1,Sabbah-Yu-2015,sabbah18irregularhodge}. We take the proof of Theorem~\ref{thm:hodge number-Kl} as an example here. 
\begin{itemize}
	\item In \cite{fresan2018hodge}, Fres\'an, Sabbah, and Yu first constructed the motives $\mathrm{M}_{n+1}^{k}(\Kl)$ as exponential motives (in the sense of \cite{fresanexponential}) and then showed them to be isomorphic to (classical) motives. Inspired by the exponential motivic nature of $\mathrm{M}_{n+1}^{k}(\Kl)$, they showed that the de Rham realizations $\mathrm{M}_{n+1}^{k}(\Kl)_{\mathrm{dR}}$ underlie \textit{exponential mixed Hodge structures}, as defined by Kontsevich and Soibelman \cite{Kontsevich2011}. Since each exponential mixed Hodge structure is equipped with an irregular Hodge filtration \cite[Prop.\,A.10]{fresan2018hodge}, there are both classical and irregular Hodge filtration on $\mathrm{M}_{n+1}^{k}(\Kl)_{\mathrm{dR}}$.
	Fortunately, these two types of filtration coincide \cite[A.13]{fresan2018hodge}. So, it suffices to calculate the irregular Hodge filtration on $\mathrm{M}_{n+1}^{k}(\Kl)_{\mathrm{dR}}$.

	\item Let $f_k\colon \bb{G}_{m}^{nk+1}\to \bb{A}^1$ be the Laurent polynomial
		\begin{equation*}
			\begin{split}
				\sum_{i=1}^k\biggl(\sum_{j=1}^nx_{i,j}+\frac{z}{\prod_jx_{i,j}}\biggr).
			\end{split}
		\end{equation*}
	The twisted de Rham cohomology $\rr{H}^k(\bb{G}_m^{nk+1},f_k)$ of the pair $(\bb{G}_m^{nk+1},f_k)$ is defined as the hypercohomology $\bb{H}^k\bigl(\Omega^\bullet_{\bb{G}_m^{nk+1}},\rr{d}+\rr{d}f_k\bigr)$, which also underlies an exponential mixed Hodge structure \cite[Def.\,A.18]{fresan2018hodge}. We embed $\rr{M}_{n+1}^k(\rr{Kl})_{\rr{dR}}$  in $\mathrm{H}^{nk+1}_{\mathrm{dR}}(\bb{G}_m^{nk+1}, f_k)$ via the inclusion \eqref{eq:explicite-cohomology-class}, which is compatible with their irregular Hodge filtrations.

	\item Then we suitably choose a basis for each $\rr{M}_{n+1}^k(\rr{Kl})_{\rr{dR}}$. In the proof of \cite[Thm.\,1.6]{fresan2018hodge} for the case $n=1$, Fres\'an, Sabbah and Yu found an explicit basis for $\rr{M}_{2}^k(\rr{Kl})_{\rr{dR}}$ through the identification of $\rr{M}_{n+1}^k(\rr{Kl})_{\rr{dR}}$ with the middle de Rham cohomology of $\Sym^k\Kl_{n+1}$.
	
	For $n>1$, it is more complicated to get explicit bases for $\rr{M}_{n+1}^k(\rr{Kl})_{\rr{dR}}$ following their approach due to the complexity of the de Rham cohomology of $\Sym^k\Kl_{n+1}$. The most significant difficulty comes from the matrix $N_k$ in the differential \eqref{eq:theta_z} of the de Rham complex of $\Sym^k\Kl_{n+1}$ in \eqref{eq:derham}. The matrix $N_k$ is in fact the nilpotent part of the local monodromy operator of the connection $\Sym^k\Kl_{n+1}$ at $0$. When $n=1$, it is always nilpotent with one single Jordan block, while it is nilpotent with several Jordan blocks when $n>1$ and $k>1$. 
	
	This key point of this article is that the irregular Hodge filtration on $\rr{M}_{n+1}^k(\rr{Kl})_{\rr{dR}}$ depends not on an explicit basis but on the degree defined in \eqref{eq:grading}. With the help of a counting result (Proposition~\ref{prop:counting}), we choose bases for $\rr{M}_{n+1}^k(\rr{Kl})_{\rr{dR}}$ with prescribed degrees in Theorems~\ref{thm:mid-cohomology_classes} and~\ref{thm:mid-cohomology_classes-kl3}, which are good enough for our purpose.

	\item If $\gcd(k,n+1)=1$ (in Theorem~\ref{thm:hodge number-Kl} only the motives $\mathrm{M}_3^{3k}(\Kl)$ do not satisfy this condition), by the work of Adolphson--Sperber, Esnault--Sabbah--Yu, and Yu \cite{Adolphson1997twisted,esnault-sabbah-yu2017e1,yu2012irregular}, the filtration on the twisted de Rham cohomology $\mathrm{H}^{nk+1}_{\mathrm{dR}}(\bb{G}_m^{nk+1}, f_k)$ has a geometric interpretation in terms of the so-called \textit{Newton polyhedron filtration}, see Section~\ref{subsec:irreg-hodge-fil}. In Lemma~\ref{prop:hodge-number-non-degenerate-newton}, we show that each element $\omega$ in the basis for $\rr{M}_{n+1}^k(\rr{Kl})_{\rr{dR}}$ lies in 
		\begin{equation*}
			\begin{split}
				F^{p(\omega)}\mathrm{H}^{nk+1}_{\mathrm{dR}}(\bb{G}_m^{nk+1}, f_k)
			\end{split}
		\end{equation*}
	for an integer $p(\omega)$ depending only on the degree of $\omega$. We deduce that $\omega$ lies in $F^{p(\omega)}\rr{M}_{n+1}^k(\rr{Kl})_{\rr{dR}}$. Thanks to our suitably chosen bases, we conclude using the Hodge symmetry and our counting result (Proposition~\ref{prop:counting}) that each $\omega$ defines a nonzero class in $\gr_F^{p(\omega)}\rr{M}_{n+1}^k(\rr{Kl})_{\rr{dR}}$ in Theorem~\ref{thm:hodgenumberskl3-3nmidk}. Hence, the elements of the bases are adapted to the Hodge filtration, and we get the Hodge numbers of $\rr{M}_{n+1}^k(\rr{Kl})_{\rr{dR}}$.
	\item If $\gcd(k, n+1)>1$, as in the case where $k$ is even in \cite{fresan2018hodge,sabbah2021airy}, we need extra information to finish the calculation. For the case of the motives $\mathrm{M}_3^{3k}(\Kl)$, see Sections~\ref{subsec:inverse-fourier}, \ref{subsec:hdogekl3tilde}, and~\ref{subsec:hodgenumberskl3-3midk}. In general, we expect that the same argument works for $\rr{M}_{n+1}^k(\Kl)$ when the nilpotent part of the local monodromy operator of $\Sym^k\Kl_{n+1}$ at $0$ has Jordan blocks of different sizes, which is the case of $\Sym^4\Kl_4$, for example.
\end{itemize}

\subsection{Organization of the paper} 
In Section~\ref{sec:proposition-of-connection}, we recall some notation from the theory of $\scr{D}$-modules and some basic properties of Kloosterman and Airy connections. In Section~\ref{sec:EMHS}, we recall the definition of exponential mixed Hodge structures and the construction of exponential mixed Hodge structures attached to Kloosterman and Airy moments. Next, we determine the bases of cohomology classes of the de Rham cohomologies of Kloosterman and Airy connections in Section~\ref{sec:cohomology}. In Section~\ref{sec:hodge number}, we prove Theorems~\ref{thm:hodge number-Kl} and~\ref{thm:hodge number-Ai}.
\subsection*{Acknowledgement}
{This work is based on the author's Ph.D. thesis completed at Centre de Math\'ematiques Laurent Schwartz in \'Ecole Polytechnique. The author thanks his supervisors, Javier Fres\'an and Claude Sabbah, for proposing this question to him and for their guidance and fruitful discussions. The author also thanks Alberto Casta\~no Dom\'inguez, Lei Fu, and Christian Sevenheck for their feedback on a previous version of this article and to Gabriel Ribeiro, Bin Wang, Jeng-Daw Yu, and Bingyu Zhang for valuable discussions. Lastly, the author appreciates an anonymous referee for numerous suggestions to correct inaccuracies and enhance this paper's presentation.}

\section{Properties of Kloosterman and Airy connections}\label{sec:proposition-of-connection}
In this section, we gather some properties of Kloosterman and Airy connections, namely those concerning the local structures at $0$ and $\infty$, the irregularities, and the dimensions of their de Rham cohomologies.

\subsection{Notation for \texorpdfstring{$\scr{D}$}{D}-modules}

We recall some facts from the theory of $\scr{D}$-modules. For details, see \cite{HTT-dmodule} for example. 

Let $X$ be a smooth algebraic variety over $\bb{C}$. We denote by $\scr{D}_X$ the \textit{sheaf of differential operators} on $X$, which is a subsheaf of $\mathcal{E}nd_{\bb{C}_X}(\mathcal{O}_X)$, generated by $\mathcal{O}_X$ (acting on $\mathcal{O}_X$ by left multiplication), and the sheaf of vector fields $\Theta_X=(\Omega_X^1)^\vee$. For example, if $X=\bb{A}^1$, the sheaf $\scr{D}_X$ is associated with the Weyl algebra $\bb{C}[t]\langle\partial_t\rangle$, satisfying $\partial_t\cdot t-t\cdot \partial_t=1$. A left (resp. right) $\scr{D}_X$-module is an $\mathcal{O}_X$-module with a left (resp. right) action of $\scr{D}_X$. When $X$ is affine, we identify $\scr{D}_X$-modules with their global sections. For example, $\scr{D}_{\bb{A}^1}$-modules are identified with $\bb{C}[t]\langle\partial_t\rangle$-modules. We denote by $\mathrm{Mod}(\scr{D}_X)$ the category of (left) $\scr{D}_X$-modules on $X$, and by $\mathrm{D}^b(\scr{D}_X)$ the bounded derived category of $\scr{D}_X$-modules.

Let $f\colon X\to Y$ be a morphism of smooth complex varieties. Let $\scr{D}_{X\to Y}$ and $\scr{D}_{Y \leftarrow X}$ be the transfer modules associated with $f$. For a complex of $\mathscr{D}_X$-modules $M$ and a complex of $\mathscr{D}_Y$-modules $N$, the direct image of $M$ and the inverse image of $N$ are defined by 
	\begin{equation*}
		\begin{split}
			f_+M=\mathrm{R}f_*(\scr{D}_{Y\leftarrow X}\otimes^{\rr{L}}_{\scr{D}_X}M)\in \mathrm{D}^b(\scr{D}_Y)
		\end{split}
	\end{equation*}
and 
	\begin{equation*}
		\begin{split}
			f^+N=\scr{D}_{X\to Y}\otimes^{\rr{L}}_{f\inv \scr{D}_Y}f\inv N\in \mathrm{D}^b(\scr{D}_X )
		\end{split}
	\end{equation*}
respectively. Let $\mathbb{D}_X$ and $\bb{D}_Y$ be the duality functors. We put $f_\dagger=\bb{D}_Y\circ f_+\circ \bb{D}_X.$ Then there is a canonical forget supports morphism $f_\dagger M \to f_+M$, which is an isomorphism if $f$ is proper.

A $\scr{D}_X$-module is called \textit{holonomic} if the dimension of its characteristic variety is equal to $\dim X$. We denote by $\rr{D}^b_h(\scr{D}_X)$ the full subcategory of $\rr{D}^b(\scr{D}_X)$ whose objects are complexes with holonomic cohomologies. 

\subsubsection{De Rham cohomology}
For a $\scr{D}_X$-module $M$, we denote by $\mathrm{DR}(M)$ its de Rham complex $\Omega_{X}\otimes_{\mathcal{O}_X}^{\mathrm{L}}M$. The de Rham cohomology $\mathrm{H}^r_{\dR}(X, M)$ of a $\scr{D}_X$-module $M$ is the hypercohomology of the de Rham complex of $M$, which is finite-dimensional if $M$ is holonomic.

Now let $j\colon U\hookrightarrow X$ be the inclusion of an open subvariety $U$ in a projective smooth variety $X$, whose complement is a divisor $D$. For a connection on $U$, i.e., a locally free $\scr{D}_U$-module $M$, the de Rham cohomology with compact support $\mathrm{H}^r_{\mathrm{dR, c}}(U,M)$ is the hypercohomology of $\rr{DR}(j_\dagger M)$. We denote by $j_{\dagger+}M$ the intermediate extension of $M$, which is the $\scr{D}_X$-module $N$ such that $j^+N=M$ and $N$ admits no subobjects or quotient objects supported on $D$.

If $X$ is a curve, we denote by $\mathrm{H}^r_{\mathrm{dR, mid}}(U,{M})$ the \textit{middle de Rham cohomology} of ${M}$, i.e., the image of the de Rham cohomology with compact support in the usual de Rham cohomology. It is identified with the de Rham cohomology $\mathrm{H}^1_{\mathrm{dR}}(X,j_{\dagger+}{M})$ of the intermediate extension $j_{\dagger+}{M}$ \cite[(3.1)]{fresan2020quadratic}.

\subsubsection{The Fourier transform}
Let $\pr_t$ and $\pr_\tau$ be the projections of $\bb{A}^1_t\times \bb{A}^1_\tau$ to the first and the second factors respectively. We denote by $\mathcal{E}^{t\tau}$ the rank $1$ connection $(\cc{O}_{\bb{A}^1_t\times \bb{A}^1_\tau},\rr{d}+\rr{d}(t\tau))$ on $\bb{A}^1_t\times \bb{A}^1_\tau$. Then the \textit{Fourier transform} of a $\scr{D}_{\bb{A}^1_t}$-module $M$ on $\bb{A}^1_t$ is given by 
        \begin{equation*}
			\begin{split}
				\rr{FT}\,M=\pr_{\tau+}\bigl(\pr_t^+M\otimes_{\cc{O}_{\bb{A}^1\times \bb{A}^1}} \cc{E}^{t\tau}\bigr).
			\end{split}
		\end{equation*}
If $M$ is a holonomic $\scr{D}_{\bb{A}^1_t}$-module, its Fourier transform is a holonomic $\scr{D}_{\bb{A}^1_\tau}$-module.

\subsubsection{The nearby cycle and the vanishing cycle functors}
Let $f\colon X\to \bb{A}^1$ be a regular function. In terms of the \textit{Kashiwara--Malgrange filtration} of a holonomic $\scr{D}_X$-module $M$, we can define the nearby cycle $\psi_fM$ and the vanishing cycle $\phi_fM$ of $M$, which are holonomic $\scr{D}_X$-modules supported on $f\inv(0)$. 

By construction, there is an automorphism $T$ on both $\psi_fM$ and $\phi_fM$. So we can decompose them, and denote by $\psi_{f,\lambda}M$ and $\phi_{f,\lambda}$ the generalized eigenspaces corresponding to an eigenvalue $\lambda\in \bb{C}\cros$. If $\lambda\neq1$, there is an isomorphism $\psi_{f,\lambda}M\simeq \phi_{f,\lambda}M$ compatible with the automorphism $T$. If $\lambda=1$, there are morphisms $\rr{can}$ and $\rr{var}$ as follows:
        $$\begin{tikzcd}
                \psi_{f,1}M \ar[r,bend left,"\rr{can}"] &  \phi_{f,1}M,\ar[l,bend left,"\rr{var}"]
        \end{tikzcd}$$
such that $\exp(2\pi i\, \rr{var}\circ \rr{can})$ and $\exp(2\pi i \,\rr{can}\circ \rr{var})$ are equal to the unipotent automorphisms $T$ on $\psi_{f,1}M$ and $\phi_{f,1}M$ respectively. The nilpotent endomorphisms $\rr{var}\circ\rr{can}$ and $\rr{can}\circ\rr{var}$ on $\psi_{f,1}M$ and $\phi_{f,1}$ are denoted by $N$.

\subsection{Kloosterman connections}\label{sec:Kloosterman-connection}

Let $\{x_i\}_{1\leq i\leq n}$ and $z$ be the coordinates of $\bb{G}_{m}^{n+1}$, and we consider the diagram
	\begin{equation}\label{defining-diagram}
		\begin{tikzcd}
			 &	\bb{G}_m^{n+1}\ar[rd,"\pr_z"]\ar[ld,"f"']	&	\\
			 \bb{A}^1_x & &\bb{G}_{m,z}	
		\end{tikzcd}
	\end{equation}
where
	\begin{equation}\label{eq:f-kloosterman}
		\begin{split}
			f(x_i,z)=\sum_{i=1}^n x_i+\frac{z}{\prod_{i=1}^n x_i},
		\end{split}
	\end{equation} 
and $\pr_z$ is the projection to $\bb{G}_{m,z}$. The \textit{Kloosterman connection} is defined as
	\begin{equation*}
		\begin{split}
			\Kl_{n+1}:=\mathcal{H}^0\pr_{z+}f^+\mathcal{E}^x,
		\end{split}
	\end{equation*}
where $\mathcal{E}^x$ denotes the rank $1$ vector bundle with connection $(\mathcal{O}_{\bb{A}^1},\mathrm{d}+\mathrm{d}x).$ Besides, we consider the pullback $\widetilde{\Kl}_{n+1}:=[n+1]^+\Kl_{n+1}$ of $\Kl_{n+1}$ along the cover 
	\begin{equation*}
		\begin{split}
			[n+1]\colon \bb{G}_{m,t}\to \bb{G}_{m,z},\,t\mapsto t^{n+1}=z.
		\end{split}
	\end{equation*}
Notice that $\mu_{n+1}$ is the automorphism group of the cover $[n+1]$. We can recover the connection $\Kl_{n+1}$ as the $\mu_{n+1}$-invariants of $[n+1]_+\tilde{\Kl}_{n+1}$. 

\begin{prop}
	The Kloosterman connections have the following properties:
	\begin{enumerate}
		\item $\Kl_{n+1}$ is a free $\mathcal{O}_{\bb{G}_m}$-module of rank $n+1$, whose connection, in terms of some basis $\{v_0,v_1,\ldots,v_n\}$, is given by
		\begin{equation}\label{eq:connection}
			\mathrm{d}+N\frac{\mathrm{d}z}{z}+E\,\mathrm{d}z,
		\end{equation}
		where $N$ is the lower triangular Jordan block of size $n+1$ with eigenvalue $0$, and $E$ is the matrix with entry $1$ at row $1$ and column $n+1$ and with entries $0$ elsewhere.
		 
		\item  $\Kl_{n+1}$ has a regular singularity at $0$ and an irregular singularity at $\infty$ of slope $\frac{1}{n+1}$. 
		\item We have an isomorphism
		\begin{equation*}
			\begin{split}
				\Kl_{n+1}^\vee\simeq \iota_{n+1}^+\Kl_{n+1}
			\end{split}
		\end{equation*}
	where $\Kl_{n+1}^{\vee}$ is the dual of $\Kl_{n+1}$, and $\iota_{n+1}$ is the involution $z\mapsto (-1)^{n+1}z$.
	\end{enumerate}
\end{prop}
\begin{proof}
	These are the first three properties in \cite[Prop.\,2.4]{fresan2018hodge}.
\end{proof}
From the first property above, the Kloosterman connections are examples of Frenkel--Gross connections for $\rr{SL}_{n+1}$ \cite[\S6.1]{Frenkel2009}. We denote by $\Sym^k\Kl_{n+1}$ (resp. $\Sym^k\widetilde\Kl_{n+1}$) the $k$-th symmetric power of $\Kl_{n+1}$ (resp. $\widetilde \Kl_{n+1}$), which is of rank $\binom{n+k}{k}$. 

\subsubsection{The local structures at \texorpdfstring{$0$}{0}}\label{subsec:local-at-0}
We study the formal local structures at $0$ of $\Sym^k\Kl_{n+1}$ and $\Sym^k\widetilde \Kl_{n+1}$. Let $V=\bb{C}^{n+1}$ be the standard representation of $\SL_{n+1}$, and $V_k:=\Sym^kV$ the $k$-th symmetric power of the representation $V$. We denote by $\{v_0,\ldots,v_{n}\}$ the standard basis for $V$ and by $N$ the matrix 
	\begin{equation*}
		\begin{split}
			N=\left(\begin{matrix} 
				0 & 0 & \cdots & 0 & 0 \\ 
				1 & 0 & \cdots & 0 & 0 \\ 
				0 & 1 & \cdots & 0 & 0 \\
				\cdots & \cdots & \cdots & \cdots & \cdots \\
				0 & 0 & \cdots & 1 & 0 \end{matrix}\right).
		\end{split}
	\end{equation*}
By construction, we have $Nv_i=v_{i+1}$ for $i=0,\ldots, n-1$ and $Nv_n=0$. The action of $N$ on $V$ can be enhanced to a Lie algebra representation $\rho\colon \mathrm{sl}_2\to \mathrm{End}(V)$, such that 
	$\rho\left(\begin{matrix}0&0\\1&0\end{matrix}\right)=N$,
	$\rho\left(\begin{matrix}1&0\\0&-1\end{matrix}\right)=\rr{diag}(n,n-2,\ldots,-n)$, and 
    $\rho\left(\begin{matrix}0&1\\0&0\end{matrix}\right)$ is a matrix with nonzero items $i(i-1-n)$ at $i$-th column and $(n+1-i)$-th row.
Using $\rho$, we view $V_k$ as a representation of $\mathrm{sl}_2$, such that $\rho\left(\begin{matrix}0&0\\1&0\end{matrix}\right)=\Sym^kN$. For simplicity, we denote $N_k:=\Sym^k N$.

Recall that the category of representations of $\mathrm{sl}_2$ is semisimple, and all irreducible representations of $\mathrm{sl}_2$ are of the form $\Sym^d(\bb{C}^2)$. In the following lemma, we make the decomposition of $V_k$ into irreducible representations explicit.

\begin{lem}\label{lem:decomposition}
As representations of $\mathrm{sl}_2$, we have
	\begin{equation*}
		\begin{split}
			V_k=\bigoplus_{d=0}^{\lfloor\frac{nk}{2}\rfloor}\Sym^{nk-2d}(\bb{C}^2)^{\oplus q_{d,k}},
		\end{split}
	\end{equation*}
where $q_{d,k}$ are the coefficients of $t^d$ in the formal power series expansion of 
	\begin{equation*}
		\begin{split}
			Q_k(t)=\frac{(1-t^{n+1})\cdots (1-t^{n+k})}{(1-t^2)\cdots (1-t^k)}.
		\end{split}
	\end{equation*}
In particular, the cokernel of $N_k$ on $V_k$ has dimension $\sum_{d=0}^{\lfloor\frac{nk}{2}\rfloor} q_{d,k}$. 
\end{lem}

\begin{proof} 
The representation $V_k$ is of the form $\bigoplus_{d=0}\Sym^{d}(\bb{C}^2)^{p_{d,k}}$. Fu and Wan showed that $p_{d,k}$ are those numbers $c_k(d)-c_k(d-1)$ from \cite[p.\,559]{fu2006trivial} and gave the formula of $p_{d,k}$ in\footnote{Our formula is slightly different from the original formula in \cite[Thm.\,0.1]{fu2006trivial} because we study $\Sym^k\Kl_{n+1}$ instead of $\Sym^k\Kl_n$.} \cite[Thm.\,0.1]{fu2006trivial}.
\end{proof}

The formal structures of the connections $\Sym^k\Kl_{n+1}$ and $\Sym^k\widetilde\Kl_{n+1}$ at $0$ are isomorphic to
	\begin{equation*}
		\begin{split}
			\bb{C}((z)) \otimes_{\bb{C}[z,z\inv]}  \Sym^k\Kl_{n+1}\simeq \Bigl(\mathcal{O}_{\bb{G}_m}^{\binom{n+k}{n}},\mathrm{d}-N_k\frac{\mathrm{d}z}{z}\Bigr)
		\end{split}
	\end{equation*}
and 
	\begin{equation*}
		\begin{split}
			\bb{C}((t)) \otimes_{\bb{C}[t,t\inv]}  \Sym^k\widetilde\Kl_{n+1}
	\simeq \Bigl(\mathcal{O}_{\bb{G}_m}^{\binom{n+k}{n}},\mathrm{d}-(n+1)N_k\frac{\mathrm{d}t}{t}\Bigr)
		\end{split}
	\end{equation*}
respectively. Using Lemma~\ref{lem:decomposition}, we have the following result:

\begin{prop}\label{prop:local-formal-at-0}
	The formal structures of the connections $\Sym^k\Kl_{n+1}$ and $\Sym^k\widetilde\Kl_{n+1}$ at $0$ are isomorphic to
		\begin{equation*}
			\begin{split}
				\bigoplus_{d=0}^{\lfloor\frac{nk}{2}\rfloor}\Bigl(\mathcal{O}_{\bb{G}_m}^{nk-2d+1},\mathrm{d}-J_{nk-2d+1}(0)\frac{\mathrm{d}z}{z}\Bigr)^{\oplus q_{d,k}}
			\end{split}
		\end{equation*}
	and 
		\begin{equation*}
			\begin{split}
				\bigoplus_{d=0}^{\lfloor\frac{nk}{2}\rfloor}\Bigl(\mathcal{O}_{\bb{G}_m}^{nk-2d+1},\mathrm{d}-(n+1)J_{nk-2d+1}(0)\frac{\mathrm{d}t}{t}\Bigr)^{\oplus q_{d,k}}
			\end{split}
		\end{equation*}
	respectively, where each matrix $J_{nk-2d+1}(0)$ represents a Jordan block of size $nk-2d+1$ with eigenvalue $0$. 
\end{prop}

\begin{cor}\label{cor::soln0}
	Let $\rr{Soln}_0$ and $\tilde{\rr{Soln}}_0$ be the dimensions of the formal solution spaces $\rr{H}om(\bb{C}((z))\otimes \Sym^k\Kl_{n+1},\bb{C}((z)))$ and $\rr{H}om(\bb{C}((t))\otimes \Sym^k\widetilde\Kl_{n+1},\bb{C}((t)))$ respectively. Then $\rr{Soln}_0=\tilde{\rr{Soln}}_0=\sum_{d=0}^{\lfloor\frac{nk}{2}\rfloor}q_{d,k}$.
\end{cor}
\begin{proof}
	Let $\cc{M}$ be $\Sym^k\Kl_{n+1}$ or $\Sym^k\tilde{\Kl}_{n+1}$ and $j_0\colon \bb{G}_m\hookrightarrow \bb{A}^1$ the inclusion. By Lemma~\ref{lem:decomposition} and Proposition~\ref{prop:local-formal-at-0}, we know that the cokernel of 
			\begin{equation*}
				\begin{split}
					j_{0\dagger+}\mathcal{M}\to j_{0+}\mathcal{M}
				\end{split}
			\end{equation*}
	is supported at $0$, of rank $\sum_{d=0}^{\lfloor\frac{nk}{2}\rfloor}q_{d,k}$. Hence, $\mathrm{Soln}_0=\sum_{d=0}^{\lfloor\frac{nk}{2}\rfloor}q_{d,k}$ by \cite[Prop.\,2.9.8]{katz1990exponential}.
\end{proof}

\begin{exe}
	If $n=2$, the representation $V_k$ is of the form 
		\begin{equation*}
			\begin{split}
				\bigoplus_{d=0}^{\lfloor\frac{k}{2}\rfloor}\Sym^{2k-2d}(\bb{C}^2),
			\end{split}
		\end{equation*}
	and $\rr{Soln}_0=\sum_{d=0}^{k}q_{d,k}$ is $1+\lfloor\frac{k}{2}\rfloor$.
\end{exe}

\subsubsection{The local structures at \texorpdfstring{$\infty$}{infinity}}\label{subsec:local-at-infty}
We study the formal local structures of $\Sym^k\Kl_{n+1}$ and $\Sym^k\widetilde \Kl_{n+1}$ at $\infty$. Let $[n+1]\colon \bb{G}_{m,t}\to \bb{G}_{m,z}$ be the $(n+1)$-th power map and $[2]\colon \bb{G}_{m,s}\to \bb{G}_{m,t}$ the square map, with Galois groups $\mu_{n+1}$ and $\mu_2$ respectively. We denote by $\zeta$ a primitive $2(n+1)$-th root of unity, by $\zeta_{n+1}$ the square of $\zeta$, and by $\mathbb{L}_{-1}$ the connection $(\cc{O}_{\bb{G}_m},\mathrm{d}+\frac{1}{2}\tfrac{\mathrm{d} t\inv}{t\inv})$ on $\bb{G}_{m,t}$. If there is no confusion, we also use the same symbol $\cc{E}^t$ and $\bb{L}_{-1}$ to denote their formal completions at $\infty$ for simplicity.

Let $\hat{\Kl}_{n+1}$ (resp. $\hat{\tilde{\Kl}}_{n+1}$) be the formal connection $\bb{C}((z\inv)) \otimes_{\bb{C}[z,z\inv]}  \Kl_{n+1}$ (resp. $\bb{C}((t\inv)) \otimes_{\bb{C}[t,t\inv]}  \tilde{\Kl}_{n+1}$). The formal local structures of $\Kl_{n+1}$ and $\widetilde \Kl_{n+1}$ are determined in the following lemma:
\begin{lem}
We have isomorphisms of formal connections 
	\begin{equation}\label{eq:Kl3-infinity}
		\hat{\Kl}_{n+1}\simeq [n+1]_+\bigl(\mathcal{E}^{(n+1)t}\otimes \bb{L}_{-1}^{\otimes n}\bigr)\simeq \Bigl([2n+2]_+\mathcal{E}^{(n+1)s^2}\Bigr)^{\mu_{2},\chi^n},
	\end{equation}
and
	\begin{equation}\label{eq:Kl3tilde-infty}
		\hat{\tilde \Kl}_{n+1}\simeq \bigoplus_{i=0}^n \mathcal{E}^{(n+1)\zeta^it}\otimes \bb{L}_{-1}^{\otimes n}.
	\end{equation}
Here $\chi$ is the unique quadratic character of $\mu_2$, and the exponent $(\mu_2,\chi^n)$ means taking the $\chi^n$-isotypic components.
\end{lem}
\begin{proof}

Let $\mathrm{inv}\colon \bb{G}_{m,z}\to \bb{G}_{m,z}$ be the inversion map $z\mapsto z\inv$ and $j\colon \bb{G}_{m,z}\hookrightarrow \bb{A}^1_z$ the inclusion map. By \cite[Lem.\,2.5]{fresan2018hodge}, we have 
	\begin{equation*}
		\begin{split}
			\Kl_{n+1}=j^+\mathrm{FT}(j_+\mathrm{inv}^+\Kl_{n})
		\end{split}
	\end{equation*}
for each $n\in \bb{N}_{>0}$, where $\Kl_1=\cc{E}^t$.  We prove the first isomorphism in the formula \eqref{eq:Kl3-infinity} by induction on $n$.

The case that $n=1$ is verified by the equation~(4.3) in \textit{loc. cit}. Assume that the formula for $\Kl_{n}$ holds. Let $j_\infty\colon \bb{G}_m\hookrightarrow \bb{P}^1\backslash\{0\}$ be the inclusion. Applying the formal stationary phase formula \cite[Thm 5.1]{sabbah2008explicit} to $j_{\infty+}\Kl_{n+1}=j_{\infty+}j^+\mathrm{FT}(j_+\mathrm{inv}^+\Kl_n)$, in other words, to $\rho= t^{n+1},\,\varphi=(n+1)t$ and $R=\bb{L}_{-1}^{\otimes (n-1)}$ according to the notation in \textit{loc. cit.}, we have the first part of the formula \eqref{eq:Kl3-infinity}. 

The second isomorphism in \eqref{eq:Kl3-infinity} follows from 
\begin{equation*}
	\begin{split}
		[n+1]_+\bigl(\mathcal{E}^{(n+1)t}\otimes \bb{L}_{-1}^{\otimes n}\bigr)&\simeq [n+1]_+\Bigl(\mathcal{E}^{(n+1)t}\otimes \Bigl([2]_+\Bigl(\bb{C}((s\inv))\Bigr)\Bigr)^{\mu_{2},\chi^n}\Bigr)\\
		&\simeq [n+1]_+\Bigl([2]_+\mathcal{E}^{(n+1)s^2}\Bigr)^{\mu_{2},\chi^n}\\
		&\simeq \Bigl([2n+2]_+\mathcal{E}^{(n+1)s^2}\Bigr)^{\mu_{2},\chi^n},
	\end{split}
\end{equation*} 
where we use the isomorphism $\bb{L}_{-1}^{\otimes n}=\bigl([2]_+\bb{C}((s\inv))\bigr)^{\mu_2,\chi^n}$ in the first isomorphism and the projection formula in the second isomorphism.

At last, let $M$ be the connection $(\cc{O}_{\bb{G}_{m,z}},\rr{d}+\frac{1}{2n+2}\frac{\rr{d}z}{z})$. By the projection formula, we have
\begin{equation*}
	\begin{split}
		\hat{\tilde \Kl}_{n+1}
		&=[n+1]^+\bigl([n+1]_+(\cc{E}^{(n+1)t}\otimes \bb{L}_{-1}^{\otimes n})\bigr)\\
		&=[n+1]^+\bigl([n+1]_+(\cc{E}^{(n+1)t})\otimes M^{\otimes n}\bigr)\\
		&=[n+1]^+[n+1]_+\cc{E}^{(n+1)t}\otimes [n+1]^+M^{\otimes n}.
	\end{split}
\end{equation*}
Therefore, we get \eqref{eq:Kl3tilde-infty} by the isomorphism $[n+1]^+[n+1]_+\mathcal{E}^{(n+1)t}=\bigoplus_{i=0}^n \mathcal{E}^{(n+1)\zeta^it}$ \cite[Lem.\,2.4]{sabbah2008explicit} and $[n+1]^+M=\bb{L}_{-1}$.
\end{proof}

For each multi-index $\underline I=(I_0,\ldots,I_n)$ in $\bb{N}^{n+1}$, we denote by $C_{\underline I}$ the complex number $\sum_{i=0}^{n}I_i\zeta^i$. Let $d(k,n+1)$ be the cardinality of the set 
\begin{equation}\label{eq:dk}
	\begin{split}
		\Bigl\{\underline I\in \bb{N}^{n+1}\,\Big{|}\, |\underline I|:=\sum_{j=0}^n I_j=k,\,C_{\underline I}=0\Bigr\}.
	\end{split}
\end{equation}
We determine the formal structures of $\Sym^k\tilde \Kl_{n+1}$ at $\infty$ in the following proposition:

\begin{prop}\label{prop:local-infinity-Symk}
	\begin{enumerate}
		\item We have an isomorphism of formal connections
		\begin{equation*}
			\begin{split}
				\Sym^k\hat{\widetilde\Kl}_{n+1}
			=\bigoplus_{|\underline I|=k,\,C_{\underline I=0}}\bb{L}_{-1}^{\otimes nk} \bigoplus \bigoplus_{|\underline I|=k,\,C_{\underline I\neq 0}} \mathcal{E}^{(n+1)C_{\underline I}t} \otimes \bb{L}_{-1}^{\otimes nk}.
			\end{split}
		\end{equation*}
		\item The irregularities of $\Sym^k\Kl_{n+1}$ and $\Sym^k\widetilde \Kl_{n+1}$ at $\infty$ are $\bigl(\binom{n+k}{n}-d(k,n+1)\bigr)\big{/}(n+1)$ and $\binom{n+k}{n}-d(k,n+1)$ respectively.
	\end{enumerate}
\end{prop}

\begin{proof}
	\noindent\textbf{1.} By taking the symmetric power of \eqref{eq:Kl3tilde-infty}, we have
		\begin{equation*}
			\begin{split}
				\Sym^k\hat{\widetilde\Kl}_{n+1}=\Sym^k\Bigl(\oplus_{i=0}^n \mathcal{E}^{(n+1)\zeta^i t}\Bigr)\otimes \bb{L}_{-1}^{\otimes nk}=\bigoplus_{|\underline I|=k}\mathcal{E}^{(n+1)C_{\underline I}t}\otimes \bb{L}_{-1}^{\otimes nk}.
			\end{split}
		\end{equation*}
	Then notice that when $C_{\underline I}=0$, we have $\cc{E}^{(n+1)C_{\underline I}t}=\cc{O}_{\bb{A}^1}$.

	\noindent\textbf{2.} The irregularity of $\mathcal{E}^{\lambda t}$ at $\infty$ is $1$ if $\lambda\neq 0$ and $0$ if $\lambda=0$. By the local structure of $\Sym^k\tilde \Kl_{n+1}$, the irregularity of $\Sym^k\widetilde\Kl_{n+1}$ at $\infty$ is $\rk\,\Sym^k\widetilde\Kl_{n+1}-d(k,n+1)$, i.e., the number of $\underline I$ such that $C_{\underline I}\neq 0$.

	Recall that we have $[n+1]^+[n+1]_+\cc{E}^{at}=\bigoplus_{i=0}^n\cc{E}^{a\cdot \zeta_{n+1}^i}$ for $a\neq 0$. By the definition of slopes (see \cite[(2.2.5)]{katz1987galois}), we deduce that $[n+1]_+\cc{E}^{at}$ is of rank $n+1$ and slope $\frac{1}{n+1}$. Therefore, the irregularity of $[n+1]_+\mathcal{E}^{(n+1)C_{\underline I}t}$ at $\infty$ is $1$ if $C_{\underline I}\neq 0$, and the irregularity of $\Sym^k\Kl_{n+1}$ at $\infty$ is $\rr{Irr}(\Sym^k\tilde\Kl_{n+1})/(n+1)$.
\end{proof}

Next, we determine the dimensions of the formal solution spaces at $\infty$ of symmetric power of Kloosterman connections, an analog of the work of Fu--Wan for $\ell$-adic sheaves \cite{fu2005functions}. 

The element $\sigma=(0,\ldots,n)\in S_{n+1}$ acts on the set of multi-indices by 
	\begin{equation*}
		\begin{split}
			\sigma \cdot \underline I = (I_n,I_0,\ldots,I_{n-1}).
		\end{split}
	\end{equation*}
Let $a(k,n+1)$ be the cardinality of the set 
\begin{equation}\label{eq:ak}
	\begin{split}
		\Sigma(k,n+1):=\Bigl\{\underline I \mid |\underline I|=\sum_{i=0}^nI_i=k,\,C_{\underline I}=0\Bigr\}\Big{/}(\underline I\sim \sigma\cdot \underline I),
	\end{split}
\end{equation}
i.e., the set of orbits of multi-indices with $|\underline I|=k$ and $C_{\underline I}=0$. We denote by $b(k,n+1)$ the cardinality of the set 
\begin{equation}\label{eq:bk}
	\begin{split}
		\Bigl\{[\underline I]\in \Sigma(k,n+1)\,\Big{|}\, \sum_{i=0}^n (-1)^{\sum_{j=n+1-i}^n I_j}\sigma^i\cdot \underline I\neq 0 \text{ in } \bb{Z}[\bb{N}^{n+1}]\Bigr\}.
	\end{split}
\end{equation}

\begin{prop}\label{prop:soln-infty}
	Let $\rr{Soln}_\infty$ and $\tilde{\rr{Soln}}_\infty$ be the dimensions of the solution spaces $\rr{Hom}(\Sym^k{\hat\Kl}_{n+1},\bb{C}((z\inv)))$ and $\rr{Hom}(\Sym^k\hat{\widetilde\Kl}_{n+1},\bb{C}((t\inv)))$ respectively. Then 
	\begin{enumerate}
		\item $\tilde{\rr{Soln}}_\infty$ is $d(k,n+1)$ if $2\mid nk$ and $0$ otherwise.
		\item  $\rr{Soln}_\infty$ is 
		$\begin{cases}
			a(k,n+1) & 2\mid n,\\
			0		 & 2\nmid nk,\\
			b(k,n+1) & \text{else}.
		\end{cases}$
	\end{enumerate}
\end{prop}
\begin{proof} Let $\cc{M}$ be either $\Sym^k\Kl_{n+1}$ or $\Sym^k\tilde\Kl_{n+1}$ and $j_\infty\colon \bb{G}_m\hookrightarrow \bb{P}^1\backslash\{0\}$ be the inclusion. The rank of the cokernel of the injective morphism 
		\begin{equation*}
			\begin{split}
				j_{\infty\dagger+}\mathcal{M}\to j_{\infty+}\mathcal{M}
			\end{split}
		\end{equation*}
	is $\mathrm{Soln}_\infty$ or $\tilde{\mathrm{Soln}}_\infty$, see \cite[Prop.\,2.9.8]{katz1990exponential}.

\noindent\textbf{1.} By Proposition~\ref{prop:local-infinity-Symk}, the direct summand $\mathcal{E}^{(n+1)C_{\underline I}t} \otimes \bb{L}_{-1}^{\otimes nk}$ is trivial if and only if $C_{\underline I}=0$ and $2\mid nk$. So the number $\tilde{\mathrm{Soln}}_\infty$ is $d(k,n+1)$ if $2\mid nk$ and $0$ otherwise.

\noindent \textbf{2.} Recall that the cover $[2n+2]\colon\bb{G}_{m,s}\to \bb{G}_{m,z}$ of $\bb{G}_{m,z}$ has Galois group $\mu_{2n+2}$. We have an isomorphism 
	\begin{equation*}
		\begin{split}
			\Kl_{n+1}=\bigl([2n+2]_+[2n+2]^+\Kl_{n+1}\bigr)^{\mu_{2n+2}}.
		\end{split}
	\end{equation*}
We deduce from the above that
	\begin{equation*}
		\begin{split}
			\Sym^k\Kl_{n+1}=\bigl([2n+2]_+\Sym^k\bigl([2n+2]^+\Kl_{n+1}\bigr)\bigr)^{\mu_{2n+2}}.
		\end{split}
	\end{equation*}

We choose a generator $e$ of $\mathcal{E}^{(n+1)s^2}$. The set $\{s^ie\mid 0\leq i\leq 2n+1, i\equiv n\, (\rr{mod}\,2)\}$ is a basis for $ \hat \Kl_{n+1}=\bigl([2n+2]_+\mathcal{E}^{(n+1)s^2}\bigr)^{\mu_2,\chi^n}$. Let $e_i$ be the element $s^{-2i}\otimes s^{2i}e$ (resp. $s^{-2i-1}\otimes s^{2i+1}e$) in $[2n+2]^+\hat\Kl_{n+1}$ when $2\mid n$ (resp. $2\nmid n$) for $0\leq i\leq n$. Then we choose the set $\{e_{i}\mid 0\leq i\leq n\}$ as a basis for $[2n+2]^+\hat\Kl_{n+1}$. By \cite[(1.2)]{sabbah2008explicit}, we have 
	\begin{equation}\label{eq:e_i}
		s\partial_s\,\cdot e_i=(2n+2)s^2 e_{i+1} 
	\end{equation}
for $0\leq i\leq n$ (we take $e_{n+1}=e_0$). Let $f_{j}=\sum_{i=0}^n \zeta^{-2ij}e_{i}$ (resp. $\sum_{i=0}^n \zeta^{-(2i+1)j}e_{i}$) for $0\leq j\leq n$ if $2\mid n$ (resp. $2\nmid n$). Since $\zeta\cdot s^{-i}\otimes s^ie = (\zeta\cdot s^{-i})\otimes s^ie=\zeta^{-i}s^{-i}\otimes s^ie$. We have 
	\begin{equation*}
		\begin{split}
			\zeta\cdot f_{j}=f_{j+1} \text{ for } 0\leq j\leq n-1, \quad \text{and} \quad \zeta\cdot f_n=(-1)^n f_0.
		\end{split}
	\end{equation*}

By \eqref{eq:e_i}, each $f_{j}$ generates a copy of $\cc{E}^{(n+1)\zeta^{j}s^2}$ in $[2n+2]^+\Kl_{n+1}$, and we have 
	\begin{equation*}
		\begin{split}
			[2n+2]^+\hat\Kl_{n+1}=\bigoplus_{j=0}^n \cc{E}^{(n+1)\zeta^{j}s^2}.
		\end{split}
	\end{equation*}
Taking the symmetric power on both sides, we have 
	\begin{equation*}
		\begin{split}
			[2n+2]^+\Sym^k\hat\Kl_{n+1}=\bigoplus_{\underline I\in \bb{N}^{n+1},\,|\underline I|=k} \bb{C}((s\inv))\bigoplus \bigoplus_{\underline I\in \bb{N}^{n+1},\,|\underline I|=k} \cc{E}^{(n+1)C_{\underline I}s^2},
		\end{split}
	\end{equation*}
and each component for $\underline I$ is generated by $f^{\underline I}:= \prod_{j=0}^n f_{j}^{I_j}$. Notice that the action of $\mu_{2n+2}$ on the basis $\{f^{\underline I}\}$ is given by
	\begin{equation*}
		\begin{split}
			\zeta\cdot f^{\underline I}=(-1)^{n\cdot I_n}f^{\sigma\underline I}.
		\end{split}
	\end{equation*}

The number $\rr{Soln}_\infty=\dim \rr{Hom}(\Sym^k{\hat\Kl}_{n+1},\bb{C}((z\inv)))$ is the rank of trivial connections in the fixed part of $\bigoplus_{\underline I \text{ s.t. } C_{\underline I}=0}\bb{C}((s\inv))$ under the action of $\mu_{2n+2}$, which is generated by the set 
	\begin{equation*}
		\begin{split}
			\Bigl\{\sum_{i=0}^{2n+1}\zeta^i\cdot f^{\underline I} \mid C_{\underline I}=0\Bigr\}.
		\end{split}
	\end{equation*}
If $2\nmid nk$, since $\zeta^{n+1}\cdot f^{\underline I}=(-1)^{nk}f^{\underline I}=-f^{\underline I}$, each sum of the form $\sum_{i=0}^{2n+1}\zeta^i\cdot f^{\underline I} $ is $0$. So $\rr{Soln}_\infty= 0$. If $2\mid n$ (resp. $2\nmid n$ and $2\mid k$), we can check that $\Sigma$ has cardinality $a(k,n+1)$ (resp. $b(k,n+1)$).
\end{proof}

\subsubsection{The dimension of the de Rham cohomologies}
We compute the dimension of the de Rham cohomologies of $\Sym^k\Kl_{n+1}$ and $\Sym^k\widetilde \Kl_{n+1}$. Recall that $q_{d,k}$, $d(k,n+1)$, $a(k,n+1)$, and $b(k,n+1)$ are numbers defined in Lemma~\ref{lem:decomposition}, \eqref{eq:dk}, \eqref{eq:ak}, and \eqref{eq:bk} respectively.
\begin{prop}\label{prop:dimensionSymk}
	The dimensions of $\Hdr{1}(\bb{G}_m,\Sym^k\Kl_{n+1})$ and $\Hdr{1}(\bb{G}_m,\Sym^k\widetilde \Kl_{n+1})$ are
		\begin{equation*}
			\begin{split}
				\frac{1}{n+1}\biggl(\binom{n+k}{n}-d(k,n+1)\biggr) \quad \text{and} \quad \binom{n+k}{n}-d(k,n+1)
			\end{split}
		\end{equation*}
	respectively.
\end{prop}
\begin{proof}
	Let $\mathcal{M}$ be either $\Sym^k\Kl_{n+1}$ or $\Sym^k\widetilde\Kl_{n+1}$. By an analog of the Grothendieck--Ogg--Shafarevich formula \cite[2.9.8.2]{katz1990exponential}, we have
		\begin{equation*}
			\begin{split}
				\chi(\bb{G}_m,\mathcal{M})=-\mathrm{Irr}_\infty(\mathcal{M}).
			\end{split}
		\end{equation*}
	Since $\bb{G}_m$ is affine, we have $\Hdr{2}(\bb{G}_m,\mathcal{M})=0$. By \cite[Prop.\,2.7]{fresan2018hodge}, $\mathcal{M}$ is irreducible, which implies that $\Hdr{0}(\bb{G}_m,\mathcal{M})=0$. So the only non-vanishing cohomology of $\mathcal{M}$ is $\Hdr{1}(\bb{G}_m,\mathcal{M})$. Therefore, the dimension of $\mathrm{H}^1_{\mathrm{dR}}(\bb{G}_m,\mathcal{M})$ is equal to the irregularity of $\mathcal{M}$ at $\infty$ that we computed in Proposition~\ref{prop:local-infinity-Symk}.
\end{proof}

\begin{cor}\label{lem:dim_of_inv+mid}
	The dimensions of  $\mathrm{H}^1_{\mathrm{dR,mid}}(\bb{G}_m,\Sym^k\Kl_{n+1})$ and $\mathrm{H}^1_{\mathrm{dR,mid}}(\bb{G}_m,\Sym^k\widetilde \Kl_{n+1})$ are
		\begin{equation*}
			\begin{split}
				\frac{1}{n+1}\biggl(\binom{n+k}{n}-d(k,n+1)\biggr)-\sum_{d=0}^{\lfloor\frac{nk}{2}\rfloor}q_{d,k}-
		\begin{cases}
			a(k,n+1) & 2\mid n,\\
			0		 & 2\nmid nk,\\
			b(k,n+1) & \text{else},
		\end{cases}
			\end{split}
		\end{equation*}
	and
		\begin{equation*}
			\begin{split}
				\binom{n+k}{n}-d(k,n+1)-\sum_{d=0}^{\lfloor\frac{nk}{2}\rfloor}q_{d,k}-
		\begin{cases}
			d(k,n+1) & 2\mid nk,\\
			0        & \text{else},
		\end{cases}
			\end{split}
		\end{equation*}
	respectively.
\end{cor}
\begin{proof}
	Applying \cite[Cor.\,2.9.8.1]{katz1990exponential} to $j_{\dagger+}\Sym^k\Kl_{n+1}$, where $j\colon \bb{G}_m\hookrightarrow\bb{P}^1$ is the inclusion, we have
		\begin{equation*}
			\begin{split}
				\chi(\bb{P}^1,j_{\dagger+}\Sym^k\Kl_{n+1})=\chi(\bb{G}_m,\Sym^k\Kl_{n+1})+\mathrm{Soln}_0+\mathrm{Soln}_\infty.
			\end{split}
		\end{equation*}
	So we conclude the dimension formula for the middle de Rham cohomology of $\Sym^k\Kl_{n+1}$ by Corollary~\ref{cor::soln0} and Proposition~\ref{prop:soln-infty}. We can similarly compute the dimension formula for the middle de Rham cohomology of $\Sym^k\tilde \Kl_{n+1}$.
\end{proof}
\begin{exe}
	If $n=2$, the dimensions of the de middle de Rham cohomologies $\mathrm{H}^1_{\mathrm{dR,mid}}(\bb{G}_m,\Sym^k\Kl_{3})$ and $\mathrm{H}^1_{\mathrm{dR,mid}}(\bb{G}_m,\Sym^k\widetilde \Kl_{3})$ are 	\begin{equation*}
		\begin{split}
			\frac{k(k+1)}{6}-\frac{1}{3}d(k,3)-1-\Bigl\lfloor\frac{k}{2}\Bigr\rfloor-a(k,3)
		\end{split}
	\end{equation*}
and
	\begin{equation*}
		\begin{split}
			\frac{k(k+1)}{2}-1-\Bigl\lfloor\frac{k}{2}\Bigr\rfloor-2d(k,3)
		\end{split}
	\end{equation*}
respectively, where $d(k,3)=a(k,3)$ is $0$ if $3\nmid k$ and $1$ if $3\mid k$.
\end{exe}

\subsection{Airy connections}\label{sec:Airy-connection}
We recall some basic results from \cite[\S 6]{sabbah2021airy} about Airy connections. Let $n\geq 2$ be an integer. Consider the diagram 
	$$\begin{tikzcd}
			& \bb{A}_x^1\times \bb{A}^1_z \ar[ld,"f"'] \ar[rd,"\pr_z"]&	\\
			\bb{A}^1_x &	& \bb{A}^1_z
	\end{tikzcd}$$
where 
\begin{equation}\label{eq:f-Airy}
	\begin{split}
		f(x,z)=\frac{1}{n+1}x^{n+1}-xz
	\end{split}
\end{equation}
and $\pr_z$ is the projection to $\bb{A}^1_z$. Here, we use the same letter $f$ for both Laurent polynomials in \eqref{eq:f-kloosterman} and \eqref{eq:f-Airy} because they play the same roles in the two parallel sides of Kloosterman and Airy moments. We define the \textit{Airy connection} by
	\begin{equation*}
		\begin{split}
			\rr{Ai}_n:=\cc{H}^0\pr_{z+}f^+\cc{E}^x.
		\end{split}
	\end{equation*}
Equivalently, the Airy connection $\rr{Ai}_n$ is the (negative) Fourier transform of $\cc{E}^{x^{n+1}/n+1}$.

\begin{prop}
	The Airy connections have the following properties.
	\begin{enumerate}
		\item $\rr{Ai}_{n}$ is a free $\mathcal{O}_{\bb{A}^1}$-module of rank $n$, whose connection is given by
		\begin{equation}\label{eq:connection-Airy}
			\mathrm{d}+N\mathrm{d}z+z E\,\mathrm{d}z,
		\end{equation}
		in terms of some basis $\{v_0,v_1,\ldots,v_{n-1}\}$, where $N$ and $E$ are the same matrices in \eqref{eq:connection}. 
		\item The connection $\rr{Ai}_{n}$ has an irregular singularity at $\infty$, with slope $\frac{n+1}{n}$.
		\item We have an isomorphism
			\begin{equation*}
				\begin{split}
					\rr{Ai}_{n}^\vee\simeq \iota_{n}^+\rr{Ai}_{n}
				\end{split}
			\end{equation*}
		where $\rr{Ai}_{n}^{\vee}$ is the dual of $\rr{Ai}_{n}$, and $\iota_{n}$ is the involution $z\mapsto (-1)^{n}z$. 
	\end{enumerate}
\end{prop}  
\begin{proof}
	The first statement is proven in \cite[Lem.\,6.1]{sabbah2021airy}. The second and the third ones are proven in Lemma 6.4 and the paragraph below in \textit{loc. cit}.
\end{proof}

\begin{prop}\label{prop:dimensionSymk-Ai}
	We denote by $\Sym^k\rr{Ai}_n$ the $k$-th symmetric power of $\rr{Ai}_n$. Then
	\begin{enumerate}
		\item $\Sym^k\rr{Ai}_{n}$ has rank $\binom{k+n-1}{n-1}$,
		\item  $\Sym^k\rr{Ai}_{n}$ is irreducible,
		\item If $\gcd(k,n)=1$, the de Rham cohomologies and the middle de Rham cohomologies of $\Sym^k\rr{Ai}_{n}$ coincide, and their dimensions are $\frac{1}{n}\binom{k+n-1}{n-1}$ 
	\end{enumerate}
\end{prop}
\begin{proof}
	The first statement is straightforward because $\rk\,\rr{Ai}_n$ has rank $n$. The second statement is proven in \cite[Lem.\,6.6]{sabbah2021airy}. The last statement is deduced from Corollary 6.9 in \textit{loc. cit}.
\end{proof}

\section{Motives and Exponential mixed Hodge structures}\label{sec:EMHS}

In this section, we quickly recall the construction of the motives $\rr{M}_{n+1}^k(\Kl)$ and $\rr{M}_{n}^k(\rr{Ai})$ as exponential motives, although they are Nori motives and ulterior motives respectively. The de Rham realizations of these motives underlie some exponential mixed Hodge structures, which agree with the mixed Hodge structures and the finite monodromic mixed Hodge structures on $\rr{M}_{n+1}^k(\Kl)_{\rr{dR}}$ and $\rr{M}_{n}^k(\rr{Ai})_{\rr{dR}}$ respectively.

\subsection{Preliminary on Exponential mixed Hodge structures}\label{subsec:emhs}

We recall some basic properties of exponential mixed Hodge structures from \cite[Appx.]{fresan2018hodge}.

Let $X$ be a smooth complex algebraic variety. The category $\mathrm{MHM}(X)$ of \textit{mixed Hodge modules} on $X$ is an abelian category. In particular, if $X$ is a point, the category is nothing but the category of polarized mixed Hodge structures. The bounded derived categories $\mathrm{D}^b(\mathrm{MHM}(X))$ admit the six functor formalism. For more details about mixed Hodge modules, see \cite{Saito1990}.

The category $\rm EMHS$ of \textit{exponential mixed Hodge structures} is defined in \cite{Kontsevich2011} as the full subcategory of $\mathrm{MHM}(\bb{A}^1)$ whose objects $N^{\mathrm{H}}$ have vanishing cohomology on $\bb{A}^1$, i.e., satisfying $\pi_*N^{\mathrm{H}}=0$ for the structure morphism $\pi\colon \bb{A}^1\to \mathrm{Spec}(\bb{C})$. 

Let $j\colon \bb{G}_m\to \bb{A}^1$ be the inclusion and $s\colon\bb{A}^1\times \bb{A}^1\to \bb{A}^1$ the summation map. The functor
\begin{equation}\label{eq:projector}
	N^{\mathrm{H}} \mapsto s_*(N^{\mathrm{H}}\boxtimes j_!\,\mathbb{Q}_{\bb{G}_m}^{\rr{H}}),
\end{equation}
is exact and defines an endofunctor $\Pi\colon \mathrm{MHM}(\bb{A}^1)\to \mathrm{MHM}(\bb{A}^1)$, which is a projector onto $\mathrm{EMHS}$. By abuse of notation, we also denote by $\Pi$ the endofunctor on the category of regular holonomic $\mathscr{D}_{\bb{A}^1}$-modules, defined by sending $N\mapsto s_+(N\boxtimes  j_{0\dagger}\mathcal{O}_{\bb{G}_m})$. In particular, the endofunctor in \eqref{eq:projector} is the lifting of the endofunctor for holonomic $\scr{D}$-modules.

If we view the category of polarizable mixed Hodge structures as the category of mixed Hodge modules at the point $0$, we can embed it into $\rm EMHS$ as a full subcategory via the functor
	\begin{equation}\label{eq:embed-classical-mhs-to-emhs}
		V^{\rr{H}}\mapsto \Pi\bigl(i_{0!}V^{\rr{H}}\bigr).
	\end{equation}

	We define for each object $\Pi(N^{\rr{H}})$ in the category $\mathrm{EMHS}$ a weight filtration $W^{\mathrm{EMHS}}$ on $N^{\mathrm{H}}$ in terms of that on $N^{\mathrm{H}}$, i.e., $W^{\mathrm{EMHS}}\Pi(N^{\mathrm{H}}):=\Pi(W^{\rr{MHM}}_n N^{\rr{H}})$. If there is no confusion, we omit the superscript of $W^{\rr{EMHS}}$ for simplicity.

The \textit{de Rham fiber} functor from $\rm EMHS$ to $\mathrm{Vect}_{\bb{C}}$ is defined by 
\begin{equation}\label{eq:de Rham fiber}
	\Pi(N^{\mathrm{H}})\mapsto\Hdr{1}(\bb{A}^1,\Pi(N)\otimes \mathcal{E}^{\theta}),
\end{equation}
where $N$ is the underlying $\scr{D}$-module of $N^{\rr{H}}$, $\theta$ is the coordinate of $\bb{A}^1$, and $\mathcal{E}^{\theta}$ denotes the rank $1$ connection $(\mathcal{O}_{\bb{A}^1},\mathrm{d}+\mathrm{d}\theta)$. Each de Rham fiber is equipped with a weight filtration $W_{\bullet}$. Moreover, one can associate an \textit{irregular Hodge filtration} $F^{\bullet}_{irr}$ on every de Rham fiber by methods in \cite{Sabbah-2010-laplace,Sabbah-Yu-2015,esnault-sabbah-yu2017e1}. 

For a regular function $f\colon U\to \bb{A}^1$ on a smooth complex quasi-projective variety $U$ of dimension $d$, we define the exponential mixed Hodge structures
\begin{equation*}
	\mathrm{H}^r(U,f)^{\mathrm{H}}:=\Pi\bigl(\mathcal{H}^{r-d}f_*\bb{Q}_U^{\mathrm{H}}\bigr),\quad \mathrm{H}^r_{\rr{c}}(U,f)^{\rr{H}}:=\Pi\bigl(\mathcal{H}^{r-d}f_!\bb{Q}_U^{\mathrm{H}}\bigr)
\end{equation*}
and
	\begin{equation*}
		\begin{split}
			\mathrm{H}^r_{\mathrm{mid}}(U,f)^{\mathrm{H}}:=\Pi\bigl(\im\bigl(\mathcal{H}^{r-d}f_!\bb{Q}_U^{\mathrm{H}}\to \mathcal{H}^{r-d}f_*\bb{Q}_U^{\mathrm{H}}\bigr)\bigr).
		\end{split}
	\end{equation*}
In particular, the de Rham fiber of $\mathrm{H}^r_{?}(U,f)^{\mathrm{H}}$ is $\mathrm{H}^r_{\mathrm{dR},?}(U,f)$. 
\subsection{The irregular Hodge filtration on twisted de Rham cohomology}\label{subsec:irreg-hodge-fil}
Let $U$ be a complex smooth quasi-projective variety and $f$ a regular function on $U$. The irregular Hodge filtration on the de Rham fiber of $\mathrm{H}^r(U,f)^{\mathrm{H}}$ has a geometric interpretation in \cite{yu2012irregular}.

Let $j\colon U\to X$ be a compactification, and $S:=X\backslash U$ be the boundary divisor. The pair $(X,S)$ is called a \textit{good compactification} of the pair $(U,f)$ if $S$ is normal crossing and $f$ extends to a morphism $\bar f\colon X\to \bb{P}^1_{\bb{C}}$. We denote by $P$ the pole divisor of $f$. 

For the twisted de Rham complex $(\Omega(X)^\bullet(*S),\nabla=\mathrm{d}+\mathrm{d}f)$, there is a decreasing filtration $F^\lambda(\nabla):=F^0(\lambda)^{\geq \lceil\lambda\rceil}$ indexed by non-negative real numbers, where $F^0(\lambda)$ is the complex
	\begin{equation*}
		\begin{split}
			\mathcal{O}(\lfloor -\lambda P\rfloor)\xrightarrow{\nabla} \Omega^1(\log S)(\lfloor (1-\lambda)P\rfloor)\to \cdots \to \Omega^p(\log S)(\lfloor (p-\lambda)P)\to \cdots.
		\end{split}
	\end{equation*}
\begin{defn}
The \textit{irregular Hodge filtration of the de Rham cohomology} $\mathrm{H}^i_{\dR}(U,\nabla)$ is defined by
	\begin{equation*}
		\begin{split}
			F^\lambda\mathrm{H}^i_{\mathrm{dR}}(U,\nabla):=\im(\bb{H}^i(X,F^\lambda(\nabla))\to \mathrm{H}^i_{\mathrm{dR}}(U,\nabla)).
		\end{split}
	\end{equation*}
\end{defn}

The filtration is independent of the choice of good compactifications $(X, S)$ \cite[Thm.\,1.7]{yu2012irregular}.

\subsubsection*{Newton polyhedral filtration}
When $U$ is $\bb{G}_m^n$, a regular function $f=\sum_{\alpha}c(\alpha)x^\alpha$ on $U$ is a Laurent polynomial. The \textit{Newton polytope} $\Delta=\Delta(f)$ is the convex hull of $\{0\}\bigcup \mathrm{Supp}(f)$ inside the character lattice $M:=\mathrm{Hom}(U,\bb{C}\cros)$. We say that $f$ is \textit{non-degenerate with respect to $\Delta(f)$} if, for each face $\delta$ of $\Delta(f)$ not containing $0$, the function $f_{\delta}=\sum_{\alpha\in \delta}c(\alpha)x^\alpha$ has no critical points in $U$. 

We denote $M\otimes_{\bb{Z}}\bb{R}$ by $M_{\bb{R}}$. Let $F:=F(\Delta)$ be the normal fan of the Newton polytope $\Delta$ on the dual space $N_{\bb{R}}=\Hom(M_{\bb{R}},\bb{R})$, i.e., each ray of the normal fan corresponds to a facet of $\Delta$, pointing inward with respect to the paring $N_{\bb{R}}\times M_{\bb{R}}\to \bb{R}$. We refine the fan $F$ to $\widetilde F$ such that the corresponding toric variety $X_{\mathrm{tor}}$ is smooth proper and the toric boundary $S:=X_{\tor}\backslash U$ is simple normal crossing. We again denote by $P$ the pole divisor of $f$. Then each ray in $\widetilde F$ (resp. $F$) corresponds to an irreducible component of $S$ (resp. $P$). 

In \cite{yu2012irregular}, the filtration $F^{\lambda}_{\mathrm{NP}}(\nabla)$ is defined similarly to $F^{\lambda}(\nabla)$ by replacing the good compactification $(X,S)$ with $(X_{\tor},S)$. The newton monomial filtration $F^\lambda_{\mathrm{NP}}(\mathrm{H}^i_{\dR}(U,\nabla))$ is defined by		
	\begin{equation}\label{defn:newton-monomial-fil}
		F^\lambda_{\rr{NP}}\mathrm{H}^i_{\mathrm{dR}}(U,\nabla):=\im(\bb{H}^i(X,F_{\rr{NP}}^\lambda(\nabla))\to \mathrm{H}^i_{\mathrm{dR}}(U,\nabla)).
	\end{equation}
By \cite[Thm 1.4]{Adolphson1997twisted} and \cite[Thm 4.6]{yu2012irregular}, when $f$ is non-degenerate, the irregular Hodge filtration $F^\bullet_{irr}$ and the Newton monomial filtration $F^\bullet_{\mathrm{NP}}$ on $\mathrm{H}^i_{\dR}(U,\nabla)$ agree.

\subsection{Motives and exponential mixed Hodge structures attached to Kloosterman connections}\label{sec:EMHS-Kl}
Let $\{x_{i,j}|1\leq i\leq k, 1\leq j\leq n\}$ and $z$ be the coordinates of $\bb{G}_{m}^{nk+1}$, and $t$ be the coordinate of the source of the map 
	$[n+1]\colon \bb{G}_{m,t}\to \bb{G}_{m,z},\ a\mapsto a^{n+1}$. 
The group $S_k$ acts on $\bb{G}_{m,(x_{i,j},z)}^{nk+1}$ by fixing $z$ and $\sigma \cdot x_{i,j}=x_{\sigma(i),j}$, and the group $S_k\times \mu_{n+1}$ acts on $\bb{G}_{m,(x_{i,j},t)}^{nk+1}$ by 
	\begin{equation*}
		\begin{split}
			(\sigma\times \mu)\cdot x_{i,j}=x_{\sigma(i),j}\quad  \text{and} \quad (\sigma\times \mu)\cdot t=\mu\cdot t.
		\end{split}
	\end{equation*}
We denote by $\chi_n$ the character $S_{k}\times \mu_{n+1}\twoheadrightarrow S_k \xrightarrow{\rr{sign}^n}\{\pm 1\}$. Let $f_{k}\colon \bb{G}_{m,(x_{i,j},z)}^{nk+1}\to \bb{A}^1$ (resp. $\widetilde f_{k}\colon \bb{G}_{m,(x_{i,j},t)}^{nk+1}\to \bb{A}^1$) be the Laurent polynomial 
	\begin{equation}\label{eq:fk-Kloosterman}
		\begin{split}
			\sum_{i=1}^k\biggl(\sum_{j=1}^n x_{i,j}+\frac{z}{\prod_{j=1}^nx_{i,j}}\biggr)\quad  \biggl(\text{ resp. } \sum_{i=1}^k\biggl(\sum_{j=1}^n x_{i,j}+\frac{t^{n+1}}{\prod_{j=1}^n x_{i,j}}\biggr)\biggr).
		\end{split}
	\end{equation}

\begin{defn}\label{def:Kloosterman-motive}
	The motive $\rr{M}_{n+1}^k(\rr{Kl})$ attached to Kloosterman moments is defined as the exponential motive 
		\begin{equation*}
			\begin{split}
				\im\Bigl(\rr{H}^{nk+1}_{\rr{c}}(\bb{G}_m^{nk+1},f_k)^{S_k,\chi_n}\to \rr{H}^{nk+1}(\bb{G}_m^{nk+1},f_k)^{S_k,\chi_n}\Bigr),
			\end{split}
		\end{equation*}
	in the sense of \cite[\S4.2,\,\S4.9]{fresanexponential}, where the exponent $(S_k,\chi_n)$ means taking the $\chi_n$-isotypic components.
\end{defn}

\begin{rek}\label{rek:mot-classical}
	This exponential motive is isomorphic to a Nori motive 
		\begin{equation*}
			\begin{split}
				\gr^W_{nk+1}\rr{H}^{nk-1}_{\rr{c}}(\cc{K})^{S_k\times \mu_{n+1},\chi_n}(-1)
			\end{split}
		\end{equation*}
	by an argument similar to that in \cite[Thm\,3.8]{fresan2018hodge}. Here $\cc{K}\subset \bb{G}^{nk}_{m,(y_{i,j})}$ is the hypersurface defined by $g_k(y_{i,j})=\sum_{i=1}^k\bigl(\sum_{j=1}^n y_{i,j}+\frac{1}{\prod_{j=1}^n y_{i,j}}\bigr)$, and the action of $S_k\times \mu_{n+1}$ on the motive $\rr{H}^{nk-1}_{\rr{c}}(\cc{K})$ is induced by functoriality from that on $\cc{K}\subset \mathbb{G}_{m,y_{i,j}}^{nk}$, defined by\footnote{This action is related to the action of $S_k\times \mu_{n+1}$ on $\mathbb{G}_{m, (x_{i,j},t)}^{nk+1}$ by the relation $x_{i,j}=t\cdot y_{i,j}$.} $(\sigma\times \mu )\cdot y_{i,j}:=\mu\inv y_{\sigma(i),j}$, and hence is compatible with the weight filtration.
\end{rek}

The name of this motive is justified by \cite[Thm.\,3.12,\,Thm.\,5.8,\,and\,Thm.\,5.17]{fresan2018hodge}, because its $\ell$-adic realizations encode the Kloosterman moments. Besides, the de Rham realization of $\rr{M}^k_{n+1}(\Kl)$ is identified with the middle de Rham cohomology of $\Sym^k\Kl_{n+1}$ \cite[Cor.\,2.15]{fresan2018hodge}. More precisely, for $?\in\{\emptyset, \mathrm{c},\rr{mid}\}$, we have
		\begin{equation*}
			\mathrm{H}^1_{\mathrm{dR},?}(\bb{G}_m,\Sym^k\Kl_{n+1})\simeq \mathrm{H}^{nk+1}_{\mathrm{dR},?}(\bb{G}_m^{nk+1},f_{k})^{S_k,\chi_n}\simeq \mathrm{H}^{nk+1}_{\mathrm{dR},?}(\bb{G}_m^{nk+1},\widetilde f_{k})^{S_k\times\mu_{n+1},\chi_n},
		\end{equation*}
		and 
		\begin{equation*}
			\mathrm{H}^1_{\mathrm{dR},?}(\bb{G}_m,\Sym^k\widetilde\Kl_{n+1})\simeq \mathrm{H}^{nk+1}_{\mathrm{dR},?}(\bb{G}_m^{nk+1},\widetilde f_{k})^{S_k,\chi_n}.
		\end{equation*}
	In particular, the $\mu_{n+1}$-invariants of $\mathrm{H}^1_{\mathrm{dR},?}(\bb{G}_m,\Sym^k\widetilde\Kl_{n+1})$ is $\mathrm{H}^1_{\mathrm{dR},?}(\bb{G}_m,\Sym^k\Kl_{n+1})$.

\begin{defn}\label{eq:EMHS-Kl}
	For $?\in \{\emptyset, \rm c, mid\}$, we define the exponential mixed Hodge structures attached to Kloosterman moments as
		\begin{equation*}
			\begin{split}
				\mathrm{H}^1_?(\bb{G}_{m},\Sym^k\Kl_{n+1})^{\rr{H}}:=(\mathrm{H}^{nk+1}_?(\bb{G}_m^{nk+1},\widetilde f_{k})^{\rr{H}})^{S_k\times \mu_{n+1},\chi_n} 
			\end{split}
		\end{equation*}
	and 
		\begin{equation*}
			\begin{split}
				\mathrm{H}^1_?(\bb{G}_m,\Sym^k\widetilde\Kl_{n+1})^{\rr{H}}:=(\mathrm{H}^{nk+1}_?(\bb{G}_m^{nk+1},\widetilde f_{k})^{\rr{H}})^{S_k,\chi_n}.
			\end{split}
		\end{equation*}
\end{defn}
The de Rham fibers \eqref{eq:de Rham fiber} of the above exponential mixed Hodge structures are the de Rham cohomologies
	\begin{equation*}
		\begin{split}
			\mathrm{H}^1_{\mathrm{dR},?}(\bb{G}_m,\Sym^k\Kl_{n+1}) \quad \text{and} \quad  \mathrm{H}^1_{\mathrm{dR},?}(\bb{G}_m,\Sym^k\widetilde\Kl_{n+1})
		\end{split}
	\end{equation*}
respectively. By \cite[Thm.\,3.2 and Thm.\,3.8]{fresan2018hodge}, these exponential mixed Hodge structures have the following properties:
\begin{itemize}
	\item For $?\in \{\emptyset, \rm c, mid\}$, the exponential mixed Hodge structures $\mathrm{H}^1_?(\bb{G}_{m},\Sym^k\Kl_{n+1})^{\rr{H}}$ and $\mathrm{H}^1_?(\bb{G}_{m},\Sym^k\widetilde\Kl_{n+1})^{\rr{H}}$ are isomorphic to (classical) mixed Hodge structures, i.e., lie in the essential image of \eqref{eq:embed-classical-mhs-to-emhs}. They are (mixed) of weights $\geq nk+1$, $\leq nk+1$, and $nk+1$, respectively. 
	\item We have isomorphisms of pure Hodge structures
	\begin{equation*}
	\begin{split}
		\mathrm{H}^1_{\mathrm{mid}}(\bb{G}_m,\Sym^k\Kl_{n+1})^{\rr{H}}\simeq \mathrm{gr}^W_{nk+1}(\mathrm{H}^{nk-1}_{\mathrm{c}}(\mathcal{K})^{\rr{H}})^{S_k\times \mu_{n+1},\chi_n}(-1),
	\end{split}
	\end{equation*}
	and 
	\begin{equation*}
	\begin{split}
		\mathrm{H}^1_{\mathrm{mid}}(\bb{G}_m,\Sym^k\widetilde{\Kl}_{n+1})^{\rr{H}}\simeq \mathrm{gr}^W_{nk+1}(\mathrm{H}^{nk-1}_{\mathrm{c}}(\mathcal{K})^{\rr{H}})^{S_k,\chi_n}(-1).
	\end{split}
	\end{equation*} 
	\item The natural forget support morphisms induce isomorphisms 
		\begin{equation*}
			\begin{split}
				\mathrm{gr}^W_{nk+1} \mathrm{H}^1_{\mathrm{c}}(\bb{G}_m,\Sym^k\Kl_{n+1})^{\rr{H}}\to \mathrm{gr}^W_{nk+1}\mathrm{H}^1(\bb{G}_m,\Sym^k\Kl_{n+1})^{\rr{H}},
			\end{split}
		\end{equation*}
	and 
		\begin{equation*}
			\begin{split}
				\mathrm{gr}^W_{nk+1} \mathrm{H}^1_{\mathrm{c}}(\bb{G}_m,\Sym^k\widetilde\Kl_{n+1})^{\rr{H}}\to \mathrm{gr}^W_{nk+1}\mathrm{H}^1(\bb{G}_m,\Sym^k\widetilde\Kl_{n+1})^{\rr{H}}.
			\end{split}
		\end{equation*}
	\item There are $(-1)^{nk+1}$-symmetric perfect pairings
		\begin{equation*}
			\begin{split}
				\mathrm{H}^1_{\mathrm{mid}}(\bb{G}_m,\Sym^k{\Kl}_{n+1})^{\rr{H}}\times \mathrm{H}^1_{\mathrm{mid}}(\bb{G}_m,\Sym^k{\Kl}_{n+1})^{\rr{H}}\to \bb{Q}(-nk-1)
			\end{split}
		\end{equation*}
	and 
		\begin{equation*}
			\begin{split}
				\mathrm{H}^1_{\mathrm{mid}}(\bb{G}_m,\Sym^k{\tilde \Kl}_{n+1})^{\rr{H}}\times \mathrm{H}^1_{\mathrm{mid}}(\bb{G}_m,\Sym^k{\tilde\Kl}_{n+1})^{\rr{H}}\to \bb{Q}(-nk-1)
			\end{split}
		\end{equation*}
\end{itemize}

\subsection{Motives and exponential mixed Hodge structures attached to Airy connections}

Let $\{x_{i}|1\leq i\leq k\}$ and $z$ be the coordinates of $\bb{G}_{m}^{k+1}$ and $t$ the coordinate of the source of the map 
	$[n]\colon \bb{A}^1_{t}\to \bb{A}^1_{z}$. 
The group $S_k$ acts on $\bb{G}_{m}^{k+1}$ by 
	\begin{equation*}
		\begin{split}
			\sigma\cdot x_{i}=x_{\sigma(i)} \quad  \text{and} \quad \sigma\cdot z= z.
		\end{split}
	\end{equation*}
We denote by $\chi$ the character $ S_k \xrightarrow{\rr{sign}}\{\pm 1\}$. Let $ f_{k}\colon \bb{A}_{m,(x_{i},z)}^{k+1}\to \bb{A}^1$ be the Laurent polynomial 
	\begin{equation}\label{eq:fk-Airy}
		\begin{split}
			\sum_i\Bigl(\frac{1}{n+1}x_{i}^{n+1}-zx_i\Bigr).
		\end{split}
	\end{equation}
Here, we use the same notation in \eqref{eq:fk-Kloosterman} for this Laurent polynomial because it plays the same role as $f_k$ in the side of Kloosterman moments in Section~\ref{sec:EMHS-Kl}.

\begin{defn}\label{def:Airy-motive}
	The motive $\rr{M}_n^k(\rr{Ai})$ attached to Airy moments is defined as the exponential motive 
		\begin{equation*}
			\begin{split}
				\im \Bigl(\rr{H}^{k+1}_{\mathrm{c}}(\bb{A}^{k+1},f_k)^{S_k,\chi}\to \rr{H}^{k+1}(\bb{A}^{k+1},f_k)^{S_k,\chi}\Bigr),
			\end{split}
		\end{equation*}
	where the exponent $(S_k,\chi)$ means taking the $\chi$-isotypic component.
\end{defn}

The de Rham realization of $\rr{M}_n^k(\rr{Ai})$ is identified with the middle de Rham cohomology of $\Sym^k\rr{Ai}_n$. More precisely, for $?\in\{\emptyset, \mathrm{c}, \rr{mid}\}$, we have
		\begin{equation*}
			\mathrm{H}^1_{\mathrm{dR},?}(\bb{A}^1,\Sym^k\rr{Ai}_{n})\simeq \mathrm{H}^{nk+1}_{\mathrm{dR},?}(\bb{A}^{k+1},  f_{k})^{S_k ,\chi}.
		\end{equation*}

\begin{defn}\label{eq:EMHS-Ai}
	For $?\in \{\emptyset, \rm c, mid\}$, we define the exponential mixed Hodge structures attached to Airy moments
		\begin{equation*}
			\begin{split}
				\mathrm{H}^1_?(\bb{A}^1,\Sym^k\rr{Ai}_{n})^{\rr{H}}:=(\mathrm{H}^{k+1}_?(\bb{A}^{k+1},f_{k})^{\rr{H}})^{S_k,\chi}.
			\end{split}
		\end{equation*}
\end{defn}

The de Rham fibers \eqref{eq:de Rham fiber} of the above exponential mixed Hodge structures are de Rham cohomologies
	\begin{equation*}
		\begin{split}
			\mathrm{H}^1_{\mathrm{dR},?}(\bb{A}^1,\Sym^k\rr{Ai}_{n})
		\end{split}
	\end{equation*}
for $?\in \{\emptyset,\rr{c, mid}\}$. By \cite[Thm.\,6.17]{sabbah2021airy},  these exponential mixed Hodge structures have the following properties: 
\begin{itemize}
	\item For $?\in \{\emptyset, \rm c, mid\}$, the exponential mixed Hodge structures $\mathrm{H}^1_?(\bb{A}^1,\Sym^k\rr{Ai}_{n})^{\rr{H}}$ are (mixed) of weights $\geq k+1$, $\leq k+1$ and $k+1$ respectively.
	\item The natural forget support morphism induces an isomorphism
		\begin{equation*}
			\begin{split}
				\mathrm{gr}^W_{k+1} \mathrm{H}^1_{\mathrm{c}}(\bb{A}^1,\Sym^k\rr{Ai}_{n})^{\rr{H}}\to \mathrm{gr}^W_{k+1}\mathrm{H}^1(\bb{A}^1,\Sym^k\rr{Ai}_{n})^{\rr{H}},
			\end{split}
		\end{equation*}
	and $\mathrm{H}^1_{\mathrm{mid}}(\bb{A}^1,\Sym^k\rr{Ai}_{n})^{\rr{H}}$ is equal to its image.
	\item There is a $(-1)^{k+1}$-symmetric perfect pairing
		\begin{equation*}
			\begin{split}
				\mathrm{H}^1_{\rr{mid}}(\bb{A}^1,\Sym^k\rr{Ai}_{n})^{\rr{H}}\times \mathrm{H}^1_{\rr{mid}}(\bb{A}^1,\Sym^k\rr{Ai}_{n})^{\rr{H}}\to \bb{Q}(-k-1).
			\end{split}
		\end{equation*}
\end{itemize}

\section{Cohomology basis}\label{sec:cohomology}
In this section, we emphasize some bases of the de Rham cohomologies and the middle de Rham cohomologies of connections $\Sym^k\Kl_{n+1}$ and $\Sym^k\rr{Ai}_n$. The main results are Theorems~\ref{thm:cohomology_classes}, \ref{thm:mid-cohomology_classes}, \ref{thm:cohomology_classes-kl3}, \ref{thm:mid-cohomology_classes-kl3}, and~\ref{thm:cohomology_classes-Ai}.

\subsection{Counting results}
	We prepare some counting results here, which will be useful in selecting bases of the de Rham cohomologies of Kloosterman and Airy connections.
\subsubsection{Some generating series}
Let $Q_k(t)$ be the rational function 
	\begin{equation*}
		\begin{split}
			\frac{(1-t^{n+1})\cdots (1-t^{n+k})}{(1-t^2)\cdots (1-t^k)}
		\end{split}
	\end{equation*}
in Lemma~\ref{lem:decomposition}, and we expand $Q_k(t)$ as $\sum_{d\geq 0}q_{d,k}t^d$ for $q_{d,k}\in \bb{Q}$.

\begin{lem}
	The coefficients $q_{d,k}$ satisfies the following properties:
	\begin{enumerate}
		\item $q_{d,k}=0$ if $d > nk+1$, 
		\item $q_{d,k}=-q_{nk+1-d,k}$ if $0\leq d\leq nk+1$.
	\end{enumerate}
	In other words, the rational function $Q_k(t)$ is a polynomial of degree at most $nk+1$, and its coefficients are antisymmetric with respect to the degree.
\end{lem}
\begin{proof}
	We first prove that the rational function $Q_k(t)$ is a polynomial. We decompose both the nominator and the denominator of $Q_k(t)$ as products of the form
    \begin{equation*}
        \prod_{\ell\geq 2}\prod_{\zeta^\ell=1 \text{ primitive }}(1-\zeta\cdot t)^{N_{\zeta}}\quad \text{and} \quad \prod_{\ell\geq 2}\prod_{\zeta^\ell=1 \text{ primitive }}(1-\zeta\cdot t)^{D_{\zeta}}
    \end{equation*}
    respectively, where 
    \begin{equation*}
		N_{\zeta_\ell}=\#\{i\mid 1\leq i\leq k,\ \ell\mid n+i\} \quad \text{and} \quad D_{\zeta_\ell}=\#\{i\mid 1\leq i\leq k,\ \ell\mid i\}.
	\end{equation*} 
 Noticing that $\binom{n+k}{n}$ is an integer, we deduce that $N_\zeta-D_\zeta\geq 0$ for any $\ell\geq 2$ and any primitive $\ell$-th roots of unity $\zeta$. Hence, $Q_k(t)$ is a polynomial.

	Notice that 
	\begin{equation*}
		\begin{split}
			t^{nk+1}Q_k(t\inv)
			&=t^{nk+1}\frac{(1-t^{-n-1})\cdots (1-t^{-n-k})}{(1-t^{-2})\cdots (1-t^{-k})}\\
			&=-\frac{(1-t^{n+1})\cdots (1-t^{n+k})}{(1-t^{2})\cdots (1-t^{k})}=-Q_k(t).
		\end{split}
	\end{equation*}
	So 
		\begin{equation*}
			\begin{split}
				\sum_{d\leq  nk+1}q_{nk+1-d,k}t^d=-\sum_{d\geq 0}q_{d,k}t^d.
			\end{split}
		\end{equation*}
	By comparing the coefficients of both sides, we deduce the lemma.
\end{proof}

Let $n$ be a natural number. We denote by $N_{d,k}$ the cardinalities of the sets
	\begin{equation}\label{eq:N_{d,k}}
		\{ (a,I_0,I_1,\ldots,I_n)\in \bb{N}^{n+2}\mid |\underline I|=k,\ (n+1)a+I_1+2I_2+\dots+nI_n=d\}
	\end{equation}
for $d,k\geq 0$, and we set $n_{d,k}=N_{d,k}-N_{d-1,k}$.
\begin{prop}\label{prop:counting}
	Assume that $\gcd(n+1,k)=1$. Then
	\begin{enumerate}
		\item $n_{d,k}=0$ if $d>nk-n-k+1$.
		\item $n_{d,k}-n_{nk-n-k+1-d,k}=0$.
		\item $n_{d,k}-n_{nk+1-d,k}=q_{d,k}$.
	\end{enumerate}
\end{prop}
\begin{proof}
Let $H(t,x)$ be the generating series 
	$\sum_{d,k\geq 0}N_{d,k}\,t^d\cdot x^k$ 
of $N_{d,k}$. It is the formal power series expansion of the rational function
	\begin{equation*}
		\begin{split}
			\frac{1}{(1-t^{n+1})(1-x)(1-t\cdot x)\cdots (1-t^n\cdot x) }.
		\end{split}
	\end{equation*}
Then, the generating series of $n_{d,k}$ is the formal power series expansion of the rational function
	\begin{equation}\label{eq:h(t,x)}
		h(t,x):=(1-t)H(t,x)=\frac{1-t}{(1-t^{n+1})(1-x)(1-t\cdot x)\cdots (1-t^n\cdot x) }.
	\end{equation}
We write $h(t,x)=\sum_{k\geq 0} h_k(t)x^k$, where each $h_k(t)=\sum_{d\geq 0}n_{d,k}t^d\in \mathbb{Q}[[t]]$. Notice that $h_k(t)$ can be considered as a holomorphic function for small $|t|$.

\begin{lem}
	The formal power series $h_k(t)$ are polynomials in $t$ if $\gcd(k,n+1)=1$.
\end{lem}
\begin{proof}
	Let $r(t,x):=(1-t^{n+1})\cdot H(t,x)$ be the rational function 
		\begin{equation}\label{eq:r(x,t)}
			\frac{1}{(1-x)(1-t\cdot x)\cdots (1-t^n\cdot x)}.
		\end{equation}
	We can expand $r(t,x)$ as a formal power series 
		\begin{equation*}
			\begin{split}
				\prod_{i=0}^n \Bigl(\sum_{j\geq 0} (t^i\cdot x)^j\Bigr)=\sum_{d,k\geq 0}r_{d,k}t^dx^k=\sum_{k\geq 0} r_k(t)x^k,
			\end{split}
		\end{equation*} 
	which converges absolutely when $|t|\leq 2$ and $|x|\leq 2^{-n-1}$. Observe that $r_k(t)$ is the polynomial $h_k(t)\cdot (1+t+\cdots+t^n)$. To prove that $h_k(t)$ are polynomials when $\gcd(k,n+1)=1$, it suffices to show that $1+t+\cdots+t^n\mid r_k(t)$.

	Observe that $r(t,x)$ and $r(t,t\cdot x)=\sum_{d,k\geq 0}r_{d,k}t^{d+k}x^k$ both converge absolutely when $|t|\leq 2$ and $|x|\leq 2^{-n-2}$. Using the expression \eqref{eq:r(x,t)}, we have $r(\zeta_{n+1},x)=r(\zeta_{n+1},\zeta_{n+1}\cdot x)$ for each $(n+1)$-th root of unity $\zeta_{n+1}\neq 1$. By comparing the coefficients of $x^k$ on both sides of $r(\zeta_{n+1},x)=r(\zeta_{n+1},\zeta_{n+1} \cdot x)$, we conclude that $r_k(\zeta_{n+1})=\zeta_{n+1}^k\cdot r_k(\zeta_{n+1})$. When $\gcd(k,n+1)=1$, we deduce that $\zeta_{n+1}$ is a root of $r_k(t)$. Since this holds for all nontrivial $(n+1)$-th roots of unity, we conclude that $1+t+\cdots+t^n\mid r_k(t)$. Therefore, $h_k(t)$ is a polynomial when $\gcd(k,n+1)=1$.
\end{proof}

By a direct computation, we have 
\begin{equation*}
	\begin{split}
		\tilde h_1(t,x):
		=\sum_{k\geq 0} t^{nk-n}h_k(t\inv)x^k
		&=t^{-n}\cdot \sum_{k\geq 0}h_{k}(t\inv)(t^{n}\cdot x)^k\\
		&=t^{-n}\cdot h(t\inv, t^{n}\cdot x)=h(t,x).
	\end{split}
\end{equation*}
Comparing the coefficients of $x^k$ in the above identity, we have $t^{nk-n}h_k(t\inv)=h_k(t).$ When $\gcd(k,n+1)=1$, since $h_k(t)$ is a polynomial, we deduce that $n_{d,k}=0$ if $d\geq nk-n-k+1$. We verified the first and the second statements.

As for the last statement, we consider 
\begin{equation}\label{eq:powerseries}
	\begin{split}
		\tilde h_2(t,x):&=\sum_{k\geq 0} t^{nk+1}h_k(t\inv)x^k
		=t\cdot \sum_{k\geq 0}h_{k}(t\inv)(t^n\cdot x)^k=t\cdot h(t\inv, t^n\cdot x).
	\end{split}
\end{equation}
By a direct calculation, the difference $h(t,x)-\tilde h_2(t,x)$ is the formal power series expansion of 
\begin{equation*}
	\begin{split}
		\frac{1-t}{(1-x)(1-t\cdot x)\cdots (1-t^n\cdot x) }.
	\end{split}
\end{equation*}

\begin{lem}[{\cite[\S3]{fu2006trivial}}]
	Let $Q_k(t)$ be as above, then
		\begin{equation*}
			\begin{split}
				\frac{1-t}{(1-x)(1-t\cdot x)\cdots (1-t^n\cdot x) }=\sum_{k\geq 0}Q_k(t)x^k.
			\end{split}
		\end{equation*}
\end{lem}

By the lemma, we get 
	\begin{equation*}
		\begin{split}
			\sum_{d\geq 0} (n_{d,k}-n_{nk+1-d,k})t^d=\sum_{d\geq 0}q_{d,k}t^d
		\end{split}
	\end{equation*}
if we look at the coefficients of $x^k$ in \eqref{eq:powerseries}. Then, we deduce the last statement of the proposition by comparing the coefficients of $t^d$ above.
\end{proof}

\begin{exe}\label{exe:n=2-counting}
	Assume that $n=2$ and $k$ is any positive integer. The numbers $q_{d,k}$ and $n_{d,k}$ can be made explicit when $d\leq k$. Under the assumption $\gcd(k,3)=1$, we get the rest $q_{d,k}$ and $n_{d,k}$ when $d\geq k+1$ by the symmetric property.
	
	The numbers $q_{d,k}$ when $d\leq k$ are 
	$\begin{cases}
		1 & 2\mid d\\
		0 & 2\nmid d
	\end{cases}$
	because we have
		\begin{equation*}
			\begin{split}
				Q_k(t)\equiv \frac{1}{1-t^2}\, (\rr{mod}\, t^{k+1})
			\end{split}
		\end{equation*}
	by Lemma~\ref{lem:decomposition}.

	As for $n_{d,k}$, by the definition of $N_{d,k}$, the number $N_{d,k}-N_{d-3,k}$ is equal to the cardinality of the set 
		\begin{equation*}
			\begin{split}
				\{(0, I_0,I_1,I_2)\in \bb{N}^4\mid I_0+I_1+I_2=k,\, I_1+2I_2=d\},
			\end{split}
		\end{equation*}
	which is 	
		$\#\{i_0\mid \max\{0,k-d\}\leq i_0\leq k-\lceil\tfrac{d}{2}\rceil\}.$
	It follows that 
	\begin{equation}\label{eq:recursive-1}
		\begin{split}
			N_{d,k}-N_{d-3,k}
			=\begin{cases}
				\lfloor\tfrac{d}{2}\rfloor+1 & d\leq k,\\
				k-\lceil\tfrac{d}{2}\rceil+1  & k+1\leq d\leq 2k,\\
				0  &2k+1\leq d.
			\end{cases}
		\end{split}
	\end{equation}
	If $d\geq 6$, to compute $n_{d,k}=N_{d,k}-N_{d-1,k}$, we can write $n_{d,k}-n_{d-6,k}$ as 
		\begin{equation*}
			\begin{split}
				(N_{d,k}-N_{d-3,k}+N_{d-3,k}-N_{d-6,k})
			-(N_{d-1,k}-N_{d-4,k}+N_{d-4,k}-N_{d-7,k}).
			\end{split}
		\end{equation*}
	Using \eqref{eq:recursive-1}, we deduce that
	\begin{equation*}\label{eq:recursive-2}
		\begin{split}
			n_{d,k}-n_{d-6,k}= \begin{cases}
					1 & d\leq k,\\
					0 & k+1\leq d\leq k+3,\\
					-1& k+4\leq d\leq 2k+2 \text{ or }d=2k+4,\\
					0& d=2k+3\text{ or }2k+5\leq d.
				  \end{cases}
		\end{split}
	\end{equation*} 
If $d\leq \min(5,k)$, we check directly that 
	$n_{d,k}=\#W_{d,k}=\begin{cases}
		1 & d\neq 1\\
		0		& d=1
	\end{cases}$. 
It follows that 
	$n_{d,k}=\begin{cases}
		\lfloor\frac{d}{6}\rfloor+1 & 6\nmid d-1\\
		\lfloor\frac{d}{6}\rfloor		& 6\mid d-1
	\end{cases}$
if $d\leq k$.
\end{exe}

\subsubsection{Some results in linear algebra}\label{sec:4.1.2}
Let $V_k=\Sym^kV$ be the $k$-th symmetric power of the standard representation $V=\mathbb{C}^{n+1}$ of $\SL_{n+1}$, $N$ and $E$ the matrices from \eqref{eq:connection}, and $N_k=\Sym^kN$ and $E_k=\Sym^kE$ the induced endomorphisms on $V_k$. Let $t$ and $z$ be two variables such that $z=t^{n+1}$. We denote by $G^+$ the vector space $V_k[z]:=V_k\otimes_{\bb{C}}\bb{C}[z]$, on which the endomorphism $N_k+z E_k$ acts. Similarly, we denote by $\tilde G^+=G^+\otimes_{\bb{C}[z]}\bb{C}[t]$, on which $N_k+t^{n+1}E_k$ acts.

We choose the standard basis $\{v_0,v_1\ldots,v_n\}$ of $V$. Then $\{v^{\underline I}\mid |\underline I|=k\}$ is a basis for $V_k$. Let $\deg$ be the grading on $G^+$ defined by
	\begin{equation}\label{eq:grading}
		\deg(z^av^{\underline I}):=(n+1)a+\sum_{i=0}^n i\cdot I_i.
	\end{equation}
Then $N_k+z E_k$ is homogeneous of degree $1$. Similarly, there is a grading on $\tilde{G}^+$ defined by
	\begin{equation}\label{eq:grading-t}
		\deg(t^av^{\underline I}):=a+\sum_{i=0}^n i\cdot I_i.
	\end{equation}

Let $\zeta$ be a primitive $(n+1)$-th root of unity. The elements 		
	$f_i=\sum_{j=0}^n{ \zeta^{i(n-j)}t^{n-j}v_j}$
are eigenvectors of $N_k+t^{n+1}E_k$ of the eigenvalues $t\cdot \zeta^i$ for $0\leq i\leq n$. Moreover, the set 
$\{f^{\underline{I}}:=\prod_{j=0}^n f_j^{I_j}\mid \underline{I}\in \mathbb{N}^{n+1},|\underline{I}|=k\}$
is as a basis for $\tilde{G}^+$ as a $\mathbb{C}[t ]$-module. In particular, we have 
	\begin{equation*}
		\begin{split}
			(N_k+t^{n+1}E_k)f^{\underline I}=t\cdot C_{\underline I}f^{\underline I},
		\end{split}
	\end{equation*}
where $C_{  \underline I}$ is the sum $\sum_{j=0}^n \zeta^j\cdot I_j$.

\begin{lem}\label{lem:kernel}
	The kernel of $N_k+t^{n+1}E_k$ is a free $\bb{C}[t]$-module generated by the set $\{f^{\underline I}\mid C_{\underline I}=0\}$. In particular, the maps $N_k+zE_k$ and $N_k+t^{n+1}E_k$ are injective if $\gcd(k,n+1)=1$.
\end{lem}
\begin{proof}
	Let $\sum_{\underline{i}}g_{\underline{I}}(t)f^{\underline{I}}$ be an element in $\ker(N_k+t^{n+1}E_k)$, then 
		\begin{equation*}
			\begin{split}
				0=(N_k+t^{n+1}E_k)\biggl(\sum_{\underline{I}}g_{\underline{I}}(t)f^{\underline{I}}\biggr)
				=\sum_{\underline{I}} t\cdot C_{\underline I}\cdot g_{\underline{I}}(t)f^{\underline{I}}.
			\end{split}
		\end{equation*}
	Since $f^{\underline{I}}$ are linearly independent, we conclude that the coefficients $g_{\underline I}$ are $0$ as long as  $C_{  \underline I}\neq 0$. So $\ker(N_k+t^{n+1}E_k)=\oplus_{{\underline I} \text{ s.t. }C_{\underline I=0}}f^{\underline I}\,\bb{C}[t]\subset \tilde{G}^+$.

	By the commutative diagram 
		$$\begin{tikzcd}
			G^+\ar[rr,"N_k+zE_k"]\ar[d] &&G^+\ar[d]\\
			\tilde{G}^+\ar[rr,"N_k+t^{n+1}E_k"] &&\tilde{G}^+
		\end{tikzcd}$$
	we observe that $\ker( N_k+zE_k)=\ker( N_k+t^{n+1}E_k)\cap G^+$. If $\gcd(k,n+1)=1$, there is no $\underline I$ such that $C_{\underline I}=0$. Hence, $N_k+t^{n+1}E_k$ and $N_k+zE_k$ are both injective.
\end{proof}

Now, we determine the cokernel of $ N_k+zE_k$.

\begin{prop}\label{prop:cokernel}
	If $\gcd(k,n+1)=1$, there exist finite subsets $W_{d,k}\subset G^+$ for $0\leq d\leq nk+1$ such that 
	\begin{enumerate}[label=\arabic*., ref=\ref{prop:cokernel}.\theenumi]
		\item \label{label:cokernel-1}each element in $W_{d,k}$ has degree $d$,
	\item \label{label:cokernel-3}the generating series $\sum_{d,k\geq 0} \#W_{d,k}t^dx^k$ is the formal power series expansion of the rational function $h(t,x)$ in \eqref{eq:h(t,x)},
	\item \label{label:cokernel-4} $\coker(N_k+zE_k)=\rr{span}(W_{d,k}\mid 0\leq d\leq nk+1)$,
	\item $\#W_{d,k}-\#W_{nk+1-d,k}=q_{d,k}$ for all $0\leq d \leq nk+1$,
	\item $\#W_{d,k}-\#W_{nk-n-k+1-d,k}=0$ for all $0\leq d\leq nk-n-k+1$.
	\end{enumerate}
\end{prop}
\begin{proof}
	We can decompose $G^+=\gr\, G^+$ as $\oplus_{d\geq 0} G_{d,k}$ where $G_{d,k}$ is generated by elements of degree $d$, i.e., linear combinations of $z^a v^{I}$ such that $(n+1)a+\sum_{i=0}^n i\cdot I_i=d$. Since $ N_k+zE_k$ is homogeneous of degree $1$ with respect to the degree \eqref{eq:grading}, the kernel and cokernel of $ N_k+zE_k$ can be decomposed as 
		\begin{equation*}
			\begin{split}
				\ker( N_k+zE_k)= \oplus_d \ker( N_k+zE_k)\cap G_{d,k}
			\end{split}
		\end{equation*}
	and 
		\begin{equation*}
			\begin{split}
				\coker( N_k+zE_k)=\oplus_d G_{d,k}/( N_k+zE_k)(G_{d-1,k})
			\end{split}
		\end{equation*}
	respectively.

By Lemma~\ref{lem:kernel}, the kernel $\ker( N_k+zE_k)$ is $0$. We take $W_{d,k}\subset G_{d,k}$ such that the image of $W_{d,k}$ in $G_{d,k}/( N_k+zE_k)(G_{d-1,k})$ form a basis for $G_{d,k}/( N_k+zE_k)(G_{d-1,k})$. This choice of $W_{d,k}$ satisfies the first statement.

By construction, we have $\#W_{d,k}=\dim G_{d,k}-\dim G_{d-1,k}$ and
		\begin{equation*}
			\begin{split}
				\dim G_{d,k}=\#\Bigl\{(a,\underline I) \mid |\underline I|=k,\ (n+1)a+\sum_{i=0}^n i\cdot I_i=d\Bigr\}.
			\end{split}
		\end{equation*}
Since $\dim G_{d,k}$ is the number $N_{d,k}$ in \eqref{eq:N_{d,k}} and $\#W_{d,k}=n_{d,k}$, we deduce the second statement.

At last, by our construction of $W_{d,k}$, they form a basis for $\coker ( N_k+zE_k)$. By Proposition~\ref{prop:counting}, the sets $W_{d,k}$ are empty when $d>nk-n-k+1$, and we conclude the last three statements of the proposition.
\end{proof}

\subsection{The case of \texorpdfstring{$\Sym^k\Kl_{n+1}$}{SymkKln+1} when \texorpdfstring{$\gcd(k,n+1)=1$}{gcd(k, n+1)=1}}\label{sec:de-rham-cohomology}
Recall that $\Kl_{n+1}$ is the connection 
	\begin{equation*}
		\begin{split}
			\biggl(\mathcal{O}_{\mathbb{G}_{m}}^{n+1},\mathrm{d}+N\frac{\mathrm{d}z}{z}+E\,\mathrm{d}z\biggr)
		\end{split}
	\end{equation*}
defined in \eqref{eq:connection}. We can choose a basis $\{v_0,\ldots,v_n\}$ of $\Kl_{n+1}$ (as a $\bb{C}[z,z\inv]\langle\partial_z\rangle$-module) such that
	\begin{equation*}
		\begin{split}
			z\partial_z(v_i)=v_{i+1}\ \text{for }i=0,\ldots,n-1 \quad  \text{and} \quad z\partial_z(v_n)=zv_0,
		\end{split}
	\end{equation*}
i.e., $v_i$ generate $\bb{C}[z,z\inv]\langle\partial_z\rangle/((z\partial_z)^{n+1}-z)$. The set $\{v_i\}$ also serve as a basis for the connection 
	\begin{equation*}
		\begin{split}
			\widetilde{\Kl}_{n+1}=\biggl(\mathcal{O}_{\mathbb{G}_{m}}^{n+1},\mathrm{d}+(n+1)N\frac{\mathrm{d}t}{t}+(n+1)t^{n}E\mathrm{d}t\biggr),
		\end{split}
	\end{equation*}
and satisfies similar conditions
	\begin{equation*}
		\begin{split}
			t\partial_t(v_i)=(n+1)v_{i+1}\ \text{ for }i=0,\ldots,n-1 \quad \text{and} \quad  t\partial_t(v_n)=(n+1)t^{n+1}v_0.
		\end{split}
	\end{equation*}
In this way, the elements $\{v^{\underline I}:=v_0^{I_0}\cdots v_n^{I_n}\mid  |\underline I|=I_0+\dots+I_n=k\}$ form a basis for $\Sym^k\Kl_{n+1}$ (resp. $\Sym^{k}\widetilde \Kl_{n+1}$) as a $\bb{C}[z,z\inv]\langle\partial_z\rangle$-module (resp. $\bb{C}[t,t\inv]\langle\partial_t\rangle$-module).

Let $\deg$ be the degree on the symmetric power $\Sym^k\Kl_{n+1}$ (resp. $ \Sym^k\widetilde\Kl_{n+1}$), defined by the same formula as in \eqref{eq:grading} and \eqref{eq:grading-t}. We have the following theorems:

\begin{thm}\label{thm:cohomology_classes} 
	If $\gcd(k,n+1)=1$, there exist finite subsets $W_{d,k}\subset \Sym^k\Kl_{n+1}$ for $0\leq d\leq nk+1$ such that 
	\begin{enumerate}[label=\arabic*., ref=\ref{prop:cokernel}.(\theenumi]
		\item each element in $W_{d,k}$ has degree $d$,
		\item the generating series $\sum_{d,k\geq 0} \#W_{d,k}t^dx^k$ is the formal power series expansion of the rational function $h(t,x)$ in \eqref{eq:h(t,x)},
		\item the de Rham cohomology $\mathrm{H}^1_{\mathrm{dR}}(\bb{G}_m,\Sym^{k} \Kl_{n+1})$ is spanned by the set $\bigcup_d W_{d,k}$,
		\item $\#W_{d,k}-\#W_{nk+1-d,k}=q_{d,k}$ for all $0\leq d \leq nk+1$, where $q_{d,k}$ is the coefficient of $t^d$ in the formal power series expansion of the rational function $Q_k(t)$ in Lemma~\ref{lem:decomposition}.
	\end{enumerate}
\end{thm}
Moreover, we can choose a basis for the middle de Rham cohomologies. 
\begin{thm}\label{thm:mid-cohomology_classes}
	If $\gcd(k,n+1)=1$, there exist finite subsets $W_{d,k}^{\rr{mid}}\subset \Sym^k\Kl_{n+1}$ for $0\leq d\leq nk+1$ such that 
	\begin{enumerate}[label=\arabic*., ref=\ref{prop:cokernel}.(\theenumi)]
		\item each element in $W_{d,k}^{\rr{mid}}$ has degree $d$,
		\item the generating series $\sum_{d,k\geq 0}\#W^{\rr{mid}}_{d,k}t^dx^k$ is
			\begin{equation*}
				\begin{split}
					h(t,x)-\sum_{k\geq 0}\bar{Q}_k(t)x^k,
				\end{split}
			\end{equation*}
		where $\bar{Q}_k(t)=\sum_{d=0}^{\lfloor\frac{nk}{2}\rfloor}q_{d,k}t^d\bigl(\equiv Q_k(t)\mod t^{\lfloor\frac{nk}{2}\rfloor+1}\bigr)$,
		\item the middle de Rham cohomology $\mathrm{H}^1_{\mathrm{dR,mid}}(\bb{G}_m,\Sym^{k} \Kl_{n+1})$ is spanned by $\bigcup_d W_{d,k}^{\rr{mid}}$,
		\item $\#W^{\rr{mid}}_{d,k}=\#W^{\rr{mid}}_{nk+1-d,k}$ for all $0\leq d\leq nk+1$.
	\end{enumerate}
\end{thm}

\subsubsection{Bases of de Rham cohomologies}\label{sec:cohomology-class}

We take the notation $V_k$, $N_k$ and $E_k$ from Section~\ref{sec:4.1.2}. Let $G=G_k$ be $V_k[z,z\inv]:=V_k\otimes_{\bb{C}}\bb{C}[z,z\inv]$, on which $\theta_z:=z\partial_z$ acts by 
	\begin{equation}\label{eq:theta_z}
		\begin{split}
			\theta_z(z^\ell v^{\underline I})=\ell z^\ell v^{\underline I} + z^\ell (N_k + z E_k)v^{\underline I}.
		\end{split}
	\end{equation}
Similarly, the endomorphism $\theta_t:=t\partial_t$ acts on $\tilde{G}=G\otimes_{\bb{C}[z]}\bb{C}[t]$ ($z=t^{n+1}$) by 
	\begin{equation*}
		\begin{split}
			\theta_t(t^\ell v^{\underline I})=\ell t^\ell v^{\underline I}+ t^\ell (n+1)(N_k + t^{n+1} E_k)v^{\underline I}.
		\end{split}
	\end{equation*}
The first de Rham cohomologies $\mathrm{H}^1_{\mathrm{dR}}(\bb{G}_m,\Sym^k\Kl_{n+1})$ and $\mathrm{H}^1_{\mathrm{dR}}(\bb{G}_m,\Sym^k\widetilde \Kl_{n+1})$ are identified with the cokernels of the two-term complexes
	\begin{equation}\label{eq:derham}
		 G\xrightarrow{\theta_z}G \quad \text{and}\quad  \tilde{G}\xrightarrow{\theta_t}\tilde{G}
	\end{equation}
respectively.
\begin{lem}\label{lem:reduction}
	Let $G^+:=V_k\otimes_{\bb{C}}\bb{C}[z]$ and $\tilde{G}^+=G^+\otimes_{\bb{C}[z]}\bb{C}[t]$. Then, the two-term complexes
		\begin{equation*}
			\begin{split}
				G^+\xrightarrow{\theta_z}G^+ \quad \text{and} \quad \tilde{G}^+\xrightarrow{\theta_t} \tilde{G}^+
			\end{split}
		\end{equation*}
	are quasi-isomorphic to the complexes in \eqref{eq:derham} respectively.
\end{lem}
\begin{proof}
	The proof of the lemma is essentially that of \cite[Lem.\,4.15]{fresan2018hodge}. We only give the proof for $G^+$, and the proof for $\tilde{G}^+$ is similar. Observe that $ G=\bigcup_{r\geq 0}z^{-r}G^+$. It suffices to show that $\theta_z$ is invertible on the quotient $z^{-r-1}G^+/z^{-r}G^+$ for $r>0$. 
	
	In fact, the induced endomorphism $\theta_z$ on the quotient $z^{-r-1}G^+/z^{-r}G^+$ is now $z\partial_z+ N_k$. Using the decomposition with respect to $N_k$ in Lemma~\ref{lem:decomposition}, we have 
	\begin{equation*}
	\begin{split}
		z^{-r-1}G^+/z^{-r}G^+=\sum_{m=0}^{\lfloor\frac{kn}{2}\rfloor}z^{-r-1}(\Sym^{nk-2d}(\bb{C}^2))^{\oplus q_{d,k}}.
	\end{split}
	\end{equation*}
	Here $\theta_z$ acts on $z^{-r-1}\Sym^{nk-2d}(\bb{C}^2)$ by an invertible matrix $J_{nk-2d+1}(-r-1)$, where $J_{nk-2d+1}(-r-1)$ is the lower Jordan block with diagonal $-r-1$ of size $nk-2d+1$.
\end{proof}

The operators $\theta_z$ and $\theta_t$ are (inhomogeneous) of degree $1$ with respect to the degrees \eqref{eq:grading} and \eqref{eq:grading-t} respectively. We denote by $\bar\theta_z=(N_k+zE_k)$ (resp. $\bar\theta_t=(n+1)(N_k+t^{n+1}E_k)$) the induced map on the graded quotient $\mathrm{gr}\,G^+$ (resp. $\gr\,\tilde{G}^+$). We also identify $\gr\, G^+$ (resp. $\gr\, \tilde{G}^+$) with $G^+$ (resp. $\tilde{G}^+$).
\begin{proof}[Proof of Theorem~\ref{thm:cohomology_classes}]
Let $F_p^{\rr{deg}}G^+$ be the (increasing) filtration on $G^+$ defined by $\deg$ in \eqref{eq:grading}. We filter the two-term complex in Lemma~\ref{lem:reduction} by $\tilde{F}^pG^+[i]:=F^{\rr{deg}}_{-p+i}G^+[i]$ for $i=0,1$. Consider the spectral sequence associated with $\tilde F^\bullet $
	\begin{equation*}
		\begin{split}
			E_{1}^{p,q}=\mathrm{H}^{p+q}\bigl(\mathrm{gr}^{F^{\rr{deg}}}_{-p}G^+\xrightarrow{\bar{\theta}_z}\mathrm{gr}^{F^{\rr{deg}}}_{-p+1}G^+\bigr)\Longrightarrow \mathrm{H}^{p+q}(G^+,\theta_z)\quad (p\leq 0, p+q\in \{0,1\}),
		\end{split}
	\end{equation*}
which degenerates at the $E_2$-page. Since $\gcd(k,n+1)=1$, by Proposition~\ref{label:cokernel-4}, we have the identification
	\begin{equation*}
		\begin{split}
			\coker \bar\theta_z=\bigoplus_d \bigoplus_{w\in W_{d,k}} \bb{C}w.
		\end{split}
	\end{equation*}
Since $\bar\theta_z$ is injective by Lemma~\ref{lem:kernel}, the spectral sequence already degenerates at the $E_1$-page. So there exists a decreasing filtration $\tilde{F}^\bullet$ on $\coker \theta_z$ such that the graded pieces of $\coker\theta_z$ are $\oplus_{w\in W_{d,k}} \bb{C}w$. So we can take the union of these finite sets $W_{d,k}$ as a basis for $\coker(\theta_z) $. We deduce all statements in the proposition by Proposition~\ref{prop:cokernel}.
\end{proof}

\subsubsection{Bases of the middle de Rham cohomologies}

Now we determine bases for the middle de Rham cohomologies of $\Sym^k\Kl_{n+1}$. Let $\hat G_0$, $\hat G_\infty$, $\hat {\tilde G}_0$ and $\hat{\tilde G}_\infty$ be $G((z))$, $ G((z\inv))$, $\tilde G((t))$ and $\tilde G((t\inv))$ respectively. The key to determining the cohomology classes is the following proposition:
\begin{prop}[{\cite[Cor.\,3.5]{fresan2020quadratic} }]\label{prop:lemma from bessel moments paper}
	We can identify the cohomology with compact support $\mathrm{H}^1_{\mathrm{dR,c}}(\mathbb{G}_m,\Sym^k \Kl_{n+1})$ with the quotient of the $\mathbb{C}$-vector space
		\begin{equation*}
			\begin{split}
				\{(m_0,m_\infty,\omega)\in \hat G_0\times \hat G_\infty \times G \mid \theta_z(m_0,m_\infty)=(\omega|_{0},\omega|_{\infty})\}
			\end{split}
		\end{equation*}
	by the $\mathbb{C}$-vector space $\{(h|_0,h|_\infty,\theta(h))\mid h \in G\}$.
\end{prop}

\begin{lem}\label{solution_at_0}
	The cokernel of $\theta_z\colon \hat G_0\to \hat G_0$ (resp. $\theta_t\colon \hat{\tilde G}_0\to \hat{\tilde G}_0$) is $\coker N_k$. In particular, we have $z\cdot G^+\subset \theta_z(\hat G_0)$ and $t\cdot \tilde{G}^+\subset \theta_t(\hat{\tilde G}_0)$.
\end{lem}

\begin{proof}
	We give the proof for the case of $\theta_z$, and the proof for the case of $\theta_t$ is similar. Using a similar argument as that of Lemma~\ref{lem:reduction}, the complex $ \hat{G}_0\xrightarrow{\theta_z} \hat{G}_0$ is quasi-isomorphic to $ \hat{G}^+_0\xrightarrow{\theta_z}\hat G^+_0$, where $\hat G^+_0:= G^+[[z]]$.

	We define a decreasing filtration ${F}^\bullet$ on 
		\begin{equation}\label{eq::spectral}
			\hat G_0^+\xrightarrow{\theta_z}\hat G_0^+,
		\end{equation} 
	by ${F}^{m}(\hat G_0^+)=\oplus_{j=m}^\infty z^j\cdot V_k$ if $m\geq 0$ and ${F}^m={F}^0$ if $m<0$. The complex is complete with respect to the filtration ${F}^\bullet$, and the induced spectral sequence $ \{E^{p,q}_r\}$ is complete, exhaustive, and regular. By the complete convergence theorem \cite[Thm.\,5.5.10]{weibel-homological}, the associated spectral sequence $ \{E^{p,q}_r\}$ converges to the cohomologies $\rr{H}^{p+q}(\hat G_0^+\xrightarrow{\theta_z}\hat G_0^+)$ of \eqref{eq::spectral}, concentrated in degree $0$ and $1$.
    Here, the terms in the $E_0$-page are given by 
		\begin{equation*}
			\begin{split}
				E^{p,q}_0:=\begin{cases}
					\mathrm{gr}^p_F( \hat G_0^+)\simeq z^p\cdot V_k,\ &p\geq 0 \quad \text{and} \quad  p+q\in \{0,1\}\\
					0,\ &\text{else}\end{cases}
			\end{split}
		\end{equation*}
	and the morphisms $d_0^{p,-p}\colon z^p\cdot V_k\to z^p\cdot V_k$ are induced by $ \theta_z$ (in fact they coincide with $p\cdot \id+ N_k$ in the sense that $d_0^{p,-p}(z^pv)=pz^pv+z^pN_k^{1}v$). The morphism $d^{p,-p}_0=p\cdot\id+N_k$ is an isomorphism if $p\neq 0$. So all terms in the $E_1$-page are $0$ except $E_1^{0,0}=\ker N_k$ and $E_1^{0,1}=\coker N_k$. Therefore, the cokernel of \eqref{eq::spectral} is $\coker N_k$.
\end{proof}

\begin{lem}\label{solution_at_infty}
	If $\gcd(k,n+1)=1$, the cokernels of $\theta_z\colon \hat{G}_{\infty}\to \hat{G}_{\infty}$ and $\theta_t\colon \hat{\tilde G}_\infty\to \hat{\tilde G}_\infty$ are $0$. 
\end{lem}
\begin{proof}
	We give the proof for $\theta_z$, and the proof for $\theta_t$ is similar. Recall that in $ \hat{G}_\infty$, the subspace of elements of degree $0$ with respect to $z$ is $V_k$. Similar to the proof of Theorem~\ref{thm:cohomology_classes}, we consider the filtration on $\hat{ G}_\infty$ induced by the degree \eqref{eq:grading}, i.e., $\tilde{F}^p\hat{G}_\infty^+[i]:=F^{\rr{deg}}_{-p+i}\hat{G}_\infty^+[i]$ for $i=0,1$. The induced spectral sequence 
		\begin{equation*}
			\begin{split}
				E_{1}^{p,q}=\mathrm{H}^{p+q}\bigl(\mathrm{gr}^{F^{\rr{deg}}}_{-p}\hat{  G}_\infty
		\xrightarrow{\bar{\theta}_z}
		\mathrm{gr}^{F^{\rr{deg}}}_{-p+1}\hat{  G}_\infty\bigr)
		\Longrightarrow \mathrm{H}^{p+q}(\hat{ G}_\infty,\theta_z)\quad (p+q\in \{0,1\})
			\end{split}
		\end{equation*}
	degenerates at the $E_2$-page. The terms in the $E_0$-page are given by 
		\begin{equation*}
			\begin{split}
				E^{p,q}_0:=\begin{cases}
					\mathrm{gr}^{F^{\rr{deg}}}_{-p+(p+q)}( \hat G_\infty)=\mathrm{gr}^{F^{\rr{deg}}}_{q}( \hat G_\infty),\ &p+q\in \{0,1\}\\
					0,\ &\text{else}\end{cases}
			\end{split}
		\end{equation*}
and the morphisms $d_0^{p,-p}$ are induced by $ \theta_z$ (in fact they coincide with $ N_k+zE_k$ in the sense that $d_0^{p,-p}v=(N_k+zE_k)v$). Since the morphism $d^{p,-p}_0=N_k+zE_k$ is invertible, all terms in the $E_1$-page are $0$. Therefore, the cokernel of $\theta_z\colon \hat{  G}_\infty\to \hat{  G}_\infty$ is $0$.
\end{proof}

\begin{proof}[Proof of Theorem~\ref{thm:mid-cohomology_classes}]
	There are commutative diagrams
		$$\begin{tikzcd}
			G^+ \ar[r,"\theta_z"]\ar[d] & G^+\ar[d] \ar[r] & \coker \theta_z\ar[d]\\
			\hat G_0 \ar[r,"\theta_z"] & \hat G_0 \ar[r]& \coker N_k
		\end{tikzcd}$$
	and 
	$$\begin{tikzcd}
		G^+ \ar[r,"\theta_z"]\ar[d] & G^+\ar[d] \ar[r] & \coker \theta_z\ar[d]\\
		\hat{ G}_\infty \ar[r,"\theta_z"] & \hat{G}_\infty \ar[r]& 0
	\end{tikzcd}$$
by Lemma~\ref{solution_at_0} and Lemma~\ref{solution_at_infty}. Since $N_k$ is also homogeneous of degree $1$, there exist finite sets $\Sigma_{d,k}\subset G_{d,k}$ such that the image of $\bigcup_d \Sigma_{d,k}$ in $\coker \theta_z$ form a basis for $\coker N_k$. We may assume that $\Sigma_{d,k}\subset W_{d,k}$.

For $w=\sum_{i=0}^{m} z^iw_i\in  G_{d,k}$, we can decompose $w$ as $w'+w''$, where $w'$ is the constant term $w_0$ with respect to $z$, and $w''=\sum_{i=1}^m z^iw_i$. Up to replacing $w$ by $ w+\theta_z(h)$ for some $h\in G_{d-1,k}\cap V_k$, we can assume that $w'\in \coker N_k$ (If we write $w'=-N_k h+ g$ for $h\in G_{d-1,k}\cap V_k$ and $g\in G_{d,k}\cap\coker N_k$, then $w+\theta_z(h)= g +(w''+z\cdot E_k h)$). By Lemma~\ref{solution_at_0} and Lemma~\ref{solution_at_infty}, we know that $w''$ is in the image of both $\theta_z\colon \hat{G}_0\to \hat{G}_0$ and $\theta_z\colon \hat{G}_{\infty}\to \hat{G}_{\infty}$.

Let $W'_{d,k}$ be the set of $w''\in G_{d,k}\cap zG^+$ such that there exists $w'\in  G_{d,k}\cap \coker N_k$ satisfying  $w'+w''\in W_{d,k}$. We take $W_{d,k}^{\rr{mid}}\subset G_{d,k}$ as a maximally linearly independent subset of $W'_{d,k}$. Then 
\begin{itemize}
	\item each element in $ W_{d,k}^{\rr{mid}}$ has degree $d$,
	\item $\rr{span}( W_{d,k})=\rr{span}( W^{\rr{mid}}_{d,k})\oplus \rr{span}(\Sigma_{d,k})$,
	\item elements in $W_{d,k}^{\rr{mid}}$ are in the image of both $\theta_z\colon \hat{G}_0\to \hat{G}_0$ and $\theta_z\colon \hat{G}_{\infty}\to \hat{G}_{\infty}$.
\end{itemize} 

By the construction of $W^{\rr{mid}}_{d,k}$, the first statement in Theorem~\ref{thm:mid-cohomology_classes} is verified. Notice that 
	\begin{equation*}
		\begin{split}
			\#W_{d,k}^{\rr{mid}}=
	\begin{cases}
		n_{d,k}-q_{d,k} & 0\leq d\leq \frac{nk}{2}\\
		n_{d,k} & \frac{nk+1}{2}\leq d\leq  nk+1
	\end{cases}
		\end{split}
	\end{equation*}
and $\#\Sigma_{d,k}=q_{d,k}$, we deduce the second and the forth statements by Proposition~\ref{prop:counting}. 

At last, by the key Proposition~\ref{prop:lemma from bessel moments paper}, we have 
	\begin{equation*}
		\begin{split}
			\rr{span}\Bigl(\bigcup  W_{d,k}^{\rr{mid}}\Bigr)\subset \mathrm{H}^1_{\mathrm{dR,mid}}(\bb{G}_m,\Sym^{k}\Kl_{n+1}).
		\end{split}
	\end{equation*}
We use the computation of dimensions from Lemma~\ref{lem:dim_of_inv+mid} to conclude that the above inclusion is an identity. In other words, the third statement in Theorem~\ref{thm:mid-cohomology_classes} holds.
\end{proof}

\begin{rek}\label{rek:inv-0-infty}
	We obtain from the above proof an isomorphism of vector spaces 
		\begin{equation*}
			\begin{split}
				\mathrm{H}^1_{\rr{dR}}(\bb{G}_m,\Sym^k\Kl_{n+1})/\mathrm{H}^1_{\rr{dR, mid}}(\bb{G}_m,\Sym^k\Kl_{n+1})\simeq \coker N_k
			\end{split}
		\end{equation*}
	when $\gcd(k,n+1)=1$.
\end{rek}

\subsection{The case of \texorpdfstring{$\Sym^k\Kl_3$}{SymkKl3} when \texorpdfstring{$3\mid k$}{3|k}}
We take the same notation from Section~\ref{sec:de-rham-cohomology} and let $n=2$. Recall that 
	\begin{equation*}
		\begin{split}
			\Hdr{1}(\bb{G}_m,\Sym^k\Kl_3)=\Hdr{1}(\bb{G}_m,\Sym^k\tilde\Kl_3)^{\mu_3}.
		\end{split}
	\end{equation*}
So we can choose the basis for $\Hdr{1}(\bb{G}_m,\Sym^k\Kl_3)$ as a subset of that of $\Hdr{1}(\bb{G}_m,\Sym^k\tilde\Kl_3)$.

\begin{thm}\label{thm:cohomology_classes-kl3} 
	For $0\leq d \leq 2k+1$, there exist finite subsets $W_{d,k}\subset \Sym^k\Kl_3$ and $\widetilde W_{d,k} \subset \Sym^k\widetilde\Kl_3$ such that 
	\begin{enumerate}[label=\arabic*., ref=\ref{prop:cokernel}.\theenumi]
		\item  $W_{d,k}\subset \widetilde W_{d,k}$,
		\item each element in $W_{d,k}$ (resp. $\widetilde W_{d,k}$) has degree $d$,
		\item the de Rham cohomology $\mathrm{H}^1_{\mathrm{dR}}(\bb{G}_m,\Sym^{k} \Kl_3)$ (resp. $\mathrm{H}^1_{\mathrm{dR}}(\bb{G}_m,\Sym^{k}\widetilde\Kl_3)$) is spanned by the set $\bigcup_d W_{d,k}$ (resp. $\bigcup_d \widetilde W_{d,k})$,
		\item if $d\leq k$, the cardinality of $W_{d,k}$ (resp. $\widetilde W_{d,k}$) is
			\begin{equation*}
				\begin{split}
					n_{d,k}=\begin{cases}
						\lfloor\frac{d}{6}\rfloor+1 & 6\nmid d-1\\
						\lfloor\frac{d}{6}\rfloor		& 6\mid d-1
					\end{cases}
				\end{split}
			\end{equation*}
		(resp, $\tilde n_{d,k}=\lfloor\frac{d}{2}\rfloor+1$).
	\end{enumerate}
\end{thm}

As in the case of $\Sym^k\Kl_{n+1}$ when $\gcd(k,n+1)= 1$, we choose the bases of the middle de Rham cohomologies. 
\begin{thm}\label{thm:mid-cohomology_classes-kl3} 
	For $0\leq d\leq 2k+1$, there exist finite subsets $W_{d,k}^{\rr{mid}}\subset \Sym^k\Kl_3$ and $\widetilde W_{d,k}^{\rr{mid}} \subset \Sym^k\widetilde\Kl_3$ such that 
	\begin{enumerate}[label=\arabic*., ref=\ref{prop:cokernel}.\theenumi]
		\item  $W_{d,k}^{\rr{mid}}\subset \widetilde W_{d,k}^{\rr{mid}}$,
		\item the middle de Rham cohomologies $\mathrm{H}^1_{\mathrm{dR,mid}}(\bb{G}_m,\Sym^{k} \Kl_3)$ and $\mathrm{H}^1_{\mathrm{dR,mid}}(\bb{G}_m,\Sym^{k}\widetilde\Kl_3)$ are spanned by the set $\bigcup_d W_{d,k}^{\rr{mid}} \ (\text{resp. } \bigcup_d \widetilde W_{d,k}^{\rr{mid}})$,
		\item if $d\leq k$, the cardinality of $W_{d,k}^{\rr{mid}}$ (resp. $\widetilde W_{d,k}^{\rr{mid}}$) is
			\begin{equation*}
				\begin{split}
					n_{d,k}^{\rr{mid}}=-\delta_{d,k}+\begin{cases}
						\lfloor\frac{d}{6}\rfloor & d\equiv 0,1,2,4\ \mathrm{mod} \ 6 ,
				   \\
				   \lfloor\frac{d}{6}\rfloor+1 & p\equiv 3,5\ \mathrm{mod} \ 6  ,
					\end{cases}
				\end{split}
			\end{equation*}
		(resp, $\tilde n_d^{\rr{mid}}=\lfloor\frac{d+1}{2}\rfloor-\delta_{d,k}$), where $\delta_{a,b}$ be the Kronecker delta symbol.
	\end{enumerate}
\end{thm}

\subsubsection{Bases of de Rham cohmologies}

We take the notation from Sections~\ref{sec:4.1.2} and~\ref{sec:cohomology-class}. Let $n=2$ and assume that $3\mid k$.
\begin{proof}[Proof of Theorem~\ref{thm:cohomology_classes-kl3}]
As in Proposition~\ref{prop:cokernel}, we can write $G^+=\gr\, G^+=\oplus_{d\geq 0} G_{d,k}$ and $\tilde{G}^+=\gr\,\tilde{G}^+=\oplus_{d\geq 0}\tilde G_{d,k}$, where $G_{d,k}$ and $\tilde G_{d,k}$ are generated by elements of degree $d$ in $G^+$ and $\tilde{G}^+$ respectively. Since $\bar\theta_z$ (resp. $\bar\theta_t$) is homogeneous of degree $1$ with respect to the degrees in \eqref{eq:grading} (resp. \eqref{eq:grading-t}), the kernel and cokernel of $\bar\theta_z$ (resp. $\bar\theta_t$) can be decomposed as 
		\begin{equation*}
			\begin{split}
				\ker\bar\theta_z= \bigoplus_d \ker\bar\theta_z\cap G_{d,k}
		\text{ (resp. }
		\ker\bar\theta_t= \bigoplus_d \ker\bar\theta_t\cap \tilde G_{d,k}
		\text{)}
			\end{split}
		\end{equation*}
	and 
		\begin{equation*}
			\begin{split}
				\coker\bar\theta_z=\bigoplus_d G_{d,k}/\,\bar{\theta}_z(G_{d-1,k})\text{ (resp. }
		\coker\bar\theta_t=\bigoplus_d \tilde G_{d,k}/\,\bar\theta_t(\tilde G_{d-1,k})
		\text{)}
			\end{split}
		\end{equation*}
	respectively. By the construction of $G_{d,k}$ and $\tilde G_{d,k}$, we have
		\begin{equation*}
			\begin{split}
				\dim G_{d,k}=\#\Bigl\{(a,\underline I)\ \Big{|}\ |\underline I|=k,\ (n+1)a+\sum_{i=0}^n i\cdot I_i=d\Bigr\}
			\end{split}
		\end{equation*}
	and 
		\begin{equation*}
			\begin{split}
				\dim \tilde G_{d,k}=\#\Bigl\{(a,\underline I)\ \Big{|}\  |\underline I|=k,\ a+\sum_{i=0}^n i\cdot I_i=d\Bigr\}.
			\end{split}
		\end{equation*}
	In this case, the kernels of $\bar\theta_z$ and $\bar\theta_t$ are of rank $1$ generated by $\eta^{k/3}$ by Lemma~\ref{lem:kernel}, where $\eta=f^{(1,1,1)}=z^2v_0^3+zv_1^3+v_2^3-3zv_0v_1v_2$ has degree $6$. We take $W_{d,k}\subset G_{d,k}$ and $\widetilde W_{d,k} \subset \tilde G_{d,k}$ such that 
	\begin{itemize}
		\item the elements in the sets $W_{d,k}$ and $\widetilde W_{d,k}$ have degree $d$,
		\item $W_{d,k}$ is contained in $\widetilde W_{d,k}$ and elements in $W_{d,k}$ are exactly $\mu_3$-invariant elements in $\widetilde W_{d,k}$,
		\item the image of $W_{d,k}$ in $G_{d,k}/(\bar\theta_z(G_{d-1,k})+G_{d,k}\cap \bb{C}[z]\eta^{k/3})$ is a basis,
		\item the image of $\widetilde W_{d,k}$ in $\tilde G_{d,k}/(\bar\theta_t(\tilde G_{d-1,k})+\tilde G_{d,k}\cap \bb{C}[t]\eta^{k/3})$ is a basis.
	\end{itemize}

Similar to the proof of Theorem~\ref{thm:cohomology_classes}, we consider the following spectral sequence 
		\begin{equation*}
			\begin{split}
				E_{1}^{p,q}=\mathrm{H}^{p+q}\bigl(\mathrm{gr}^{F^{\rr{deg}}}_{-p}\tilde{G}^+\xrightarrow{\bar{\theta}_t}\mathrm{gr}^{F^{\rr{deg}}}_{-p+1}\tilde{G}^+\bigr)\Longrightarrow \mathrm{H}^{p+q}(\tilde{G}^+,\theta_t)\quad (p\leq 0, p+q\in \{0,1\})
			\end{split}
		\end{equation*}
	which degenerates at the $E_2$-page. There is a similar one for $G^+$.
	
	The kernels of $\bar \theta_z$ and $\bar \theta_t$ are $\bb{C}[z]\eta^{k/3}$ and $\bb{C}[t]\eta^{k/3}$ respectively, and we have the identification
		\begin{equation*}
			\begin{split}
				\coker \bar\theta_z=\oplus_d \oplus_{w\in  W_{d,k}} \bb{C}w\oplus\bb{C}[z]\eta^{k/3} \quad \text{and} \quad \coker \bar\theta_t=\oplus_d \oplus_{w\in \widetilde W_{d,k}} \bb{C}w\oplus\bb{C}[t]\eta^{k/3}.
			\end{split}
		\end{equation*}

\begin{lem}\label{lem:induced-map}
	The morphism $\theta_t$ induces an injective morphism
		\begin{equation*}
			\theta_t\colon \ker\bar\theta_t\to \coker\bar\theta_t.
		\end{equation*} 
	Moreover, we have
			$\theta_t(t^r\eta^{k/3})\equiv \mu t^r\eta^{k/3} $
	in $\coker\bar\theta_t$ for some $\mu\neq 0$. We have similar results for $\theta_z$.
	\end{lem}
	\begin{proof}
	Write $k=3\ell$. For $r\geq 0$, we have 
			\begin{equation*}
				\begin{split}
					\begin{aligned}
						\theta_t(t^r\eta^\ell)
						=rt^r\eta^\ell + 3\ell t^r(2t^6v_0^3+t^3v_1^3-3t^3v_0v_1v_2)\eta^{\ell-1}.
						\end{aligned}
				\end{split}
			\end{equation*}
	Since 
		\begin{equation*}
			\begin{split}
				2t^6v_0^3+t^3v_1^3-3t^3v_0v_1v_2-\eta=t^6v_0^3-v_2^3=\frac{1}{3}\bar\theta_t(t^3v_0^2v_2-v_1v_2^2),
			\end{split}
		\end{equation*}
	it follows that 
			\begin{equation*}
				\begin{split}
					\begin{aligned}
						(2t^6v_0^3+t^3v_1^3-3t^3v_0v_1v_2)\eta^{\ell-1}
						&=\eta^\ell+\frac{1}{3}\bar \theta_t(t^3v_0^2v_2-v_1v_2^2)\eta^{\ell-1}\\
						&\stackrel{(*)}{=}\eta^\ell+\frac{1}{3}\bar \theta_t((t^3v_0^2v_2-v_1v_2^2)\eta^{\ell-1}).
					\end{aligned}
				\end{split}
			\end{equation*}
	Here (*) is because $\bar\theta_t(g\cdot \eta^{\ell-1})=\bar\theta_t(g)\eta^{\ell-1}+(\ell-1)g\cdot \bar\theta_t(\eta)\eta^{\ell-2}=\bar\theta_t(g)\eta^{\ell-1}$ for any $g\in \tilde G$.
	Hence, we deduce that
			\begin{equation*}
				\begin{split}
					\theta_t(t^r\eta^\ell)=\left(r+3\ell\right)t^r\eta^\ell+\bar\theta_t(\ell t^r(t^3v_0^2v_2-v_1v_2^2)\eta^{\ell-1}).
				\end{split}
			\end{equation*}
	Since $r,\ell\geq 0$, we conclude that $\theta_t$ is injective.
	\end{proof}

	The spectral sequence degenerates at the $E_2$-page because all morphisms $d_2^{p,q}$ are already $0$ by Lemma~\ref{lem:induced-map}. So the two vector spaces $\ker\theta_t$ and $\coker \theta_t$ are, respectively, the kernel and the cokernel of the induced map $\theta_t\colon \ker\bar\theta_t\to \coker\bar\theta_t$. Therefore, $\ker \theta_t=0$ and the composition of morphisms
	\begin{equation*}
		\begin{split}
			W:=\oplus_d \oplus_{w\in \widetilde W_{d,k}} \bb{C}w\hookrightarrow \oplus_d \oplus_{w\in \widetilde W_{d,k}} \bb{C}w\oplus\bb{C}[t]\eta^{k/3}= \coker \bar\theta_t \twoheadrightarrow \coker\theta_t
		\end{split}
	\end{equation*}
is an isomorphism of vector spaces. Hence, the set $\bigcup_d \widetilde W_{d,k}$ is a basis for $\rr{H}^1_{\rr{dR}}(\bb{G}_m,\Sym^k\widetilde \Kl_3)$. Similarly, the set $\bigcup_d  W_{d,k}$ is a basis for $\rr{H}^1_{\rr{dR}}(\bb{G}_m,\Sym^k \Kl_3)$.

At last, the cardinality of $W_{d,k}$ is
	\begin{equation*}
		\begin{split}
			\dim G_{d,k}-\dim G_{d-1,k}=n_{d,k}=
	\begin{cases}
		\lfloor\frac{d}{6}\rfloor+1 & 6\nmid d-1\\
		\lfloor\frac{d}{6}\rfloor		& 6\mid d-1
	\end{cases}
		\end{split}
	\end{equation*}
if $d\leq k$ by Example~\ref{exe:n=2-counting}. The cardinality of $\widetilde W_{d,k}$ is by definition the cardinality of the set 
	\begin{equation*}
		\begin{split}
			\{(0, I_0,I_1,I_2)\in \bb{N}^4\mid I_0+I_1+I_2=k,\, I_1+2I_2=d\},
		\end{split}
	\end{equation*}
which is in fact	
	$\#\{i_0\mid \max\{0,k-d\}\leq i_0\leq k-\lceil\tfrac{d}{2}\rceil\}.$
It follows that $\#\widetilde W_{d,k}=\lfloor\frac{d}{2}\rfloor+1$ if $d\leq k$.
\end{proof}

\subsubsection{Bases of the middle de Rham cohomologies}

We determine bases for the middle de Rham cohomologies when $3\mid k$.

\begin{lem}\label{solution_at_infty-kl3}
	The cokernel of $\theta_z\colon \hat{G}_{\infty}\to \hat{G}_{\infty}$ (resp. $\theta_t\colon \hat{\tilde G}_\infty\to \hat{\tilde G}_\infty$) is generated by $z^{k/3} v_0^k$ (resp. $t^{k} v_0^k$). 
\end{lem}

\begin{proof}
We show the case of $\theta_t$ and omit the case of $\theta_z$. Similar to Lemma~\ref{solution_at_infty}, we consider the filtration on $\hat{\tilde G}_\infty$ induced by the degree \eqref{eq:grading-t}. The induced spectral sequence 
		\begin{equation*}
			\begin{split}
				E_{1}^{p,q}=\mathrm{H}^{p+q}\Bigl(\mathrm{gr}^{F^{\rr{deg}}}_{-p}\hat{\tilde G}_\infty
				\xrightarrow{\bar{\theta}_t}
				\mathrm{gr}^{F^{\rr{deg}}}_{-p+1}\hat{\tilde G}_\infty\Bigr)
				\Longrightarrow \mathrm{H}^{p+q}(\hat{\tilde G}_\infty,\theta_t)\quad ( p+q\in \{0,1\})
			\end{split}
		\end{equation*}
	degenerates at the $E_2$-page. The terms in the $E_0$-page are given by 
		\begin{equation*}
			\begin{split}
				E^{p,q}_0:=\begin{cases}
	               \mathrm{gr}^{F^{\rr{deg}}}_{q}( \hat{\tilde G}_\infty^+),\ &p+q\in \{0,1\}\\
					0,\ &\text{else}\end{cases}
			\end{split}
		\end{equation*}
and the morphisms $d_0^{p,-p}$ are induced by $ \theta_t$ (in fact they coincide with $3(N_k+t^{3}E_k)$ in the sense that $d_0^{p,-p}v=3(N_k+t^{3}E_k)v$). 

Let $k=3\ell$. Recall that we showed that the kernel of $N_k+t^3E_k$ is generated by $\eta^\ell$ in Lemma~\ref{lem:kernel}. The non-zero terms in the $E_1$-page are $E_1^{-a-6\ell,a+6\ell}=E_1^{-a-6\ell+1,a+6\ell}=\bb{C}t^{a}\eta^{\ell}$ for each $a\in \bb{Z}$.

Now let us make $d_1^{-a-6\ell,a+6\ell} $ explicit. For any $a\in \bb{Z}$, we have
	\begin{equation*}\label{eq:2.5}
		\begin{split}
			\theta_t\bigl(t^{a}\eta^\ell\bigr)
			&=a t^a\eta^\ell + 3\ell t^{a}\eta^{\ell-1}(2t^6v_0^3+t^3v_1^3-3t^3v_0v_1v_2)\\[4pt]
			&=a t^a\eta^\ell + 3\ell t^{a}\eta^{\ell-1}\biggl(\eta+\sum_{|\underline J|=3, C_{\underline J}\neq 0}\lambda_{\underline J}f^{\underline J}\biggr)\\
			&=(a+3\ell)t^a\eta^\ell + t^a\sum_{|\underline I|=k,C_{  \underline I}\neq 0}\mu_{  \underline I}f^{  \underline I}
		\end{split}
	\end{equation*}
for some complex numbers $\lambda_{\underline J}$ and $\mu_{\underline I}$. It follows that $d_1^{-a-6\ell,a+6\ell}(t^a\eta^\ell)=(a+3\ell)t^a\eta^\ell.$ Hence, every term in the $E_2$-page is $0$, except for $E_2^{-3\ell,3\ell}=E_1^{-3\ell+1,3\ell}=\bb{C}t^{-3\ell}\eta^{\ell}$.

Notice that 
	\begin{equation*}
		\begin{split}
			(3t^2v_0)^k=(f_0+\zeta \inv f_1+\zeta f_2)^k= \delta \eta^\ell +\sum_{\underline I\neq (\ell,\ell,\ell)}\delta_{\underline I}f^{\underline I}
		\end{split}
	\end{equation*}
for some non-zero $\delta\in \bb{C}$ and $\delta_{\underline I}\in \bb{C}$. It follows that $t^kv_0^k$ is cohomologous to $\delta t^{-3\ell}\eta^\ell$. Therefore, the cokernel of $\theta_t\colon \hat{\tilde G}_\infty\to \hat{\tilde G}_\infty$ is $\bb{C}t^kv_0^k$.
\end{proof}

\begin{proof}[Proof of Theorem~\ref{thm:mid-cohomology_classes-kl3}]
As in the proof of Theorem~\ref{thm:mid-cohomology_classes}, we conclude that there exist finite sets $\Sigma_{d,k}\subset G_{d,k}$ such that the image of $\bigcup_d \Sigma_{d,k}$ in $\coker \theta_t$ form a basis for $\coker N_k$.

For $w=\sum_{i=0}^{d} t^iw_i\in \tilde G_{d,k}$, we can decompose $w$ as $w'+w''$, where $w'=w_0+\gamma t^kv_0^k$ for some $\gamma\in \bb{C}$, and $w''=\sum_{i=1}^d t^iw_i-\gamma t^kv_0^k$. Up to replacing $w$ by $ w+\bar\theta_z(h)$ for some $h\in G^+$, we may assume that $w_0\in \coker N_k$. By Lemma~\ref{solution_at_0} and Lemma~\ref{solution_at_infty-kl3}, we can choose $\gamma$ such that $w''$ is in the image of both $\theta_t\colon \hat{\tilde G}_0\to \hat{\tilde G}_0$ and $\theta_t\colon \hat{\tilde G}_\infty \to \hat{\tilde G}_\infty$.

As a consequence, one can choose finite sets $\widetilde W_{d,k}^{\rr{mid}}\subset \tilde G_{d,k}$ such that 
\begin{itemize}
	\item $\rr{span}(\widetilde W_{d,k})=\rr{span}(\widetilde W^{\rr{mid}}_{d,k})\oplus \rr{span}(\Sigma_{d,k})\oplus \bb{C}[t]t^{k}v_0^k$,
	\item elements in $\widetilde W_{d,k}^{\rr{mid}}$ are in the image of both $\theta_t\colon \hat{\tilde G}_0\to \hat{\tilde G}_0$ and $\theta_t\colon \hat{\tilde G}_\infty \to \hat{\tilde G}_\infty$.
\end{itemize} 

We can choose the sets $W_{d,k}^{\rr{mid}}$ similarly. Since $\coker N$ and $t^kv_0^k=z^{k/3}v_0^k$ (if $3\mid k$) are contained in $G^+$, we can moreover choose $\widetilde W_{d,k}^{\rr{mid}}$ containing $W_{d,k}^{\rr{mid}}$. This verifies the first statement of Theorem~\ref{thm:mid-cohomology_classes-kl3}.

By Proposition~\ref{prop:lemma from bessel moments paper}, we have $\rr{span}(\bigcup_d \widetilde W_{d,k}^{\rr{mid}})\subset \mathrm{H}^1_{\mathrm{dR,mid}}(\bb{G}_m,\Sym^{k}\widetilde\Kl_3)$. Then we use the calculation of dimensions from Lemma~\ref{lem:dim_of_inv+mid} to deduce that
	\begin{equation*}
		\begin{split}
			&\dim  \mathrm{H}^1_{\mathrm{dR,mid}}(\bb{G}_m,\Sym^{k}\widetilde\Kl_3)\\
			=& \dim \mathrm{H}^1_{\mathrm{dR}}(\bb{G}_m,\Sym^{k}\widetilde\Kl_3)-\Bigl\lfloor\frac{k+2}{2}\Bigr\rfloor-1\\
			= &\dim \mathrm{H}^1_{\mathrm{dR}}(\bb{G}_m,\Sym^{k}\widetilde\Kl_3)-(\dim \rr{span}(\widetilde W_{d,k})-\dim \rr{span}(\widetilde W^{\rr{mid}}_{d,k}))\\
			\leq &\dim  \mathrm{H}^1_{\mathrm{dR,mid}}(\bb{G}_m,\Sym^{k}\widetilde\Kl_3).
		\end{split}
	\end{equation*}
Hence, the set $\bigcup_d \widetilde W_{d,k}^{\rr{mid}}$ is a basis for $\mathrm{H}^1_{\mathrm{dR,mid}}(\bb{G}_m,\Sym^{k}\widetilde\Kl_3)$. Notice that 
	\begin{equation*}
		\begin{split}
			\mathrm{H}^1_{\mathrm{dR,mid}}(\bb{G}_m,\Sym^{k} \Kl_3)= \mathrm{H}^1_{\mathrm{dR,mid}}(\bb{G}_m,\Sym^{k}\widetilde\Kl_3)^{\mu_3}.
		\end{split}
	\end{equation*}
We also deduce that $\rr{span}(\bigcup_d W_{d,k}^{\rr{mid}})= \mathrm{H}^1_{\mathrm{dR,mid}}(\bb{G}_m,\Sym^{k}\Kl_3)$. This verifies the second statement.

At last, we give the formula for the numbers $n_{d,k}^{\rr{mid}}$ and $\tilde n_{d,k}^{\rr{mid}}$ in the last statements. The degree $d$ part of $\coker N$ has dimension $1$ for $d=0,2,4,\dots, 2\lfloor\frac{k}{2}\rfloor$, and $0$ otherwise. The degree $d$ part of $\bb{C}t^kv_0^k$ has dimension $1$ if $d=k$ and $0$ otherwise. By Theorem~\ref{thm:cohomology_classes-kl3}, the cardinalities $\#\widetilde W_{d,k}^{\rr{mid}}$ are thus the numbers stated in the theorem.
\end{proof}

\begin{rek}\label{rek:inv-0-infty-kl3}
From the above, we obtain an isomorphism of vector spaces 
\begin{equation*}
	\begin{split}
		&\mathrm{H}^1_{\rr{dR}}(\bb{G}_m,\Sym^k\Kl_{3})/\mathrm{H}^1_{\rr{dR, mid}}(\bb{G}_m,\Sym^k\Kl_{3})\\
		\simeq&\mathrm{H}^1_{\rr{dR}}(\bb{G}_m,\Sym^k\widetilde\Kl_{3})/\mathrm{H}^1_{\rr{dR, mid}}(\bb{G}_m,\Sym^k\widetilde\Kl_{3})\simeq \coker N_k\oplus \bb{C}z^{k/3}v_0^k
	\end{split}
\end{equation*}
when $3\mid k$.
\end{rek}

\subsection{The case of \texorpdfstring{$\Sym^k\rr{Ai}_{n}$}{SymkAin} when \texorpdfstring{$\gcd(k,n)=1$}{gcd(k, n)=1}}

As in Section~\ref{sec:cohomology-class}, we select bases of the de Rham cohomologies of $\Sym^k\rr{Ai}_{n}$ when $\gcd(k,n)=1$. Recall that $\rr{Ai}_{n}$ is the connection 
	\begin{equation*}
		\begin{split}
			\bigl(\mathcal{O}_{\bb{A}^1_z}^{n},\mathrm{d}+N \mathrm{d}z +zE\,\mathrm{d}z\bigr)
		\end{split}
	\end{equation*}
defined in \eqref{eq:connection-Airy}. We can choose a basis $\{v_0,\ldots,v_{n-1}\}$ of $\rr{Ai}_{n}$ (as a $\bb{C}[z]\langle\partial_z\rangle$-module) such that
	\begin{equation*}
		\begin{split}
			z\partial_z(v_i)=v_{i+1}\ \text{for }i=0,\ldots,n-2 \quad\text{and}\quad z\partial_z(v_{n-1})=zv_0,
		\end{split}
	\end{equation*}
i.e., $v_i$ generate $\bb{C}[z]\langle\partial_z\rangle/((z\partial_z)^{n}-z)$. In this way, the elements 
	\begin{equation*}
		\begin{split}
			\{v^{\underline I}:=v_0^{I_0}\cdots v_{n-1}^{I_{n-1}}| |\underline I|=I_0+\dots+I_{n-1}=k\}
		\end{split}
	\end{equation*}
form a basis for $\Sym^k\rr{Ai}_{n}$ as a $\bb{C}[z]\langle\partial_z\rangle$-module.

Let $\deg$ be the grading on the symmetric power $\Sym^k\rr{Ai}_{n}$, defined by the same formula as in \eqref{eq:grading}. We have the following theorem:

\begin{thm}\label{thm:cohomology_classes-Ai} 
	If $\gcd(k,n)=1$, there exist finite subsets $W_{d,k}\subset \Sym^k\rr{Ai}_{n}$ for $0\leq d\leq nk-n-k+1$ such that 
	\begin{enumerate}[label=(\arabic*)., ref=\ref{prop:cokernel-Ai}.(\theenumi)]
		\item the de Rham cohomology $\mathrm{H}^1_{\mathrm{dR}}(\bb{A}^1,\Sym^{k} \rr{Ai}_{n})$ is spanned by the set $\bigcup_d W_{d,k}$,
		\item each element in $W_{d,k}$ has degree $d$,
		\item the generating series $\sum_{d,k\geq 0} \#W_{d,k}t^dx^k$ is the formal power series expansion of the rational function $h(t,x)$ in \eqref{eq:h(t,x)},
		\item $\#W_{d,k}=\#W_{nk-n-k+1-d,k}$ for all $0\leq d \leq nk-n-k+1$.
	\end{enumerate}
\end{thm}

\begin{proof}[Proof of Theorem~\ref{thm:cohomology_classes-Ai}]

We use the notation $V_k=\Sym^kV$, $N_k$ and $E_k$ from Section~\ref{sec:4.1.2}. Let $G^+=G_k^+$ be $V_k[z]:=V_k\otimes_{\bb{C}}\bb{C}[z]$, on which $\theta_z:=\partial_z$ acts by 
	\begin{equation*}
		\begin{split}
			\theta_z(z^\ell v^{\underline I})=\ell z^{\ell-1} v^{\underline I} + z^\ell (N_k + z E_k)v^{\underline I}.
		\end{split}
	\end{equation*}

The first de Rham cohomologies $\mathrm{H}^1_{\mathrm{dR}}(\bb{A}^1,\Sym^k\rr{Ai}_{n})$ are identified with the cokernels of the two-term complexes
	\begin{equation}\label{eq:derham-Ai}
		 G^+\xrightarrow{\theta_z}G^+ .
	\end{equation}
 Then $\theta_z$ is (inhomogeneous) of degree $1$. We denote by $\bar\theta_z=(N_k+zE_k)$ the induced map on the graded quotient $\mathrm{gr}\,G^+$, and we identify $\gr\, G^+$ with $G^+$.

Consider the following spectral sequence 
		\begin{equation*}
			\begin{split}
				E_{1}^{p,q}=\mathrm{H}^{p+q}\bigl(\mathrm{gr}^{F^{\rr{deg}}}_{-p}G^+\xrightarrow{\bar{\theta}_z}\mathrm{gr}^{F^{\rr{deg}}}_{-p+1}G^+\bigr)\Longrightarrow \mathrm{H}^{p+q}(G^+,\theta_z)\quad (p\leq 0, p+q\in \{0,1\})
			\end{split}
		\end{equation*}
	associated with the (increasing) filtration by $\deg$ on $G^+$. It degenerates at the $E_2$-page. 

	Since $\gcd(k,n)=1$, by Proposition~\ref{label:cokernel-4}, we have the identification
		\begin{equation*}
			\begin{split}
				\coker \bar\theta_z=\oplus_d \oplus_{w\in W_{d,k}} \bb{C}w.
			\end{split}
		\end{equation*}
	By Lemma~\ref{lem:kernel}, the map $\bar\theta_z$ is injective. So, the spectral sequence already degenerates at the $E_1$-page. So there exists a decreasing filtration $\tilde F^\bullet$ on $\coker \theta_z$ such that the graded pieces of $\coker\theta_z$ are $\oplus_{w\in W_{d,k}} \bb{C}w$. Hence, we can take the union of these finite sets $W_{d,k}$ as a basis for $\coker(\theta_z) $.  Therefore, all statements follow from Proposition~\ref{prop:counting}.
\end{proof}

\subsection{The case of \texorpdfstring{ $\mathrm{Kl}_{\rr{SL}_3}^{V_{2,1}}$}{Kl3V21}}
Let $V$ be the standard representation of $\mathrm{SL}_3$, and $V_{2,1}$ be the representation of the highest weight $2L_1+(L_1+L_2)$, which has dimension $15$. We denote by $P$ and $Q$ the two subgroups of $S_4$ generated by $\{(12),(123)\}$ and $\{(14)\}$ respectively, and by $\chi$ the character $\chi\colon P\times Q\xrightarrow{\pr_2}Q\xrightarrow{\mathrm{sign}}\{\pm1\}.$ By the Weyl construction \cite[\S15.3]{W.Fulton2013}, the representation $V_{2,1}$ is $(V^{\otimes 4})^{(P\times Q,\chi)}$, where the  component $(P\times Q,\chi)$ means taking the isotypic component with respect to the idempotent $\frac{1}{|P|\cdot |Q|}\sum_{g\in P,h\in Q}\chi(g,h)g\cdot h$ in $\bb{Z}[S_4]$.

Let $N(V_{2,1})$ and $E(V_{2,1})$ be the corresponding nilpotent endomorphisms on $V_{2,1}$ induced from $N$ and $E$ in \eqref{eq:connection}. The associated Kloosterman connection $\mathrm{Kl}_{\rr{SL}_3}^{V_{2,1}}$ is 
	\begin{equation*}
		\begin{split}
			\Bigl(\cc{O}_{\bb{G}_m}^{15}, d-N(V_{2,1})\frac{\mathrm{d}z}{z}-E(V_{2,1})\mathrm{d}z\Bigr).
		\end{split}
	\end{equation*} 
As in the results from Sections~\ref{subsec:local-at-0} and~\ref{subsec:local-at-infty}, we can show that:
\begin{itemize}
	\item the formal structure of $\mathrm{Kl}_{\rr{SL}_3}^{V_{2,1}}$ at $0$ is isomorphic to 
		\begin{equation*}
			\begin{split}
				\Bigl(\mathcal{O}_{\bb{G}_m}^{7},\mathrm{d}-J_{7}(0)\frac{\mathrm{d}t}{t}\Bigr) \oplus \Bigl(\mathcal{O}_{\bb{G}_m}^{5},\mathrm{d}-J_{5}(0)\frac{\mathrm{d}t}{t}\Bigr) \oplus \Bigl(\mathcal{O}_{\bb{G}_m}^{3},\mathrm{d}-J_{3}(0)\frac{\mathrm{d}t}{t}\Bigr),
			\end{split}
		\end{equation*}
	\item $\mathrm{Kl}_{\rr{SL}_3}^{V_{2,1}}$ has slope $1/3$ at $\infty$, and $\rr{Irr}(\mathrm{Kl}_{\rr{SL}_3}^{V_{2,1}})=5$.
\end{itemize}

The degree on $\Kl_3^{\otimes 4}$ from \eqref{eq:grading} induces a degree on $\mathrm{Kl}_{\rr{SL}_3}^{V_{2,1}}$. As in Theorem~\ref{thm:cohomology_classes}, we have the following:
\begin{prop}\label{prop:coh-classes-V21}
	There exists finite subset $W_{d}, W_{d}^{\rr{mid}}\subset G^+$ such that
		\begin{itemize}
			\item the elements in $W_{d,k}$ and $W_{d,k}^{\rr{mid}}$ have degree $d$,
			\item the de Rham cohomology $\mathrm{H}^1_{\mathrm{dR}}(\bb{G}_m,\mathrm{Kl}_{\rr{SL}_3}^{V_{2,1}})$ is spanned by $W_1\bigcup W_2\bigcup W_3\bigcup W_4\bigcup W_5$ with $\#W_1=\#W_2=\#W_3=\#W_4=\#W_5=1$,
			\item the middle de Rham cohomology $\mathrm{H}^1_{\mathrm{dR,mid}}(\bb{G}_m,\mathrm{Kl}_{\rr{SL}_3}^{V_{2,1}})$ is spanned by 
			$W^{\rr{mid}}_4\bigcup W_5^{\rr{mid}}$
		with $\#W_4^{\rr{mid}}=\#W_5^{\rr{mid}}=1$.
		\end{itemize}
\end{prop}

\section{Calculation of Hodge numbers}\label{sec:hodge number}

In this section, we prove Theorems~\ref{thm:hodge number-Kl} and~\ref{thm:hodge number-Ai} using irregular Hodge filtrations.
 
\subsection{The case of \texorpdfstring{$\Sym^k\Kl_{n+1}$}{SymkKln+1} when \texorpdfstring{$\gcd(k,n+1)=1$}{gcd(k n+1)=1}}\label{sec:Hodge-symkkln+1}
Here, we prove the first part of Theorem~\ref{thm:hodge number-Kl}. We use the irregular Hodge filtrations on twisted de Rham cohomologies from \cite{yu2012irregular}, see Section~\ref{subsec:irreg-hodge-fil}, to do the concrete calculation.
The Kloosterman connection $\Kl_{n+1}$, as a $\bb{C}[z,z\inv]$-module, equals to the cokernel of the complex
	$$\oplus_{i=1}^n\bb{C}[x_i^\pm,z^\pm]\mathrm{d}z\wedge\mathrm{d}x_1 \wedge \cdots \wedge \widehat{\rr{d}x_i}\wedge \cdots \wedge \rr{d}x_n 
	\xrightarrow{\mathrm{d}+\sum_{i=1}^n\partial_{x_i}f\mathrm{d}x_i\wedge}\bb{C}[x_i^\pm,z^\pm]\mathrm{d}z\wedge\mathrm{d}x_1  \wedge \cdots \wedge \mathrm{d}x_n,$$
where $f$ is the Laurent polynomial in \eqref{eq:f-kloosterman}.

In this way, each $v_i$ is identified with $(z\partial_z)^i(\frac{\mathrm{d}z}{z}\frac{\mathrm{d}x_1}{x_1} \cdots  \frac{\mathrm{d}x_n}{x_n})$ in $\Hdr{n+1}(\bb{G}_m^{n+1},f)$, because $\frac{\mathrm{d}z}{z}\frac{\mathrm{d}x_1}{x_1} \cdots  \frac{\mathrm{d}x_n}{x_n}$ satisfies the Bessel differential equation. In fact, we have
	\begin{equation*}
		\begin{split}
			(z\partial_z)^i \frac{\mathrm{d}z}{z}\frac{\mathrm{d}x_1}{x_1} \cdots  \frac{\mathrm{d}x_n}{x_n}= \prod_{j=1}^{i-1}x_j \frac{\mathrm{d}z}{z}\frac{\mathrm{d}x_1}{x_1} \cdots  \frac{\mathrm{d}x_n}{x_n},
		\end{split}
	\end{equation*}
and
	\begin{equation*}
		\begin{split}
			(z\partial_z)^{n+1}\Bigl(\frac{\mathrm{d}z}{z} \frac{\mathrm{d}x_1}{x_1} \cdots  \frac{\mathrm{d}x_n}{x_n}\Bigr)=z\cdot \frac{\mathrm{d}z}{z} \frac{\mathrm{d}x_1}{x_1}\cdots   \frac{\mathrm{d}x_n}{x_n}.
		\end{split}
	\end{equation*}

Let $f_k$ be the Laurent polynomial in \eqref{eq:fk-Kloosterman}. Then, we have morphisms of exponential mixed Hodge structures
	\begin{equation}\label{eq:explicite-cohomology-class}
		\begin{split}
		\mathrm{H}^1_{\mathrm{mid}}(\bb{G}_m,\Sym^k\Kl_{n+1})^{\rr{H}}\hookrightarrow &\mathrm{H}^1(\bb{G}_m,\Sym^k\Kl_{n+1})^{\rr{H}}\\
		\stackrel{\text{Def.~\ref{eq:EMHS-Kl}}}{=}\bigl(&\mathrm{H}^{nk+1}(\bb{G}_m^{nk+1},f_k)^{\rr{H}}\bigr)^{S_k,\chi_n} \hookrightarrow \rr{H}^{nk+1}(\bb{G}_m^{nk+1},f_k)^{\rr{H}}.
		\end{split}
	\end{equation}
Using the differential form expression of $v_i$ above, we notice that an element $z^jv_0^{\otimes I_0}\otimes \cdots \otimes v_n^{\otimes I_n}$ in $ \mathrm{H}^1_{\rr{dR}}(\bb{G}_m,\Kl_{n+1}^{\otimes k})$ is sent via \eqref{eq:explicite-cohomology-class} to
\begin{equation*}
    \begin{split}
        z^j 
		\prod_{i=I_1+1}^{k}x_{i,1} 
		\prod_{i=I_2+1}^k x_{i,2}
		\cdots \prod_{i=I_{n}+1}^k x_{i,n}
		\frac{\mathrm{d}z}{z}\frac{\mathrm{d}x_{1,1}}{x_{1,1}}\cdots\frac{\mathrm{d}x_{k,n}}{x_{k,n}}
  \end{split}
	\end{equation*}
in $\Hdr{nk+1}(\bb{G}_m^{nk+1},f_k)$. Then, the average element $z^j v^{\underline I}$ in $ \mathrm{H}^1_{\rr{dR}}(\bb{G}_m,\Sym^k\Kl_{n+1})$ is sent to 
	\begin{equation}\label{eq:explicite-form}
    \begin{split}
        &\frac{1}{k!}\sum_{\sigma\in S_k}\chi_{n}(\sigma) \cdot z^j
		\prod_{i=I_1+1}^{k}x_{\sigma(i),1} 
		\prod_{i=I_2+1}^k x_{\sigma(i),2}
		\cdots \prod_{i=I_{n}+1}^k x_{\sigma(i),n}
		\frac{\mathrm{d}z}{z}\frac{\mathrm{d}x_{\sigma(1),1}}{x_{\sigma(1),1}}\cdots\frac{\mathrm{d}x_{\sigma(k),n}}{x_{\sigma(k),n}}\\
        =&\frac{1}{k!}\sum_{\sigma\in S_k}z^j
		\prod_{i=I_1+1}^{k}x_{\sigma(i),1} 
		\prod_{i=I_2+1}^k x_{\sigma(i),2}
		\cdots \prod_{i=I_{n}+1}^k x_{\sigma(i),n}
		\frac{\mathrm{d}z}{z}\frac{\mathrm{d}x_{1,1}}{x_{1,1}}\cdots\frac{\mathrm{d}x_{k,n}}{x_{k,n}}
  \end{split}
	\end{equation}
in $\Hdr{nk+1}(\bb{G}_m^{nk+1},f_k)^{S_k,\chi_n}$. So each element $w\in W_{d,k}$ from Theorem~\ref{thm:cohomology_classes} is sent to 
	\begin{equation}\label{eq:g(w)-kloosterman}
		\begin{split}
			g(w)\frac{\mathrm{d}z}{z}\frac{\mathrm{d}x_{1,1}}{x_{1,1}}\cdots\frac{\mathrm{d}x_{1,n}}{x_{1,n}}\cdots\frac{\mathrm{d}x_{k,1}}{x_{k,1}}\cdots\frac{\mathrm{d}x_{k,n}}{x_{k,n}}
		\end{split}
	\end{equation}
for a polynomial $g(w)$ in $z, x_{i,j}$, such that each monomial appearing in $g(w)$ has degree $d$ with respect to the degree \eqref{eq:grading}. By abuse of notation, we still denote by $w$ the image of $w$ under \eqref{eq:explicite-cohomology-class}.

\begin{lem}\label{prop:hodge-number-non-degenerate-newton}
	Assume $\gcd(k,n+1)=1$. Then $ W_{d,k}\subset F^p\Hdr{nk+1}(\bb{G}_m^{nk+1},f_k)$ if $p\leq nk+1-d$.
\end{lem}
\begin{proof}
The Newton polytope $\Delta(f_k)$ defined by $f_k$ has only one facet that does not contain the origin. This facet lies on the hyperplane $(n+1)\alpha+\sum_{i,j}\beta_{i,j}=1$. So the condition that $f_k$ is non-degenerate with respect to $\Delta(f_k)$ is equivalent to the condition that the Laurent polynomial $f_k$ has no critical points in $\bb{G}_m^{nk+1}$. We can check that the latter condition holds when $\gcd(k,n+1)=1$. Hence, the irregular Hodge filtration on $\mathrm{H}^{nk+1}_{\mathrm{dR}}(\bb{G}_m^{nk+1},f_k)$ can be computed via the Newton filtration on monomials in $\bb{R}_{\geq 0}\Delta(f_k)$ defined in \eqref{defn:newton-monomial-fil}.

The cone $\bb{R}_{\geq 0}\Delta(f_k)$ is given by inequalities
	\begin{equation*}
		\begin{split}
			\alpha+\sum_{i=1}^k\sum_{j=1}^n \epsilon_{i,j} \beta_{i,j}\geq 0, \quad \epsilon_{i,j}\in \{0,1\},
		\end{split}
	\end{equation*}
and $\sum_{j=1}^n\epsilon_{i,j} \leq 1$ for each $i$. We take the fan $F$ generated by rays
	\begin{equation*}
		\begin{split}
			\bb{R}_{\geq 0}\cdot \pm \Bigl(\alpha+\sum_{i=1}^k\sum_{j=1}^n \epsilon_{i,j} \beta_{i,j}\Bigr) \quad \text{and} \quad  \bb{R}_{\geq 0}\pm \Bigl((n+1)\alpha+\sum_{i,j} \beta_{i,j}\Bigr)
		\end{split}
	\end{equation*}
where $\epsilon_{i,j}\in \{0,1\}$. We can check that the simplicial polytopal fan $F$ is regular. So the corresponding toric variety $X_{\tor}$ is smooth projective. In particular, each irreducible component of the pole divisor $P$ of $f_k$ has multiplicity $1$.

Observe that the class of a form 
	\begin{equation*}
		\begin{split}
			z^a\prod_{i=1}^{k}\prod_{j=1}^nx_{i,j}^{ b_{i,j}} \cdot \frac{\mathrm{d}z}{z}\frac{\mathrm{d}x_{1,1}}{x_{1,1}}\cdots\frac{\mathrm{d}x_{1,n}}{x_{1,n}}\cdots\frac{\mathrm{d}x_{k,1}}{x_{k,1}}\cdots\frac{\mathrm{d}x_{k,n}}{x_{k,n}}
		\end{split}
	\end{equation*}
is in $F^{p}\mathrm{H}^{nk+1}_{\dR}(\bb{G}^{nk+1}_m, f_k)$ if this form belongs to $\Omega^{nk+1}(\log S)(\lfloor(nk+1-p)P\rfloor)$, which is equivalent to
		\begin{equation}\label{eq:order}
		\mathrm{ord}_{D}\Bigl(z^a\prod_{i=1}^{k}\prod_{j=1}^nx_{i,j}^{b_{i,j}}\Bigr)\geq -(nk+1-p)
		\end{equation}
for all irreducible components $D$ of $P$. By \cite[p.61]{fulton1993introduction}, 
the condition \eqref{eq:order} is equivalent to 
		\begin{equation*}
			\begin{split}
				-\xi\Bigl(z^a\prod_{i=1}^{k}\prod_{j=1}^nx_{i,j}^{b_{i,j}}\Bigr)\geq -(nk+1-p),
			\end{split}
		\end{equation*}
where $\xi\bigl(z^a\prod_{i=1}^{k}\prod_{j=1}^nx_{i,j}^{b_{i,j}}\bigr)=(n+1)a+\sum_{i,j}b_{i,j}$.

	For $w\in W_{d,k}$, the value of $\xi$ at each monomial appearing in $g(w)$ in \eqref{eq:g(w)-kloosterman} is $d$. Therefore, we conclude the lemma.
\end{proof}

\begin{thm}[Theorem~\ref{thm:hodge number-Kl-1}]\label{thm:hodgenumberskl3-3nmidk}
Assume that $\gcd(k,n+1)=1$. The Hodge numbers of the pure Hodge structure $\mathrm{H}^1_{\mathrm{mid}}(\bb{G}_m,\Sym^k\Kl_{n+1})^{\rr{H}}$ are given by $h^{p,nk+1-p}=\#W^{\rr{mid}}_{p,k},$ where $W^{\rr{mid}}_{p,k}$ are the sets from Theorem~\ref{thm:mid-cohomology_classes}.
\end{thm}

\begin{proof}
We construct an auxiliary filtration $G^\bullet $ on $\mathrm{H}^1_{\mathrm{dR,mid}}(\bb{G}_m,\Sym^k\Kl_{n+1})$ by letting the subspace $G^p$ be generated by elements $ w\in W_{d,k}^{\rr{mid}}$ such that $p\leq nk+1-d$. Lemma~\ref{prop:hodge-number-non-degenerate-newton} shows that 	
	\begin{equation*}
		\begin{split}
			G^p\mathrm{H}^1_{\mathrm{dR,mid}}(\bb{G}_m,\Sym^k\Kl_{n+1})\subset F^p\mathrm{H}^1_{\mathrm{dR,mid}}(\bb{G}_m,\Sym^k\Kl_{n+1}).
		\end{split}
	\end{equation*}
Let $d_p$ and $\delta_p$ be the dimensions of the graded quotients with respect to the two filtrations $F^\bullet$ and $G^\bullet$ on $\mathrm{H}^1_{\mathrm{dR, mid}}(\bb{G}_m,\Sym^k\Kl_{n+1})$ respectively. Then we have inequalities
	\begin{equation}\label{eq:dp-geq-deltap}
		\sum_{q\geq p}\delta_p \leq \sum_{q\geq p}d_p,
	\end{equation}
for $0\leq q\leq nk+1$, and the equality holds if $q=0$ and $q=nk+1$. By the Hodge symmetry, we have $d_p=d_{nk+1-p}$. Since $\delta_p$ are the numbers $\#W_{nk+1-p,k}^{\rr{mid}}$ in Theorem~\ref{thm:mid-cohomology_classes}, we also have $\delta_p=\delta_{nk+1-p}$.

Combining \eqref{eq:dp-geq-deltap} and the symmetry properties of $d_p$ and $\delta_p$, we find
	\begin{equation*}
		\begin{split}
			\sum_{q\geq p}\delta_p\leq \sum_{q\geq p}d_p=\sum_{nk+1-q\leq p}d_p\leq \sum_{nk+1-q\leq p}\delta_p=\sum_{q\geq p}\delta_p,
		\end{split}
	\end{equation*}
which implies that $\delta_p=d_p$ for each $p$. 
\end{proof}

\begin{exe}\label{exe:n=2-hodge}
	Assume that $3\nmid k$. The Hodge numbers of $\mathrm{H}^1_{\mathrm{mid}}(\bb{G}_m,\Sym^k\Kl_{3})^{\rr{H}}$ are given by
			\begin{equation*}
				\begin{split}
					h^{p,2k+1-p}=\begin{cases}
						\lfloor\frac{p}{6}\rfloor & p\not\equiv 3,5\mod 6;\\[4pt]
						\lfloor\frac{p}{6}\rfloor+1 & \text{else},
						\end{cases}
				\end{split}
			\end{equation*}
	if $p\leq k$, and $h^{p,2k+1-p}=h^{2k+1-p,p}$ if $k>p$.
\end{exe}

\subsection{The case of \texorpdfstring{$\Sym^k\Kl_3$}{SymkKl3} when \texorpdfstring{$3\mid k$}{3|k}}
We compute the Hodge numbers of $\mathrm{H}^1_{\rr{mid}}(\bb{G}_m,\Sym^k\Kl_3)^{\rr{H}}$ here (the case when $3\nmid k$ is provided by Example~\ref{exe:n=2-hodge}) in the following theorem, which completes the second part of Theorem~\ref{thm:hodge number-Kl}.

\begin{thm}
	Assume that $3\mid k$. The Hodge numbers of the mixed Hodge structure $\mathrm{H}^1(\bb{G}_m,\Sym^k\Kl_3)^{\rr{H}}$ are
	\begin{equation*}
		\begin{split}
			h^{p,q}=\begin{cases}
				\lfloor\frac{\min(p,q)}{6}\rfloor -(\delta_{p,k}+\delta_{p,k+1})& p\equiv 0,1,2,4\ \mathrm{mod} \ 6,\, p+q=2k+1,
				\\
				\lfloor\frac{\min(p,q)}{6}\rfloor+1 -(\delta_{p,k}+\delta_{p,k+1})& p\equiv 3,5\ \mathrm{mod} \ 6,\, p+q=2k+1,
				\\
				(1+2\lfloor\frac{k}{2}\rfloor-k) + 1 & p=q=k+1,
				\\
				1 & k+2\leq p=q \leq 2k+1,\,2\nmid p \\
				 0 &\text{else},
				\end{cases}
		\end{split}
	\end{equation*}	 
	where $\delta_{a,b}$ are the Kronecker symbols. In particular, the first and the second line of the formula give the Hodge numbers of the pure Hodge structure $\mathrm{H}^1_{\mathrm{mid}}(\bb{G}_m,\Sym^k\Kl_3)^{\rr{H}}$.
\end{thm}
\begin{rek}\label{rek:hodge-number-nonpure-part} 
	The first two lines of the formula for Hodge numbers are calculated in Theorem~\ref{thm:hodgenumberskl3-3midk}. The rest three lines of the formula for Hodge numbers are deduced from the second and the third lines of the formula for Hodge numbers of $\mathrm{H}^1(\bb{G}_m,\Sym^k\widetilde{\Kl}_3)^{\rr{H}}$ in Proposition~\ref{prop:hodgenumberstildekl3-3midk} because 
		$\mathrm{H}^1(\bb{G}_m,\Sym^k\Kl_3)^{\rr{H}}/\mathrm{H}^1_{\rr{mid}}(\bb{G}_m,\Sym^k\Kl_3)^{\rr{H}}$ 
	is isomorphic to 
		$\mathrm{H}^1(\bb{G}_m,\Sym^k\widetilde{\Kl}_3)^{\rr{H}}/ \mathrm{H}^1_{\rr{mid}}(\bb{G}_m,\Sym^k\widetilde{\Kl}_3)^{\rr{H}}$ 
	by Remark~\ref{rek:inv-0-infty-kl3}. 
\end{rek}

\subsubsection{The inverse Fourier transform}\label{subsec:inverse-fourier}

Let $j_0\colon \bb{G}_m\hookrightarrow \bb{A}^1$ be the inclusion. Recall that we defined an endofunctor $\Pi$ on the category of regular holonomic $\mathscr{D}_{\bb{A}^1}$-modules in Section~\ref{subsec:emhs}, which can be lifted to the projector $\Pi$ in \eqref{eq:projector}. Recall that the Fourier transform $\mathrm{FT}$ sends $\mathscr{D}_{\bb{A}^1_t}$-modules on the affine line $\bb{A}^1_t$ to $\mathscr{D}_{\bb{A}^1_\tau}$-modules on the dual affine line $\bb{A}^1_\tau$. In the proof of \cite[Prop.\,2.10]{fresan2018hodge}, we have an isomorphism of $\mathscr{D}$-modules
	\begin{equation*}
		\begin{split}
			\Sym^{k}\widetilde\Kl_{3}\simeq j_0^+\mathrm{FT}(g_{k+}\mathcal{O})^{S_k},
		\end{split}
	\end{equation*}
where $g_{k}\colon \bb{G}_m^{2k}\to \bb{A}^1$ is defined by sending $(y_{i,j})$ to $\sum_{i}(\sum_jy_{i,j}+\frac{1}{\prod_j y_{i,j}})$, and the action of the symmetric group $S_k$ on the coordinates is given by $\sigma\cdot y_{i,j}=y_{\sigma(i),j}$. By the definition of the intermediate extension, we have a short exact sequence
	\begin{equation*}
	\begin{split}
		0\to j_{0\dagger+}\Sym^k\widetilde\Kl_{3}\to j_{0+}\Sym^k\widetilde\Kl_{3}\to \widetilde{C}_{k}\to 0,
	\end{split}
	\end{equation*}
where $\widetilde C_{k}$ is supported at the origin. Let $\widetilde M=\mathrm{FT}\inv j_{0\dagger+}\Sym^k\widetilde\Kl_{3}$.  By the isomorphism of functors,
	\begin{equation*}
		\begin{split}
			\Pi\circ \mathrm{FT}\inv\simeq \mathrm{FT}\inv \circ j_{0+}j_0^+
		\end{split}
	\end{equation*}
\cite[Prop.\,12.3.5]{katz1990exponential}, we deduce that $\Pi(\widetilde{M})\simeq \rr{FT}\inv(j_{0+}\Sym^k\widetilde{\Kl}_3)$.  We get from the above an exact sequence of regular holonomic $\mathscr{D}$-modules on the dual affine line $\bb{A}^1_{\tau}$
	\begin{equation}\label{eq:short-exact-inverse-fourier}
	\begin{split}
		0\to \widetilde M \to \Pi(\widetilde M)\to \widetilde M'\to 0.
	\end{split}
	\end{equation}
The sequence \eqref{eq:short-exact-inverse-fourier} can also be lifted to a short exact sequence in $\mathrm{MHM}(\bb{A}^1_\tau)$
	\begin{equation*}
		\begin{split}
			0\to \widetilde M^{\mathrm{H}} \to \Pi(\widetilde M)^{\mathrm{H}}\to \widetilde M^{'\mathrm{H}}\to 0,
		\end{split}
	\end{equation*}
where $\widetilde M^{\mathrm{H}}$ is a pure Hodge module of weight $2k$, and $\widetilde M^{'\mathrm{H}}$ is a constant mixed Hodge module of weight $\geq 2k+1$, see \cite[Prop.\,2.21]{fresan2018hodge}. Moreover, the structure of $\widetilde M'$ can be made precise, as stated in Corollary~\ref{cor:extension-local-monodromy}.
\begin{prop}\label{prop:monodromy-nearby-0}
Let $S\subset \bb{A}^1$ be the set of points $x$ such that there exists a triple $\underline I\in \bb{N}^3$ such that $x=3C_{\underline I}=3(I_0+\zeta I_1+\zeta^2I_2)$ and $|\underline I|=k$. Then:
\begin{enumerate}[label=(\theenumi)., ref= \ref{prop:monodromy-nearby-0}.(\theenumi)]
	\item The inverse Fourier transform $\widetilde M$ is a regular holonomic $\mathscr{D}_{\bb{A}^1_\tau}$-module. Its generic rank is $\frac{(k+1)(k+2)}{2}-\lfloor\frac{k+2}{2}\rfloor$ and its singularity set is $S$. The vanishing cycle space at a singularity $x\in S$ has dimension $\#\{\underline I\mid x=3C_{\underline I},\ |\underline I|=k\}$ with trivial monodromy. Moreover, the monodromy at $\infty$ is unipotent, with one Jordan block of size $2k-4m$ for each $m=0,\ldots,\lfloor\frac{k}{2}\rfloor$;
	\item $\Pi(\widetilde M)$ is a regular holonomic $\mathscr{D}_{\bb{A}^1_\tau}$-module. Its generic rank is $\frac{(k+1)(k+2)}{2}$ and its singularity set is $S$. The vanishing cycle space at a singularity $x\in S$ has dimension $\#\{\underline I\mid x=3C_{\underline I}, \ |\underline I|=k\}$ with trivial monodromy. Moreover, the monodromy at $\infty$ is unipotent, with one Jordan block of size $2k-4m+1$ for each $m=0,\ldots,\lfloor\frac{k}{2}\rfloor$;
\end{enumerate}
\end{prop}
\begin{proof}
Let $\psi_t\,\Sym^k\widetilde\Kl_3$ be the nearby cycle module at $0$ of $j_{0+}\Sym^k\widetilde\Kl_3$. Recall that in Section~\ref{subsec:local-at-0}, we enhanced the action of the nilpotent part $N_k$ of the local monodromy on $\psi_t\,\Sym^k\widetilde\Kl_3$ to an $\mathrm{sl}_2$-action. From Corollary~\ref{prop:local-formal-at-0}, we conclude that the monodromy of $\psi_t\,\Sym^k\widetilde{\Kl}_3$ is unipotent, with one Jordan block of size $2k-4m+1$ for each $m\in \{0, \ldots,\lfloor\frac{k}{2}\rfloor\}$. By \cite[Lem.\,5.1.4]{Saito1988}, the vanishing cycle module $\phi_tj_{0\dagger+}\Sym^k\widetilde{\Kl}_3$ of the intermediate extension $j_{0\dagger+}\Sym^k\widetilde\Kl_3$ is identified with $\im\, N_k$. So the monodromy action of $\phi_tj_{0\dagger+}\Sym^k\widetilde{\Kl}_3$ is unipotent, with one Jordan block of size $2k-4m$ for each $m\in \{0,\ldots,\lfloor\frac{k}{2}\rfloor\}$.  

By construction, we have $\widetilde{M}=\mathrm{FT}\inv j_{0\dagger+}\Sym^k\widetilde{\Kl}_3$. Applying the (inverse) stationary phase formula \cite{sabbah2008explicit} to $\widetilde M$, it follows that the nearby cycle $\,\psi_{1/\tau}\widetilde M$ is isomorphic to $\phi_tj_{0\dagger+}\Sym^k\widetilde{\Kl}_3\simeq \im\,N_k$. Therefore, the monodromy of $\psi_{1/\tau}\widetilde M$ is unipotent, with one Jordan block of size $2k-4m$ for each $0\leq m\leq \lfloor\frac{k}{2}\rfloor$. The primitive parts of the Lefschetz decomposition of $\,\psi_{1/\tau}\widetilde M$ are
	\begin{equation*}
		\begin{split}
			\mathrm{P}_{2k-1},\mathrm{P}_{2k-5},\ldots,
		\end{split}
	\end{equation*}
and each $\mathrm{P}_{2k-1-4m}$ is $1$-dimensional. By the property of the Lefschetz decomposition, the rank of $\widetilde M$ is 
	\begin{equation*}
		\begin{split}
			\sum_{m=0}^{\lfloor\frac{k-1}{2}\rfloor}(2k-4m)=\frac{(k+1)(k+2)}{2}-\left\lfloor\frac{k+2}{2}\right\rfloor.
		\end{split}
	\end{equation*}

In the end, by applying the stationary phase formula to $j_{0\dagger+}\Sym^k\widetilde{\Kl}_3=\mathrm{FT}(\widetilde M)$, we conclude from the local structure of $j_{0\dagger+}\Sym^k\widetilde{\Kl}_3$ at $\infty$ in Proposition~\ref{prop:local-infinity-Symk} that the singular set of $\widetilde M$ is equal to $S$. In fact, each singular point $a\in S$ corresponds to a component $\mathcal{E}^{at}$ in the local structure of $j_{0\dagger+}\Sym^k\widetilde{\Kl}_3$ at $\infty$.  So we proved all claims in the first statement. The proof of the second statement is similar.
\end{proof}

\begin{cor}\label{cor:extension-local-monodromy}
The graded pieces of $\widetilde M'^{\mathrm{H}}$ with respect to the weight filtration have rank $1$ and are of weight $4k-4m$ for each $m=0,1,\ldots,\lfloor\frac{k}{2}\rfloor$.
\end{cor}

\begin{proof}
By applying the nearby cycle functor to the exact sequence \eqref{eq:short-exact-inverse-fourier}, and using the rank formula of $\widetilde M^{\mathrm{H}}$ and $\Pi\,\widetilde M^{\mathrm{H}}$ in Proposition~\ref{prop:monodromy-nearby-0}, we know that $\widetilde M^{'\mathrm{H}}$ has rank $\left\lfloor\frac{k+2}{2}\right\rfloor$. 

To prove the weight property of $\widetilde M'^{\mathrm{H}}$, it suffices to compute that of $\psi_{1/\tau}\,\widetilde M^{'\mathrm{H}}$, because the constant mixed Hodge module $\widetilde M^{'\mathrm{H}}$ extends smoothly at $\infty$. By a similar argument in the proof of Proposition~\ref{prop:monodromy-nearby-0}, we know that the primitive parts in the Lefschetz decomposition of $\psi_{1/\tau}\Pi\,\widetilde M^{\mathrm{H}}$ are
	\begin{equation*}
		\begin{split}
			\mathrm{P}_{2k+1},\mathrm{P}_{2k-3},\ldots,
		\end{split}
	\end{equation*}
all being $1$-dimensional. Since the weight filtration on $\psi_{1/\tau}\Pi\,\widetilde M^{\mathrm{H}}$ is identified with the monodromy filtration centered at $2k-1$ \cite[p.\,1740]{2013Monodromy-sabbah-appendix}, it follows that $\mathrm{P}_{2k+1-4m}$ is pure of weight $4k-4m$, of rank $1$. At last, by Equation~(A.2) in the proof of \cite[Prop.\,A.3]{2013Monodromy-sabbah-appendix}, the graded pieces of $\psi_{1/\tau}\,\widetilde M^{'\mathrm{H}}$ are identified with the primitive parts in the Lefschetz decomposition of $\psi_{1/\tau}\Pi\,\widetilde M^{\mathrm{H}}$. Therefore, the graded pieces $\gr^W_{4k-4m}\widetilde M'^{\mathrm{H}}$ with respect to the weight filtration are $P_{2k-1-4m}(-1)$, of rank $1$ and weight $4k-4m$.
\end{proof}

\subsubsection{The Hodge filtration for \texorpdfstring{$\Sym^k\widetilde{\Kl}_3$}{}}\label{subsec:hdogekl3tilde}
Before computing the Hodge filtration on the mixed Hodge structure $\mathrm{H}^1(\bb{G}_m,\Sym^k\Kl_3)^{\rr{H}}$, we calculate that on $\mathrm{H}^1(\bb{G}_m,\Sym^k\widetilde\Kl_3)^{\rr{H}}$, which is useful in \S\ref{subsec:hodgenumberskl3-3midk}. 

We define\footnote{We use \eqref{eq:emhs-fiber-of-fourier} as the definition of $\mathrm{H}^1(\bb{A}^1, j_{0\dagger+} \Sym^k\widetilde\Kl_3)^{\rr{H}}$ to avoid using the notation from \cite[Not.\,A.\,29]{fresan2018hodge}.} the mixed Hodge structure $\mathrm{H}^1(\bb{A}^1, j_{0\dagger+} \Sym^k\widetilde\Kl_3)^{\rr{H}}$ by
	\begin{equation}\label{eq:emhs-fiber-of-fourier}
	\coker(N_k\colon \psi_{\tau,1}\widetilde M^{\mathrm{H}}\to \psi_{\tau,1}\widetilde M^{\mathrm{H}}(-1)).
	\end{equation}
Applying \cite[Exe.\,A.2 \& Thm.\,A.30]{fresan2018hodge} to the short exact sequence \eqref{eq:short-exact-inverse-fourier}, we have a short exact sequence of mixed Hodge structures
	\begin{equation}\label{eq:short-exact-mhs}
		0\to \mathrm{H}^1(\bb{A}^1, j_{0\dagger+} \Sym^k\widetilde\Kl_3)^{\rr{H}}\xrightarrow{\iota} \mathrm{H}^1(\bb{G}_m,  \Sym^k\widetilde\Kl_3)^{\rr{H}}\to V^{\mathrm{H}}\to 0,
	\end{equation}
where $V^{\mathrm{H}}$ is the cokernel of $\iota$. By applying \cite[Cor.\,A31]{fresan2018hodge} to the pure Hodge module $\widetilde M^{\mathrm{H}}$, we know that $V^{\mathrm{H}}$ is mixed of weight $\geq 2k+2$, and satisfies
	\begin{equation}\label{eq:hodgenumber-M'}
		\dim \mathrm{gr}^p_{F}\mathrm{gr}^W_\ell V^{\mathrm{H}}=\rk\, \mathrm{gr}_F^{p-1}\mathrm{gr}^W_{\ell-1}\widetilde M^{'\mathrm{H}} 
	\end{equation}
for all $p,\ell\in \bb{Z}$. In particular, by the above and \cite[Prop.\,2.21]{fresan2018hodge}, the weights of the mixed Hodge structure $\mathrm{H}^1(\bb{A}^1, j_{0\dagger+} \Sym^k\widetilde\Kl_3)^{\rr{H}}$ are at least $2k+1$, and we have 
	\begin{equation*}
		\begin{split}
			\mathrm{H}^1_{\mathrm{mid}}(\bb{G}_m, \Sym^k\widetilde\Kl_3)^{\rr{H}}=\mathrm{gr}^W_{2k+1}\mathrm{H}^1(\bb{A}^1, j_{0\dagger+} \Sym^k\widetilde\Kl_3)^{\rr{H}}.
		\end{split}
	\end{equation*}

\begin{prop}\label{prop:hodgenumberstildekl3-3midk}
 The Hodge numbers of $\mathrm{H}^1(\bb{G}_m, \Sym^k\widetilde\Kl_3)^{\rr{H}}$ are given by 
	\begin{equation*}
		\begin{split}
			h^{p,q}=\begin{cases}
				\left\lfloor\frac{\min\{p,q\}+1}{2}\right\rfloor-(\delta_{p,k}+\delta_{p,k+1})\cdot d(k,3)   & p+q=2k+1\\
				(1+2\lfloor\frac{k}{2}\rfloor-k) + d(k,3) & p=q=k+1\\
				1 & k+2\leq p=q \leq 2k+1,\,2\nmid  p \\
				0 &\text{else}
			\end{cases}
		\end{split}
	\end{equation*}
where $d(k,3)$ is $1$ if $3\mid k$ and $0$ otherwise, and $\delta_{a,b}$ is the Kronecker symbol, which is $1$ if $a= b$ and $0$ otherwise. In particular, the first line of the formula for Hodge numbers gives the Hodge numbers of $\mathrm{H}^1_{\mathrm{mid}}(\bb{G}_m, \Sym^k\widetilde\Kl_3)^{\rr{H}}$.
\end{prop}

\begin{proof}We proceed in three steps.\\
\textbf{Step $1$:} We compute the rank of the Hodge filtration on $\widetilde M^H$. Using a similar argument as in the proof of Proposition~\ref{prop:monodromy-nearby-0}, the monodromy of the nearby cycle $\psi_{1/\tau}\widetilde M^{\mathrm{H}}$ at $\infty$ is unipotent, with one Jordan block of size $2k-4m$ for each $m=0,\ldots,\lfloor\frac{k}{2}\rfloor$. Therefore, the primitive parts $P_{2k-4m-1}$ in the Lefschetz decomposition have dimension one and are of Hodge--Tate type. Since $\widetilde M$ is pure of weight $2k$, the primitive part $P_{2k-4m-1}$ is pure of weight $4k-4m-2$. Hence, $P_{2k-4m-1}=\mathrm{gr}^W_{4k-4m-2}P_{2k-4m-1}$. Then by the Lefschetz decomposition, for each $\ell$ in $\bb{Z}$, the graded quotient $\mathrm{gr}^W_{2k+\ell-1}\psi_{1/\tau}\widetilde M^{\mathrm{H}}$ is Hodge--Tate
of dimension 
	\begin{equation*}
		\begin{split}
			\#\{m \mid 0\leq 4m \leq 2k -|\ell| -1\}
		\end{split}
	\end{equation*}
if $2\nmid \ell$, and $0$ if $2\mid \ell$.
By the compatibility of \cite[3.2.1]{Saito1988} between the Hodge filtration and the Kashiwara-Malgrange filtration of the filtered $\mathscr{D}$-modules underlying Hodge modules, we have 
	\begin{equation*}
		\begin{split}
			\rk\,\gr_F^p\widetilde M^{\mathrm{H}}=\dim \gr_F^p \psi_{1/\tau}\widetilde M^{\mathrm{H}}
		\end{split}
	\end{equation*}
in the case of smooth curves. From this formula, we conclude that
\begin{equation}\label{eq:graded quotients Hodge numbers}
\rk\, \mathrm{gr}_F^p\widetilde M^{\mathrm{H}}=\left\lfloor\frac{\min\{p,2k-1-p\}+2}{2}\right\rfloor.
\end{equation}
~\\
\textbf{Step $2$:} We compute the Hodge filtration on $\mathrm{H}^1(\bb{A}^1, j_{0\dagger+} \Sym^k\widetilde\Kl_3)^{\rr{H}}$.

If $3\nmid k$, by Proposition~\ref{prop:monodromy-nearby-0}, $0$ is not a singular point of $\widetilde M$. Therefore, the Hodge numbers can be computed using \cite[Cor.\,A31(ii)]{fresan2018hodge}. More precisely, the nilpotent part of the monodromy operator acts on $\psi_{\tau,1}\widetilde M^{\mathrm{H}}=\psi_{\tau}\widetilde M^{\mathrm{H}}$ by $0$. So 
\begin{equation*}
	\begin{split}
		\mathrm{H}^1(\bb{A}^1, j_{0\dagger+} \Sym^k\widetilde\Kl_3)^{\rr{H}}=\psi_{\tau}\widetilde M^{\mathrm{H}}(-1)
	\end{split}
\end{equation*} 
is pure of weight $2k+1$, which coincides with $\mathrm{H}^1_{\mathrm{mid}}(\bb{G}_m,\Sym^k\widetilde\Kl_3)^{\rr{H}}$. Its Hodge numbers $h^{p,2k+1-p}$ are the ranks of $\mathrm{gr}_F^{p-1}\widetilde M^{\mathrm{H}}$ given by \eqref{eq:graded quotients Hodge numbers}, which give the first line of the formula for the Hodge numbers when $3\nmid k$.

If $3\mid k$, the dimension of the formal regular component of $\Sym^k\widetilde \Kl_3$ at $\infty$ is $1$.  Therefore, $\tau=0$ is a singular point of $\widetilde M$, and the dimension of the vanishing cycle $\,\phi_{\tau,1}\widetilde M^{\mathrm{H}}$ is $1$. By \eqref{eq:emhs-fiber-of-fourier}, it suffices to study the monodromy on the nearby cycle $\psi_{\tau,1}\widetilde M^{\mathrm{H}}$. Since $\widetilde M^{\mathrm{H}}$ is an intermediate extension at $\tau =0$ and $\phi_{\tau,1}\widetilde M^{\mathrm{H}}$ has dimension $1$, the nilpotent part of the monodromy operator, denoted by $N$, acting on $\psi_{\tau,1}\widetilde M^{\mathrm{H}}$ satisfies $N^2=0$. We obtain that the primitive parts of the Lefschetz decomposition of $\gr^W\psi_{\tau,1}\widetilde M^{\mathrm{H}}$ are $P_1=\gr^W_{2k}\psi_{\tau,1}\widetilde M^{\mathrm{H}}$ and $P_0=\gr^W_{2k-1}\psi_{\tau,1}\widetilde M^{\mathrm{H}}$, where $\dim P_1=\dim N P_1=1$, and $\dim P_0=\rk\,\widetilde M^{\rr{H}}-2$. In conclusion, 
	\begin{equation*}
		\begin{split}
			\gr^W\mathrm{coker}N=P_0(-1)\oplus P_1(-1).
		\end{split}
	\end{equation*}
Hence, the direct summand $P_0(-1)$ is  $\mathrm{H}^1_{\mathrm{mid}}(\bb{G}_m,\Sym^k\widetilde \Kl_3)^{\rr{H}}$. Also, the direct summand $P_1(-1)$ is 
	$$\gr^W_{2k+2}\mathrm{H}^1(\bb{G}_m,\Sym^k\widetilde \Kl_3)^{\rr{H}},$$
whose rank is $1$. It is of Hodge--Tate type $(k+1,k+1)$. Then using the equality derived from the Lefschetz decomposition	
	\begin{equation*}
		\begin{split}
			\rk\,  \mathrm{gr}^{p-1}_F\widetilde{M}^{\mathrm{H}}=\rk\, \mathrm{gr}^{p-1}_F\psi_{\tau,1}\,\widetilde{M}^{\mathrm{H}}=\dim \mathrm{gr}^{p}_F(P_0(-1))+\mathrm{gr}^{p}_F(P_1(-1))+\mathrm{gr}^{p}_F(NP_1(-1)),
		\end{split}
	\end{equation*}
we proved the first line of the formula for Hodge numbers when $3\mid k$. Moreover, we show that $d(k,3)$ appears in the second line of the formula for the Hodge numbers.
~\\
\textbf{Step $3$:} We compute the Hodge filtration on $V^{\mathrm{H}}$. By Corollary~\ref{cor:extension-local-monodromy}, since the graded pieces of $\widetilde{M}^{'\mathrm{H}}$ with respect to the weight filtration have rank $1$, of weight $4k-4m$ for $m=0,1,\ldots,\lfloor\frac{k}{2}\rfloor$, they are all of Hodge--Tate type $(2k-2m,2k-2m)$. Thus, by \eqref{eq:hodgenumber-M'}, we have 
\begin{equation*}
	\begin{split}
		\dim \mathrm{gr}^{2k-2m+1}_{F}\mathrm{gr}^W_{4k-4m+2} V^{\mathrm{H}}=1
	\end{split}
\end{equation*}
for $m=0,\ldots,\lfloor\frac{k}{2}\rfloor$ and the remaining graded pieces are $0$. Therefore, we get the second and the third line of the formula for the Hodge numbers by using the exact sequence \eqref{eq:short-exact-mhs}.
\end{proof}

\subsubsection{Proof of Theorem~\ref{thm:hodge number-Kl-2}}\label{subsec:hodgenumberskl3-3midk}
Compared to the method in Section~\ref{sec:Hodge-symkkln+1}, we need to find another way to compute the irregular Hodge filtration when $3\mid k$ because the function $f_k$ is degenerate with respect to its Newton polytope $\Delta(f_k)$. We need to choose a compactification of $(\bb{G}_m^{nk+1},f_k)$ to calculate the irregular Hodge filtration, similar to the one in \cite[\S4.3.2]{fresan2018hodge}.

\paragraph{A compactification of \texorpdfstring{$(\bb{G}_m^{nk+1},f_k)$}{}}\label{par:nondegenerate-compactification}

\begin{defn}[{\cite[Def.\,2.6]{mochizuki2015GKZ}}]\label{defn:non-degenerate}
Let $K$ be a field of characteristic $0$. Let $U$ be a smooth quasi-projective variety over $K$ and $f\in \mathcal{O}_U(U)$. A \textit{non-degenerate compactification} of a pair $(U,f)$ is a compactification $X$ of $U$ such that
\begin{itemize}
	\item $D=X\backslash U$ is a strict normal crossing divisor,
	\item $f$ extends to a rational morphism $f\colon X\dashrightarrow \bb{P}^1$,
	\item \'etale locally or analytical locally near each point in the pole divisor $P$, there is a coordinate system $\{x_1,\ldots,x_r,y_1,\ldots,y_m,z_1\ldots,z_l\}$ of $X$ such that 
		\begin{equation*}
			\begin{split}
				D=V\biggl(\prod_{i=1}^r x_i\cdot \prod_{j=1}^m y_j\biggr), \quad \text{and} \quad  f=\dfrac{1}{\prod_{i=1}^r x_i^{e_i}} \text{ or } \dfrac{z_1}{\prod_{i=1}^r x_i^{e_i}}
			\end{split}
		\end{equation*}
	for some $e_i\in \bb{Z}_{>0}$.
\end{itemize}
\end{defn}

Let $\bb{G}_{m}^{2k}$ be the torus with coordinates $y_{i,j}$ for $1\leq i\leq k$ and $j=1,2$. We begin with the pair $\bigl(\bb{G}_m^{2k},g_k=\sum_{i=1}^k\bigl(y_{i,1}+y_{i,2}+\frac{1}{y_{i,1}y_{i,2}}\bigr)\bigr)$. Let $M=\bigoplus_{i=1}^k\q{ \bb{Z} y_{i,1}\oplus \bb{Z}y_{i,2}}$ be the lattice of monomials on $\bb{G}_m^{2k}$ and $N=\bigoplus_{i=1}^k\q{ \bb{Z} e_{i,1}\oplus \bb{Z}e_{i,2}}$ be the dual lattice where $e_{i,j}$ is dual to $y_{i,j}$. We consider the toric compactification $X$ of $\bb{G}_m^{2k}$ attached to the regular simplicial fan $F$ in $N_{\bb{R}}$ generated by the rays
	\begin{equation*}
		\begin{split}
			\bb{R}_{\geq 0}\cdot \sum_{i=1}^k\epsilon_{i}e_{i,1}+\eta_i e_{i,2}
		\end{split}
	\end{equation*}
where $\epsilon_{i}, \eta_i\in \{0,\pm1\}$ and $(\epsilon_{i},\eta_i)\neq (0,0)$ for at least one $i$. Each simplicial cone of maximal dimension $2k$ in $F$ is regular and provides an affine chart of $X$, which is isomorphic to $\bb{A}^{2k}$. On each chart, the function $g_k$ has the same structure. For example, we can consider the regular cone generated by
	\begin{equation*}
		\begin{split}
			\gamma_{i_0,1}:=\sum_{1\leq i\leq i_0-1,1\leq j\leq 2}e_{i,j} +e_{i_0,1} \quad \text{and} \quad  \gamma_{i_0,2}=\sum_{1\leq i\leq i_0,1\leq j\leq 2}e_{i,j}
		\end{split}
	\end{equation*}
for $1\leq i_0\leq k$, where the affine ring associated with the dual cone is the polynomial ring $\bb{Q}[u_i,v_i]$ such that $u_{i,j}=
	\begin{cases}
		y_{i,1}/y_{i,2} & j=1,\\
		y_{i,2}/y_{i+1,1} & i<k, j=2,\\
		y_{k,2} & i=k,j=2.
	\end{cases}$
In this chart, we rewrite $g_{k}$ as 
	$g_1/\bigl(u_{1,1}u_{1,2}^2\cdot \prod_{2\leq i\leq k,1\leq j\leq 2}u_{2,j}^2\bigr),$
where 
	\begin{equation*}
		\begin{split}
			g_1=1+\sum_{e=1}^{k-1}u_{1,1}u_{1,2}^2
			\cdot \prod_{\substack{2\leq i\leq e,1\leq j\leq 2}}u_{i,j}^2 
			\cdot  u_{e+1,1} 
			+u_{1,1}u_{1,2}^2\cdot \prod_{2\leq i\leq k,1\leq j\leq 2}u_{i,j}^2\cdot h
		\end{split}
	\end{equation*}
for a polynomial $h$. So the toric variety $X$ provides a non-degenerate compactification of $(\bb{G}_m^{2k},g_k)$, where the closure of the zero locus of $g_k$, and $X\backslash \bb{G}_m^{2k}$ form a strict normal crossing divisor. 

The product $\bb{P}^1_t\times X$ is a compactification of $(\bb{G}_m^{2k+1},t\cdot g_k)\simeq (\bb{G}_m^{2k+1},\widetilde {f}_k)$. Let $S_1,\ldots,S_N$ be the irreducible components of $X\backslash \bb{G}_m^{2k}$. We first do the blow-up along the intersection of $0\times X$ and $S_1$. Then, we do blow-ups along the intersection of the proper transform of $0\times X$ and the proper transform of $S_i$ for $2\leq i\leq N$ successively. The resulting variety $\bar X$ is a compactification of $\bb{G}_m^{2k+1}$. 

We can verify that $(\bar X,\tilde f_k)$ is a non-degenerate compactification of $(\bb{G}_m^{2k+1},t\cdot g_k)$ if $3\nmid k$. Otherwise, we need to do two more blow-ups. The first one is at each closed point of 
	\begin{equation*}
		\begin{split}
			\infty\times Z \Bigl(y_{i,1}^{3}-1,y_{i,2}^{3}-1,\sum_i y_{i,1}+y_{i,2}+\frac{1}{y_{i,1}y_{i,2}} \Bigr),
		\end{split}
	\end{equation*}
and the second one is on the intersection of the exceptional divisors from the previous step and the proper transform of $\infty\times X$. We denote by $E_1, E_2$ the exceptional divisors of the two steps. We denote by $\widetilde X$ the resulting variety. By a direct computation, we can verify that $\ord_{E_1}\tilde f_k=1$ and $\ord_{E_2}\tilde f_k=0$, and $(\widetilde X,\widetilde f_k)$ is a non-degenerate compactification of $(\bb{G}_m^{2k+1},\widetilde f_k)$.

\paragraph{A lemma}
Similar to \eqref{eq:explicite-cohomology-class}, we have morphisms of exponential mixed structures 
\begin{equation*}
	\mathrm{H}^1_{\mathrm{mid}}(\bb{G}_m,\Sym^k\Kl_{n+1})^{\rr{H}}\hookrightarrow \rr{H}^{nk+1}(\bb{G}_m^{nk+1},f_k)^{\rr{H}}\hookrightarrow \rr{H}^{nk+1}(\bb{G}_m^{nk+1},\tilde f_k)^{\rr{H}}.
\end{equation*}
By abuse of notation, for an element $w$ in the set $\widetilde W_{d,k}$ from Theorem~\ref{thm:cohomology_classes-kl3}, we denote by $w$ its image in $\mathrm{H}_{\mathrm{dR}}^{2k+1}(\bb{G}_m^{2k+1},\tilde f_k)$ via the above inclusion.

\begin{lem}\label{lem:Hodge-degenerate}
	Assume $3\mid k$. Then $\widetilde W_{d,k} \subset F^p\mathrm{H}^{2k+1}_{\mathrm{dR}}(\bb{G}_m^{2k+1},\tilde f_k)$ if $p\leq 2k+1-d$ and $k-d\geq 0$.
\end{lem}
\begin{proof}
Let $\widetilde X$ be the compactification as above and $D=\widetilde X\backslash \bb{G}_m^{2k+1}$ is the boundary divisor. Because the indeterminacy locus of the rational map $\tilde f_k\colon \widetilde X\dashrightarrow \bb{P}^1$ has codimension of at least $2$ in $\widetilde X$, we can take the pole divisor $P$ as the closure of $\widetilde{f_k}$. The exceptional divisors $E_1$ and $E_2$ are not contained in the support of the pole divisor because $\ord_{E_1}\tilde f_k=1$ and $\ord_{E_2}\tilde f_k=0$.

As in \eqref{eq:explicite-form}, the image of an element $t^a v_0^{I_0} v_1^{I_1}v_2^{I_2}$ of degree $d$ in 	
	\begin{equation*}
		\begin{split}
			\Hdr{2k+1}(\bb{G}_m^{2k+1},\tilde f_k)^{S_k}=\Hdr{2k+1}(\bb{G}_m^{2k+1},t\cdot g_k)^{S_k}
		\end{split}
	\end{equation*}
is of the form
	\begin{equation*}
		\begin{split}
			t^{d} \frac{1}{k!}\sum_{\sigma\in S_k}\prod_{i=I_1+1}^{k}y_{\sigma(i),1} \prod_{i=I_2+1}^k y_{\sigma(i),2}\frac{\mathrm{d}t}{t}\frac{\mathrm{d}y_{1,1}}{y_{1,1}}\frac{\mathrm{d}y_{1,2}}{y_{1,2}}\cdots\frac{\mathrm{d}y_{k,1}}{y_{k,1}}\frac{\mathrm{d}y_{k,2}}{y_{k,2}}
		\end{split}
	\end{equation*}
up to a nonzero constant (compared to \eqref{eq:explicite-form} we did a change of variable $x_{i,j}=ty_{i,j}$). By a direct computation, we find that 
\begin{equation*}
	\begin{split}
		\ord_{E_1}(t^a v_0^{I_0} v_1^{I_1}v_2^{I_2})=2k-d\quad  \text{and} \quad \ord_{E_2}(t^a v_0^{I_0} v_1^{I_1}v_2^{I_2})=2k-2d.
	\end{split}
\end{equation*}
Since each $w\in \widetilde W_{d,k}$ is a linear combination of $t^a v_0^{I_0} v_1^{I_1}v_2^{I_2}$ of degree $d$, we find 
	\begin{equation*}
		\begin{split}
			\widetilde W_{d,k}\subset \Gamma {(\widetilde X,\Omega^{2k+1}\q{\log D}\q{dP-\q{2k-d}E_1-(2k-2d)E_2})}.
		\end{split}
	\end{equation*}
We further have
	\begin{equation*}
		\begin{split}
			\widetilde W_{d,k}\subset  \Gamma{(\widetilde X,\Omega^{2k+1}\q{\log D}\q{\lfloor 2k+1-\eta\rfloor P})},
		\end{split}
	\end{equation*}
if $k-d\geq 0$ and $d \leq 2k+1-\eta$. In this case, there is a natural map
	\begin{equation*}
		\begin{split}
			\Gamma{(\widetilde X,\Omega^{2k+1}\q{\log D}\q{\lfloor 2k+1-\eta\rfloor P})}\to F^\eta \mathrm{H}^{2k+1}(\bb{G}^{2k+1}_m,\tilde f_k),
		\end{split}
	\end{equation*}
see \cite[\S4\,b)]{yu2012irregular} or \cite[\S\,4.3.3]{fresan2018hodge}, which completes the proof of the lemma.
\end{proof}

\begin{prop}\label{prop:half-filtration}
When $3\mid k$ and $p\geq k+1$, the vector space $F^p \Hdr{1}(\bb{G}_m,\Sym^k\widetilde \Kl_3)$ is 
	\begin{equation*}
		\begin{split}
			\mathrm{span}\q{\widetilde W_{d,k}\mid  p\leq 2k+1-d}.
		\end{split}
	\end{equation*}
\end{prop}

\begin{proof}

We construct a filtration $G^\bullet $ on the de Rham cohomology $\Hdr{1}(\bb{G}_m,\Sym^k\widetilde \Kl_3)$ by 
	\begin{equation*}
		\begin{split}
			G^p\Hdr{1}(\bb{G}_m,\Sym^k\widetilde \Kl_3):=\rr{Span}(\widetilde W_{d,k} \mid p\leq 2k+1-d).
		\end{split}
	\end{equation*}
Using the map in \eqref{eq:explicite-cohomology-class} and Lemma~\ref{lem:Hodge-degenerate}, we deduce that 
	\begin{equation*}
		\begin{split}
			G^p\Hdr{1}(\bb{G}_m,\Sym^k\widetilde \Kl_3)\subset F^p\Hdr{1}(\bb{G}_m,\Sym^k\widetilde \Kl_3)
		\end{split}
	\end{equation*}
if $p\geq k+1$.

Let $d_p=\dim \mathrm{gr}_F^p$ and $\delta_p=\dim \mathrm{gr}_G^p$. Then 
	\begin{equation}\label{eq:dp-geq-deltap-3midk}
		\sum_{q\geq p}\delta_p \leq \sum_{q\geq p}d_p
	\end{equation}
and $d_p=\sum_q h^{p,q}$ if $p\geq k+1$, where $h^{p,q}$ is defined by Proposition~\ref{prop:hodgenumberstildekl3-3midk}. By Theorem~\ref{thm:cohomology_classes-kl3}, we know that the number $\delta_p$ is $\tilde n_{2k+1-p,k}=\lfloor\frac{2k+1-p}{2}\rfloor +1$ for any $p\geq k+1$, which coincides with $d_p$ by Proposition~\ref{prop:hodgenumberstildekl3-3midk}. We deduce that the filtration $G^\bullet$ coincides with the Hodge filtration $F^\bullet$ when $p\geq k+1$.
\end{proof}

\paragraph{The Hodge numbers}
\begin{thm}[Theorem~\ref{thm:hodge number-Kl-2}]\label{thm:hodgenumberskl3-3midk}
Assume that $3\mid k$. For $p\leq k$, the Hodge numbers of the mixed Hodge structure $\mathrm{H}^1_{\mathrm{mid}}(\bb{G}_m,\Sym^k\Kl_3)^{\rr{H}}$ are given by
		\begin{equation*}
			\begin{split}
				h^{p,2k+1-p}=h^{2k+1-p,p}=-\delta_{k,p}+
		\begin{cases}
			\lfloor\frac{p}{6}\rfloor & p\not\equiv 3,5\ (\mathrm{mod} \ 6);\\[4pt]
			\lfloor\frac{p}{6}\rfloor+1 & p\equiv 3,5\ (\mathrm{mod} \ 6).
			\end{cases}
			\end{split}
		\end{equation*}	
\end{thm}
\begin{proof}
The set $\bigcup_d  W_{d,k}$ in Theorem~\ref{thm:cohomology_classes} is a basis for $\mathrm{H}^1_{\mathrm{dR }}(\bb{G}_m,\Sym^k\Kl_3)$. We know that $w\in W_{d,k} $ is in $F^p\mathrm{gr}^W_{2k+1}\mathrm{H}^1_{\mathrm{dR}}(\bb{G}_m,\Sym^k \Kl_3)$ if 
	\begin{equation*}
		\begin{split}
			p\leq2k+1-d \quad \text{and} \quad k-d\geq 0
		\end{split}
	\end{equation*}
by Lemma~\ref{lem:Hodge-degenerate}. Since we chose $W_{d,k}$ as a subset of $\tilde{W}_{d,k}$ in Theorem~\ref{thm:cohomology_classes-kl3}, we deduce that $w$ is nonzero in the graded quotient 
	$\mathrm{gr}_F^{2k+1-d}\mathrm{H}^1_{\mathrm{dR }}(\bb{G}_m,\Sym^k \Kl_3)$
if $k-d\geq 0$ by Proposition~\ref{prop:half-filtration}. Hence, 
	$$\dim \mathrm{gr}_F^{2k+1-d}\mathrm{H}^1_{\mathrm{dR }}(\bb{G}_m,\Sym^k \Kl_3)\geq\#W_{d,k}= n_{d,k}.$$
 
 By Proposition~\ref{prop:hodgenumberstildekl3-3midk} and Remark~\ref{rek:hodge-number-nonpure-part}, we notice that 
	\begin{equation*}
		\begin{split}
		 	&\dim \mathrm{gr}_F^{2k+1-d}\mathrm{H}^1_{\mathrm{dR}}(\bb{G}_m,\Sym^k \Kl_3)/\mathrm{H}^1_{\mathrm{dR, mid}}(\bb{G}_m,\Sym^k  \Kl_3)\\
			=&\dim \mathrm{gr}_F^{2k+1-d}\mathrm{H}^1_{\mathrm{dR}}(\bb{G}_m,\Sym^k \tilde{\Kl}_3)/\mathrm{H}^1_{\mathrm{dR, mid}}(\bb{G}_m,\Sym^k  \tilde{\Kl}_3)\\
			=&\delta_{d,k}+(1+2\lfloor\tfrac{d}{2}\rfloor-d).
		\end{split}
	\end{equation*} 
So, the Hodge numbers 
	$h^{d,2k+1-d}=h^{2k+1-d,d}=\dim \mathrm{gr}_F^{2k+1-d}\mathrm{H}^1_{\mathrm{dR,mid}}(\bb{G}_m,\Sym^k\Kl_3)$
is at least
	\begin{equation*}
		\begin{split}
			n_{d,k}-\delta_{d,k}-(1+2\lfloor\tfrac{d}{2}\rfloor-d)=n^{\rr{mid}}_{2k+1-d,k}
		\end{split}
	\end{equation*}
if $k-d\geq 0$ by the formulas of $n_{d,k}$ and $n_{d,k}^{\rr{mid}}$ in Theorems~\ref{thm:cohomology_classes-kl3} and \ref{thm:mid-cohomology_classes-kl3}. Then, we deduce that
\begin{equation*}
	\begin{split}
		\dim \mathrm{H}^1_{\mathrm{dR,mid}}(\bb{G}_m,\Sym^k\Kl_3)
		&=\sum_{d=0}^{2k+1}h^{d,2k+1-d}\geq 2\sum_{d=0}^k n^{\rr{mid}}_{d,k}\\
		&\stackrel{(*)}{=}\dim \mathrm{H}^1_{\mathrm{dR,mid}}(\bb{G}_m,\Sym^k\Kl_3),
	\end{split}
\end{equation*}
where (*) can be checked directly\footnote{We can write $k=6\ell+r$ for $r=0$ or $3$ and check the identity in terms of  $\ell$ and $r$.}. Therefore, we have $h^{d,2k+1-d}=h^{2k+1-d,d}=n^{\rr{mid}}_{d,k}$ for $d\leq k$, which are exactly numbers stated in the proposition. 
\end{proof}

\subsection{The case of \texorpdfstring{$\Sym^k\rr{Ai}_{n}$}{SymkAin+1} when \texorpdfstring{$\gcd(k,n)=1$}{gcd(k, n)=1}}
We give the proof of Theorem~\ref{thm:hodge number-Ai} here. The Airy connection $\rr{Ai}_{n}$, as a $\bb{C}[z]$-module, equals to the cokernel of the complex
	\begin{equation*}
		\begin{split}
			\bb{C}[x,z]\mathrm{d}z
			\xrightarrow{\mathrm{d}+\partial_{x}f\mathrm{d}x\wedge}\bb{C}[x,z]\mathrm{d}z\wedge\mathrm{d}x,
		\end{split}
	\end{equation*}
where $f$ is the Laurent polynomial in \eqref{eq:f-Airy}. In this way, the classes $v_i$ are sent to 
	$x^{i-1} \mathrm{d}z\mathrm{d}x$
for $0\leq i\leq n-1$, and we have 
	$(z\partial_z)^{n}\mathrm{d}z\wedge \mathrm{d}x=z\cdot \mathrm{d}z\wedge \mathrm{d}x.$

Let $f_k$ be the Laurent polynomial in \eqref{eq:fk-Airy}. Then we have morphisms of exponential mixed Hodge structures
	\begin{equation}\label{eq:explicite-cohomology-class-Ai}
		\begin{split}
		\mathrm{H}^1_{\mathrm{mid}}(\bb{A}^1,\Sym^k\rr{Ai}_{n})^{\rr{H}} 
		&=\mathrm{H}^1(\bb{A}^1,\Sym^k\rr{Ai}_{n})^{\rr{H}}\\
		&= \Bigl(\rr{H}^{k+1}(\bb{A}^{k+1},f_k)^{\rr{H}}\Bigr)^{S_k,\chi}
		\hookrightarrow \rr{H}^{k+1}(\bb{A}^{k+1},f_k)^{\rr{H}}.
		\end{split}
	\end{equation}

Via the above maps, an element $z^jv^{\underline I}$ is sent to
	\begin{equation*}
		\frac{1}{k!}\sum_{\sigma\in S_k}
		\biggl(z^j \prod_{i=I_1+1}^{k}x_{\sigma(i)} 
		\prod_{i=I_2+1}^k x_{\sigma(i)}^2
		\cdots \prod_{i=I_{n-1}+1}^k x_{\sigma(i)}^{n-1}\biggr)
		\mathrm{d}z\,\mathrm{d}x_{1}\cdots\mathrm{d}x_{k}
	\end{equation*}
in $\Hdr{k+1}(\bb{A}^{k+1},f_k)^{S_k,\chi}$, similar to \eqref{eq:explicite-form}. So each element $w\in W_{d,k}$ from Theorem~\ref{thm:cohomology_classes-Ai} is sent to 
	\begin{equation}\label{eq:g(w)-Airy}
		\begin{split}
			g(w)\mathrm{d}z\,\mathrm{d}x_{1}\cdots\mathrm{d}x_{k}
		\end{split}
	\end{equation}
for a polynomial $g(w)$ in $z, x_{i}$ such that each term of $g(w)$ has degree $d$ with respect to the degree \eqref{eq:grading}. By abuse of notation, we still denote by $w$ its image under \eqref{eq:explicite-cohomology-class-Ai}.

\begin{lem}\label{prop:hodge-number-non-degenerate-newton-Ai}
	Assume that $\gcd(k,n)=1$. Then $ W_{d,k}$ lies in  $F^p\Hdr{k+1}(\bb{A}^{k+1},f_k)$ if $p\leq \frac{nk+1-d}{n+1}$.
\end{lem}
\begin{proof}
Thanks to the morphism of exponential mixed Hodge structures
	\begin{equation*}
		\begin{split}
			\rr{H}^{k+1}(\bb{A}^{k+1},f_k)^{\rr{H}}\hookrightarrow  \rr{H}^{k+1}(\bb{G}_m^{k+1},f_k)^{\rr{H}},
		\end{split}
	\end{equation*}
 it suffices to show that $W_{d,k} \subset F^p\Hdr{k+1}(\bb{G}_m^{k+1},f_k)$ if $p\leq \frac{nk+1-d}{n+1}$.

The Newton polytope $\Delta(f_k)$ defined by $f_k$ has only one facet that does not contain the origin, which is lying on the hyperplane $n\cdot \alpha+\sum_{i}\beta_{i}=1$. We can check that $f_k$ is non-degenerate with respect to $\Delta(f_k)$ when $\gcd(k,n)=1$. So the irregular Hodge filtration on $\mathrm{H}^{k+1}_{\mathrm{dR}}(\bb{G}_m^{k+1},f_k)$ can be computed via the Newton filtration on monomials in $\bb{R}_{\geq 0}\Delta(f_k)$. 

The cone $\bb{R}_{\geq 0}\Delta(f_k)$ is given by inequalities
	\begin{equation*}
		\begin{split}
			-\alpha+\sum_{i=1}^k \beta_i\geq 0,\quad  \alpha\geq 0, \quad \text{and}\quad \beta_i\geq 0.
		\end{split}
	\end{equation*}
We take the fan $F$ generated by rays
	\begin{equation*}
		\begin{split}
			\bb{R}_{\geq 0}\cdot \pm\Bigl(-\alpha+\sum_{i=1}^k \beta_i\Bigr), \quad \bb{R}_{\geq 0}\cdot \pm \alpha, \quad  \ \bb{R}_{\geq 0}\cdot \pm \beta_i \quad \text{and} \quad \bb{R}_{\geq 0}\cdot \pm \Bigl(n\alpha+\sum_i \beta_i\Bigr).
		\end{split}
	\end{equation*}
We can verify that the simplicial polytopal fan $F$ is regular. So the corresponding toric variety $X_{\tor}$ is smooth projective. In particular, each irreducible component of the pole divisor $P$ of $f_k$ has multiplicity $1$.

As in the proof of Lemma~\ref{prop:hodge-number-non-degenerate-newton}, the class of $z^j\prod_i x_i^{a_i}
\frac{\mathrm{d}z}{z}\frac{\mathrm{d}x_{1}}{x_{1}}\cdots\frac{\mathrm{d}x_{k}}{x_{k}}$ is in $F^{p}\mathrm{H}^{k+1}_{\dR}(\bb{G}^{k+1}_m, f_k)$ if it belongs to $\Omega^{k+1}(\log S)(\lfloor(k+1-p)P\rfloor)$, which is equivalent to
		\begin{equation*}
		\mathrm{ord}_{D}\Bigl(z^a\prod_i x_i^{b_i}\Bigr)\geq -(k+1-p)
		\end{equation*}
for all irreducible components $D$ of $P$. By \cite[p.61]{fulton1993introduction}, this condition is again equivalent to 
		\begin{equation*}
			\begin{split}
				-\xi\Bigl(z^a\prod_i x_i^{b_i}\Bigr)\geq -(k+1-p),
			\end{split}
		\end{equation*}
where $\xi(z^a\prod_i x_i^{b_i})=(n\cdot a+\sum_{i}b_i)/(n+1)$.

	For $w\in W_{d,k}$, the value of $\xi$ at each terms of $g(w)$ in \eqref{eq:g(w)-Airy} is $(d+n+k)/(n+1)$. Therefore, $w\in F^p\Hdr{k+1}(\bb{G}_m^{k+1},f_k)$ if $p\leq \frac{nk+1-d}{n+1}$.
\end{proof}

\begin{thm}[Theorem~\ref{thm:hodge number-Ai}]
Assume that $\gcd(k,n)=1$. The possible jumps of the Hodge filtration of $\mathrm{H}^1_{\mathrm{mid}}(\bb{A}^1,\Sym^k\rr{Ai}_{n})^{\rr{H}}$ are $\frac{p+n+k}{n+1}$ for $0\leq p\leq nk-n-k+1$. The Hodge numbers $h^{\frac{p+n+k}{n+1},\frac{nk+1-p}{n+1}}$ are $\#W_{p,k},$ where $W_{p,k}$ are sets from Theorem~\ref{thm:cohomology_classes-Ai}.
\end{thm}

\begin{proof}
	The proof is similar to that of Theorem~\ref{thm:hodgenumberskl3-3nmidk}. We construct an auxiliary filtration $G^\bullet $ on $\mathrm{H}^1_{\mathrm{dR}}(\bb{A}^1,\Sym^k\rr{Ai}_{n})$, by letting the subspace $G^p\mathrm{H}^1_{\mathrm{dR}}(\bb{A}^1,\Sym^k\rr{Ai}_{n})$ be generated by elements $ w\in W_{d,k}$ such that $p\leq k+1-\frac{d+n+k}{n+1}=\frac{nk+1-d}{n+1}$. We deduce from Lemma~\ref{prop:hodge-number-non-degenerate-newton-Ai} that 	
		\begin{equation*}
			\begin{split}
				G^p\mathrm{H}^1_{\mathrm{dR}}(\bb{A}^1,\Sym^k\rr{Ai}_{n})\subset F^p\mathrm{H}^1_{\mathrm{dR}}(\bb{A}^1,\Sym^k\rr{Ai}_{}).
			\end{split}
		\end{equation*}

Let $d_p=\dim \mathrm{gr}_F^p\mathrm{H}^1_{\mathrm{dR}}(\bb{A}^1,\Sym^k\rr{Ai}_{n})$ and $\delta_p=\dim \mathrm{gr}_G^p\mathrm{H}^1_{\mathrm{dR}}(\bb{A}^1,\Sym^k\rr{Ai}_{n})$ be the dimensions of the graded quotients with respect to $F^\bullet$ and $G^\bullet$. Then for $0\leq q \leq k+1$ we have
	\begin{equation}\label{eq:dp-geq-deltap-Ai}
		\sum_{q\geq p}\delta_p \leq \sum_{q\geq p}d_p,
	\end{equation}
where the equality holds if $q=0$ and $q=k+1$. By the Hodge symmetry, we have $d_p=d_{k+1-p}$. Since $\delta_p$ are the numbers $n_{nk-n-k+1-p}$ in Theorem~\ref{thm:cohomology_classes-Ai}, we have $\delta_p=\delta_{k+1-p}$.

Combining \eqref{eq:dp-geq-deltap-Ai} and the symmetric properties of $d_p$ and $\delta_p$, we find
	\begin{equation*}
		\begin{split}
			\sum_{q\geq p}\delta_p\leq \sum_{q\geq p}d_p=\sum_{k+1-q\leq p}d_p\leq \sum_{k+1-q\leq p}\delta_p=\sum_{q\geq p}\delta_p,
		\end{split}
	\end{equation*}
which implies that $\delta_p=d_p$ for each $p$. 
\end{proof}

\begin{exe}
	Assume that $3\nmid k$. The Hodge numbers of $\mathrm{H}^1(\bb{A}^1,\Sym^k\rr{Ai}_{3})^{\rr{H}}$ are given by
			\begin{equation*}
				\begin{split}
					h^{\frac{p+3+k}{4},\frac{3k+1-p}{4}}=\begin{cases}
						\lfloor\frac{p}{6}\rfloor & p\not\equiv 3,5 (\rr{mod}\,6);\\[4pt]
						\lfloor\frac{p}{6}\rfloor+1 & \text{else},
						\end{cases}
				\end{split}
			\end{equation*}
	if $p\leq k$, and $h^{\frac{p+3+k}{4},\frac{3k+1-p}{4}}=h^{\frac{3k+1-p}{4},\frac{p+3+k}{4}}$ if $k>p$.
\end{exe}

\subsection{The case of \texorpdfstring{ $\mathrm{Kl}_{\rr{SL}_3}^{V_{2,1}}$}{Kl3V21}}

Let $f\colon \mathbb{G}_m^{9}\to \bb{A}^1$ be the Laurent polynomial
	\begin{equation*}
		\begin{split}
			(x_{i},y_i,z)\mapsto \sum_{i=1}^4 \Bigl((x_i+y_i+\frac{z}{x_iy_i}\Bigr).
		\end{split}
	\end{equation*}
The symmetric group $S_4$ acts on the sub-indices of the coordinates. As in \cite[Cor.\,2.15]{fresan2018hodge}, we have
\begin{equation*}
	\begin{split}
		\mathrm{H}^1_{\mathrm{dR,mid}}\Bigl(\bb{G}_m,\mathrm{Kl}_{\rr{SL}_3}^{V_{2,1}}\Bigr)\simeq \mathrm{H}^{9}_{\mathrm{dR}}(\bb{G}_m^9,f)^{P\times Q,\chi}.
	\end{split}
	\end{equation*} 
We define the associated exponential mixed Hodge structure as 
\begin{equation*}
	\begin{split}
		\mathrm{H}^1_{\mathrm{mid}}\Bigl(\bb{G}_m,\mathrm{Kl}_{\rr{SL}_3}^{V_{2,1}}\Bigr)^{\rr{H}}:= (\mathrm{H}^{9}_{\rr{mid}}(\bb{G}_m^9,f)^{\rr{H}})^{P\times Q,\chi}.
	\end{split}
	\end{equation*} 
By considering an analogue of the inclusion \eqref{eq:explicite-cohomology-class}, the elements in $W_4^{\rr{mid}}$ and $W_5^{\rr{mid}}$ are mapped to
	\begin{equation*}
		\begin{split}
			g\frac{\mathrm{d}z}{z}\prod_{i,j}\frac{\mathrm{d}x_{i,j}}{x_{i,j}}\ \text{and}\ h \frac{\mathrm{d}z}{z}\prod_{i,j}\frac{\mathrm{d}x_{i,j}}{x_{i,j}}
		\end{split}
	\end{equation*}
in $\mathrm{H}^{9}_{\mathrm{dR}}(\bb{G}_m^9,f)$, for some polynomials in $z,x_i,y_i$ of degree $4$ and $5$ respectively. 

Because $f$ is non-degenerate with respect to $\Delta(f)$, we can use a similar argument as in the proves of Lemma~\ref{prop:hodge-number-non-degenerate-newton} and Theorem~\ref{thm:hodgenumberskl3-3nmidk} to get the Hodge numbers as follows:
\begin{prop}\label{prop:hodge-numbers-Kl3V21}
	We have 
		\begin{equation*}
			\begin{split}
				W_5^{\rr{mid}}\subset \mathrm{gr}_F^4\mathrm{H}^1_{\mathrm{dR,mid}}\Bigl(\bb{G}_m,\mathrm{Kl}_{\rr{SL}_3}^{V_{2,1}}\Bigr) \quad \text{and}\quad W_4^{\rr{mid}}\subset \mathrm{gr}_F^{5}\mathrm{H}^1_{\mathrm{dR,mid}}\Bigl(\bb{G}_m,\mathrm{Kl}_{\rr{SL}_3}^{V_{2,1}}\Bigr).
			\end{split}
		\end{equation*}
	In particular, the nonzero Hodge numbers of the pure Hodge structure $\mathrm{H}^1_{\mathrm{mid}}\Bigl(\bb{G}_m,\mathrm{Kl}_{\rr{SL}_3}^{V_{2,1}}\Bigr)^{\rr{H}}$ are $h^{4,5}=h^{5,4}=1$.
\end{prop}

\bibliography{Hodge}

\begin{thebibliography}{10}

\bibitem{Adolphson1997twisted}
A.~Adolphson and S.~Sperber.
\newblock On twisted de {R}ham cohomology.
\newblock {\em Nagoya Math. J.}, 146:55--81, 1997.

\bibitem{Anderson1986motive}
G.~W. Anderson.
\newblock Cyclotomy and an extension of the {T}aniyama group.
\newblock {\em Compos. Math.}, 57(2):153--217, 1986.

\bibitem{Broadhurst2016a}
D.~Broadhurst.
\newblock Feynman integrals, {L}-series and {K}loosterman moments.
\newblock {\em Commun. Number Theory Phys.}, 10(3):527--569, 2016.

\bibitem{Broadhurst2017}
D.~Broadhurst.
\newblock Critical {L}-values for products of up to {20} {B}essel functions.
\newblock {\em Lecture slides}, 2017.

\bibitem{SGA41/2}
P.~Deligne.
\newblock {\em Cohomologie \'{e}tale}, volume 569 of {\em Lecture Notes in Mathematics}.
\newblock Springer-Verlag, Berlin, 1977.
\newblock S\'{e}minaire de g\'{e}om\'{e}trie alg\'{e}brique du Bois-Marie SGA $4\frac{1}{2}$.

\bibitem{Deligne2007}
P.~Deligne.
\newblock Th{\'e}orie de {H}odge irr{\'e}guli{\`e}re (mars 1984 {\&} ao{\^u}t 2006).
\newblock In {\em Singularit{\'e}s irr{\'e}guli{\`e}res, {C}orrespondonce et documents}, volume~5, pages 109--114,115--128. Soci\'{e}t\'{e} Math\'{e}matique de France, Paris, 2007.

\bibitem{esnault-sabbah-yu2017e1}
H.~Esnault, C.~Sabbah, and J.-D. Yu.
\newblock {$E_1$}-degeneration of the irregular {H}odge filtration.
\newblock {\em J. Reine Angew. Math.}, 729:171--227, 2017.
\newblock With an appendix by Morihiko Saito.

\bibitem{Frenkel2009}
E.~Frenkel and B.~Gross.
\newblock A rigid irregular connection on the projective line.
\newblock {\em Ann. of Math. (2)}, 170(3):1469--1512, 2009.

\bibitem{fresanexponential}
J.~Fres{\'a}n and P.~Jossen.
\newblock Exponential motives \url{http://javier.fresan.perso.math.cnrs.fr/expmot.pdf}.
\newblock In preparation.

\bibitem{fresan2018hodge}
J.~Fres\'{a}n, C.~Sabbah, and J.-D. Yu.
\newblock Hodge theory of {K}loosterman connections.
\newblock {\em Duke Math. J.}, 171(8):1649--1747, 2022.

\bibitem{fresan2020quadratic}
J.~Fres{\'a}n, C.~Sabbah, and J.-D. Yu.
\newblock Quadratic relations between periods of connections.
\newblock {\em Tohoku Math. J. (2)}, 75(2), 2023.

\bibitem{fu2005functions}
L.~Fu and D.~Wan.
\newblock {$L$}-functions for symmetric products of {K}loosterman sums.
\newblock {\em J. Reine Angew. Math.}, 589:79--103, 2005.

\bibitem{fu2006trivial}
L.~Fu and D.~Wan.
\newblock Trivial factors for {$L$}-functions of symmetric products of {K}loosterman sheaves.
\newblock {\em Finite Fields Appl.}, 14(2):549--570, 2008.

\bibitem{fulton1993introduction}
W.~Fulton.
\newblock {\em Introduction to toric varieties}, volume 131 of {\em Annals of Mathematics Studies}.
\newblock Princeton University Press, Princeton, NJ, 1993.
\newblock The William H. Roever Lectures in Geometry.

\bibitem{W.Fulton2013}
W.~Fulton and J.~Harris.
\newblock {\em Representation theory}, volume 129 of {\em Graduate Texts in Mathematics}.
\newblock Springer-Verlag, New York, 1991.
\newblock A first course, Readings in Mathematics.

\bibitem{HNY-Kloosterman}
J.~Heinloth, B.-C. Ng\^{o}, and Z.~Yun.
\newblock Kloosterman sheaves for reductive groups.
\newblock {\em Ann. of Math. (2)}, 177(1):241--310, 2013.

\bibitem{HTT-dmodule}
R.~Hotta, K.~Takeuchi, and T.~Tanisaki.
\newblock {\em {$D$}-modules, perverse sheaves, and representation theory}, volume 236 of {\em Progress in Mathematics}.
\newblock Birkh\"{a}user Boston, Inc., Boston, MA, 2008.
\newblock Translated from the 1995 Japanese edition by Takeuchi.

\bibitem{katz1987galois}
N.~M. Katz.
\newblock On the calculation of some differential {G}alois groups.
\newblock {\em Invent. Math.}, 87(1):13--61, 1987.

\bibitem{katz1990exponential}
N.~M. Katz.
\newblock {\em Exponential sums and differential equations}, volume 124 of {\em Annals of Mathematics Studies}.
\newblock Princeton University Press, Princeton, NJ, 1990.

\bibitem{Kontsevich2011}
M.~Kontsevich and Y.~Soibelman.
\newblock Cohomological {H}all algebra, exponential {H}odge structures and motivic {D}onaldson-{T}homas invariants.
\newblock {\em Commun. Number Theory Phys.}, 5(2):231--352, 2011.

\bibitem{2013Monodromy-sabbah-appendix}
Y.~Matsui and K.~Takeuchi.
\newblock Monodromy at infinity of polynomial maps and {N}ewton polyhedra (with an appendix by {C}. {S}abbah).
\newblock {\em Int. Math. Res. Not. IMRN}, (8):1691--1746, 2013.

\bibitem{mochizuki2015GKZ}
T.~Mochizuki.
\newblock {Twistor property of GKZ-hypergeometric systems}.
\newblock {\em arXiv:1501.04146 [math.AG]}, 2015.

\bibitem{patrikis_taylor_2015}
S.~Patrikis and R.~Taylor.
\newblock Automorphy and irreducibility of some {$l$}-adic representations.
\newblock {\em Compos. Math.}, 151(2):207--229, 2015.

\bibitem{Qin23function}
Y.~Qin.
\newblock {$L$}-functions of {K}loosterman sheaves.
\newblock arXiv:2305.04882 [math.AG], 2023.

\bibitem{sabbah2008explicit}
C.~Sabbah.
\newblock An explicit stationary phase formula for the local formal {F}ourier-{L}aplace transform.
\newblock In {\em Singularities {I}}, volume 474 of {\em Contemp. Math.}, pages 309--330. Amer. Math. Soc., Providence, RI, 2008.

\bibitem{Sabbah-2010-laplace}
C.~Sabbah.
\newblock Fourier-{L}aplace transform of a variation of polarized complex {H}odge structure, {II}.
\newblock In {\em New developments in algebraic geometry, integrable systems and mirror symmetry ({RIMS}, {K}yoto, 2008)}, volume~59 of {\em Adv. Stud. Pure Math.}, pages 289--347. Math. Soc. Japan, Tokyo, 2010.

\bibitem{sabbah18irregularhodge}
C.~Sabbah.
\newblock Irregular {H}odge theory.
\newblock {\em M\'{e}m. Soc. Math. Fr. (N.S.)}, (156):vi+126, 2018.
\newblock With the collaboration of Jeng-Daw Yu.

\bibitem{Sabbah-Yu-2015}
C.~Sabbah and J.-D. Yu.
\newblock On the irregular {H}odge filtration of exponentially twisted mixed {H}odge modules.
\newblock {\em Forum Math. Sigma}, 3:Paper No. e9, 71, 2015.

\bibitem{sabbah2021airy}
C.~Sabbah and J.-D. Yu.
\newblock Hodge properties of airy moments.
\newblock {\em Tunis. J. Math}, 5(2):215--271, 2023.

\bibitem{Saito1988}
M.~Saito.
\newblock Modules de {H}odge polarisables.
\newblock {\em Publ. Res. Inst. Math. Sci.}, 24(6):849--995 (1989), 1988.

\bibitem{Saito1990}
M.~Saito.
\newblock Mixed {H}odge modules.
\newblock {\em Publ. Res. Inst. Math. Sci.}, 26(2):221--333, 1990.

\bibitem{weibel-homological}
C.~A. Weibel.
\newblock {\em An introduction to homological algebra}, volume~38 of {\em Cambridge Studies in Advanced Mathematics}.
\newblock Cambridge University Press, Cambridge, 1994.

\bibitem{yu2012irregular}
J.-D. Yu.
\newblock Irregular {H}odge filtration on twisted de {R}ham cohomology.
\newblock {\em Manuscripta Math.}, 144(1-2):99--133, 2014.

\bibitem{yun2015galois}
Z.~Yun.
\newblock Galois representations attached to moments of {K}loosterman sums and conjectures of {E}vans.
\newblock {\em Compos. Math.}, 151(1):68--120, 2015.
\newblock Appendix B by Christelle Vincent.

\end{thebibliography}
\bibliographystyle{abbrv}

\end{document}